\documentclass[11pt]{amsart}
\usepackage{amsmath}
\usepackage{amsthm}
\usepackage{bm}

\usepackage{color}
\definecolor{c1}{rgb}{0.0,0.3,0.4}
\definecolor{c2}{rgb}{0.0,0.0,0.5}
\definecolor{c3}{rgb}{0.1,0.3,0.5}
\definecolor{dgrn}{rgb}{0.0,0.3,0.0}
\definecolor{dpur}{rgb}{0.5,0.0,0.5}
\usepackage[colorlinks=true, linkcolor=c2]{hyperref}

\usepackage[margin=1.0in]{geometry}

\numberwithin{equation}{section}

\newtheorem{prop}{Proposition}
\newtheorem{lemma}[prop]{Lemma}

\newtheorem{thm}[prop]{Theorem}
\newtheorem{cor}[prop]{Corollary}

\numberwithin{prop}{section}

\theoremstyle{definition}
\newtheorem{defn}[prop]{Definition}

\newtheorem{rmk}[prop]{Remark}

\newcommand{\del}{\partial}

\newcommand{\sNG}{{\Gamma}_{\varsigma^*\N}}
\newcommand{\brs}[1]{\left| #1 \right|}
\newcommand{\brk}[1]{\left[ #1 \right]}
\newcommand{\prs}[1]{\left( #1 \right)}
\newcommand{\ip}[1]{\left\langle #1 \right\rangle}
\newcommand{\gG}{\Gamma}

\newcommand{\gd}{\delta}

\newcommand{\gt}{\theta}
\newcommand{\gw}{\omega}
\newcommand{\gz}{\zeta}
\newcommand{\ga}{\alpha}
\newcommand{\gb}{\beta}

\newcommand{\N}{\nabla}
\newcommand{\tN}{\widetilde{\nabla}}
\newcommand{\tgG}{\widetilde{\Gamma}}

\renewcommand{\bar}[1]{\overline{#1}}

\newcommand{\lb}{\left[}
\newcommand{\rb}{\right]}

\newcommand{\hsp}{\hspace{0.5cm}}
\newcommand{\lap}{\Delta}
\newcommand{\la}{\lambda}
\newcommand{\vp}{\varphi}
\newcommand{\vs}{\varsigma}
\newcommand{\vr}{\varrho}

\newcommand{\ten}{\otimes}
\definecolor{grn}{rgb}{0,0.4,0}

\newtheorem{innercustomthm}{Theorem}
\newenvironment{customthm}[1]
  {\renewcommand\theinnercustomthm{#1}\innercustomthm}
  {\endinnercustomthm}

\newcommand{\slap}{\lap_{\varsigma}}
\newcommand{\w}{\wedge}
\newcommand{\gY}{\Upsilon}

\newcommand{\YM}{\mathcal {YM}}

\DeclareMathOperator{\Rc}{Rc}
\DeclareMathOperator{\Rm}{Rm}

\DeclareMathOperator{\Id}{Id}

\DeclareMathOperator{\Vol}{Vol}

\DeclareMathOperator{\supp}{supp}

\DeclareMathOperator{\End}{End}
\DeclareMathOperator{\rank}{rank}

\DeclareMathOperator{\Aut}{Aut}
\DeclareMathOperator{\Grad}{Grad}
\DeclareMathOperator{\Ad}{Ad}
\DeclareMathOperator{\Alt}{Alt}

\begin{document}
\title[Higher order Yang-Mills flow]{Higher order Yang-Mills flow}
\date{\today}
\author{Casey Kelleher}
\email{\href{mailto:clkelleh@uci.edu}{clkelleh@uci.edu}}

\address{Rowland Hall\\
         University of California\\
         Irvine, CA 92617}

\thanks{The author gratefully acknowledges her support from an NSF Graduate Research Fellowship (DGE-1321846).}

\maketitle

\begin{abstract}
 We define a family of functionals generalizing the Yang-Mills functional. We study the corresponding gradient flows and prove long-time existence and convergence results for subcritical dimensions as well as a bubbling criterion for the critical dimensions. Consequently, we have an alternate proof of the convergence of Yang-Mills flow in dimensions $2$ and $3$ given by R\aa de \cite{Rade} and the bubbling criterion in dimension $4$ of Struwe \cite{Struwe} in the case where the initial flow data is smooth.
\end{abstract}

\section{Introduction}
\subsection{Background}\label{ss:background}
Let $(E,h) \to (M,g)$ be a vector bundle over a smooth finite-dimensional manifold. The \emph{Yang-Mills energy} of a connection $\N$ on $E$ is given by
\begin{equation}\label{eq:YMenergy}
\mathcal{YM}(\N) := \tfrac{1}{2}\int_{M} | F_{\N} |^2 dV_g,\tag{\textsf{YME}}
\end{equation}
the $L^2$ norm of the curvature tensor $F_{\N}$. The study of the corresponding gradient flow to \eqref{eq:YMenergy} was initially studied by Atiyah and Bott \cite{AB} for $M$ a closed Riemann surface and proposed studying it as a means of understanding the topology of the space of connections using infinite dimensional Morse theory. Since then, one remarkable application of the flow was Donaldson's characterization of the correspondence between the algebraic and differential geometry on K\"{a}hler manifolds in \cite{Donaldson}. Here he demonstrated that the stability of a bundle is equivalent to it admitting irreducible Hermitian-Einstein connection with respect to the K\"{a}hler metric. 

In the case of $\dim M \in \{ 2 ,3 \}$, R\"ade (\cite{Rade} Theorems 1, 1', 2) demonstrated the long time existence and uniqueness of the flow. For dimension $4$, Struwe (\cite{Struwe}, Theorem 2.3) established existence, convergence and uniqueness of a solution to the gradient flow up to some time $T>0$ and characterized the `bubbling' behavior of energy of the curvature at singularities at such time. Schlatter \cite{Schlatter} later proved Struwe's conjecture of long-time existence and uniqueness of the flow modulo applications of blowups and changes of bundle topology at a finite number of times. Similar results were later demonstrated by Kozono, Maeda, and Naito \cite{KMN} as well as work by Chen and Shen (\cite{CS1}, \cite{CS2}, \cite{CS3}, et. Zhou \cite{CS4}). Work characterizing the behavior of the flow in supercritical dimensions has been established by Tao and Tian \cite{TT}, and more recent developments have been performed by Petrache and Rivi\'{e}re in the case of fixed boundary connections \cite{PR}. In addition, Naito explicitly constructed flows in \cite{Naito} which blow up in finite time for the supercritical dimensions. 

The question of long time existence and convergence of Yang-Mills gradient flow over four dimensional manifolds has yet to be determined and is an area of particular interest. For many flows, the critical dimension offers interesting results but requires nonstandard approaches to study. One of the advantages to our proposed study of the functionals below is that the corresponding family of flows have increasing critical dimension, so one may be able to provide some insight on the Yang-Mills flow and the space of connections in higher dimensions.

We now introduce our main functional of study. For each $k \in \mathbb{N}$, the \emph{Yang-Mills $k$-energy} is
\begin{equation}\label{eq:YMkenergy}
\mathcal{YM}_{k}(\N) := \tfrac{1}{2}\int_{M}{\left|\N^{(k)} F_{\N} \right|^2 dV_g}.\tag{\textsf{YMkE}}
\end{equation}
Critical points of this functional satisfy the corresponding Euler-Lagrange equation, called the \emph{Yang-Mills $k$-equation},
\begin{equation*}
0 = (-1)^k D_{\N}^* \lap^{(k)} F_{\N} + P^{(2k+1)}_1\lb F_{\N}\rb+ P^{(2k-1)}_2\lb F_{\N}\rb,
\end{equation*}
where the $P$ notation is defined in \eqref{sss:Pdefn}, and $D^*$ defined in \eqref{eq:Dstardefn}. With an eye towards future applications we construct a generalized negative gradient flow, called the \emph{generalized Yang-Mills $k$-flow} given by, for a one-parameter family $\N_t$ of connections,
\begin{equation*}
\frac{\del \N_t}{\del t} = (-1)^{k+1} D^*_{\N_t} \lap_t^{(k)} F_{\N_t} + \mho_k(\N_t),
\end{equation*}
where $\lap^{(k)}$ means $k$ iterations of the rough Laplacian, and $\mho_k(\N)$ is a lower order tensor featuring terms of the background manifold $M$ to be more precisely defined within the paper.

This higher order generalization of the Yang-Mills functional is in analogy with the work performed by Mantegazza in \cite{Mantegazza} where he studies higher order variants of mean curvature flow. These are inspired by and similar to functionals proposed by De Giorgi as a means to solving his open conjecture (cf. pp. 32 of \cite{Mantegazza} for a statement).

We now introduce a modified version of the Yang-Mills $k$-flow inspired by work performed by Hong, Tian and Yin \cite{HTY} in their study of the Yang-Mills $\ga$-flow. Consider the \emph{Yang-Mills $(\rho,k)$-energy} given by, for $\rho \in [0,\infty)$,
\begin{align}\label{eq:YMrhokenergydefn}
\begin{split}
\mathcal{YM}_{k}^{\rho} (\N) 
&:= \int_M \left( \rho \left| \N^{(k)}F_{\N} \right|^2 + \left| F \right|^2 \right) dV \\
&= \rho \mathcal{YM}_{k}(\N) + \mathcal{YM}(\N).
\end{split}\tag{\textsf{YM$\rho$kE}}
\end{align}
By studying the corresponding flow and sending $\rho \searrow 0$ one expects to, as in the case of \cite{HTY}, identify solutions to the Yang-Mills flow. The advantage of approaching with this quantity is that one is not restricted to any particular dimension; while \cite{HTY} focuses on $\dim M = 4$, in our case by choosing appropriate choices of $k$ one can regularize in any dimension.

\subsection{Main results and outline}\label{ss:mainres}

Here we state the main results of this paper and provide an outline of the layout.

\begin{customthm}{A}\label{eight}\label{thm:YMksubcritdim} Let $(E,h) \to (M,g)$ be a vector bundle over a smooth compact finite-dimensional boundariless Riemannian manifold. Let $\N$ be a smooth, metric compatible connection on $E$ and $F_{\N}$ its curvature tensor.
\begin{enumerate}
\item[(S)] (Subcritical) If $\dim M \in [2,2(k+2)) \cap \mathbb{N}$, there is a unique solution $\N_t$ to Yang-Mills $k$-flow with $\N_0 = \N$ existing for $t\in \mathbb{R}_{\geq 0}$.
\item[(C)] (Critical) If $\dim M = 2(k+2)$, there is a unique solution $\N_t$ to the Yang-Mills $k$-flow, with $\N_0 = \N$, existing on $[0,T)$ for some maximal $T \in \mathbb{R}_{\geq 0} \cup \{ \infty\}$. If $T < \infty$, then $\limsup_{t \nearrow T}{\brs{\brs{F_{\N_t}}}_{L^{\infty}(M)}} = \infty$, and there exists $\epsilon > 0$ and some $x \in M$ such that for all $r>0$, $\lim_{t \nearrow T}{\int_{B_{r}(x)}\brs{F_{\N_t}}^{k+2}} dVg \geq \epsilon$, where $B_x(r)$ denotes the geodesic ball of radius $r$ centered about $x$. Furthermore, there are only a finite number of points $x \in M$ where such behavior can occur at time $T$.
\end{enumerate}
\end{customthm}
\begin{rmk} Although in the case of our proofs we assume smoothness of the initial connection $\N$, if one assumes instead that $\N$ lies in certain Sobolev spaces depending on $k$ we can conclude a wider generalization of results. In particular for $k=0$ we may consequently confirm Theorem 1 of R\aa de \cite{Rade} and Theorem 2.3 of Struwe \cite{Struwe} in the case of Yang-Mills flow. Additionally, the properties of the Yang-Mills $1$-energy compare to those demonstrated by the analysis of Ichiyama, Inoguchi and Urakawa \cite{IIU} on the bi-Yang Mills energy, which is given by
\begin{equation}\label{eq:BYM}
\mathcal{BYM}(\N) := \tfrac{1}{2} \int_M \brs{D^*_{\N} F_{\N}}^2 dV_g.\tag{\textsf{BYME}}
\end{equation}
We will reflect on these two energies and their relationship in \S \ref{sss:BYMYM1}.
\end{rmk}

The second key result is a useful consequence of the analysis done to prove Theorem \ref{thm:YMksubcritdim}.

\begin{customthm}{B}\label{thm:regularizedflow}Let $(E,h) \to (M,g)$ be a vector bundle over a smooth compact finite-dimensional boundariless Riemannian manifold. Let $\N$ be a smooth, metric compatible connection on $E$ and $F_{\N}$ its curvature tensor. For all $\rho > 0$, if $\dim M \in [2, 2(k+2)) \cap \mathbb{N}$ there is a solution $\N_t$ to the Yang-Mills $(\rho,k)$-flow existing for $t \in \mathbb{R}_{\geq 0}$ and $\bm{\N} := \lim_{t \to \infty} \N_t$ is a solution to the Yang-Mills $(\rho,k)$-equation, given by
\begin{equation*}
0 = \rho \prs{(-1)^{k} D_{\N}^* \lap^{(k)} F_{\N} + P^{(2k+1)}_1\lb F_{\N}\rb+ P^{(2k-1)}_2\lb F_{\N}\rb} + D_{\N}^*F_{\N}. 
\end{equation*}
\end{customthm}

\begin{rmk} For Theorem \ref{thm:regularizedflow} the uniqueness of the solution at $t = \infty$ can be demonstrated by proving complete convergence rather than sequential. The proof hinges on verifying that the Yang-Mills $k$-functional satisfies the \L{}ojasiewicz-Simon gradient inequality, as in the proof of R\aa de in \cite{Rade}.
However, this verification is nontrivially technical and geometrically uninformative, so we leave this as an exercise for the not-so-faint of heart. We refer the reader to \S 7 of \cite{Feehan} for more information regarding the inequality including a discussion of its use by R\aa de.
\end{rmk}

Here is an outline of the paper. After laying out the groundwork and necessary identities in \S \ref{s:prelimcomps} we prove Theorem \ref{thm:YMksubcritdim} in \S \ref{s:flowsubcrit} (Existence and regularity results) and \S \ref{s:LTexist} (Long time existence results). 
We break down \S \ref{s:flowsubcrit} as follows.
\begin{itemize}
\item \textbf{\S \ref{ss:shorttimeexist} Short time existence.} We verify the short time existence of the generalized flow by performing gauge fixing to equate the gradient flow (being degenerate parabolic) with an equivalent parabolic system.
\item \textbf{\S \ref{ss:smoothingest} Smoothing estimates.} Using interpolation techniques we demonstrate estimates for quantities $|\N^{(\ell)}_t F_{\N_t}|$ for $\ell \in \mathbb{N}$ dependent on the temporal parameter and the supremum of the norm of curvature $|F_{\N_t}|$.
\item \textbf{\S \ref{ss:longtimeexist} Long time existence obstruction.} We demonstrate that the only impediment to long-time existence of the flow is a lack of supremal bound on the flow curvature.
\item \textbf{\S \ref{ss:blowupanalysis} Blowup analysis.} Using a blowup argument, we demonstrate that in the subcritical dimensions the supremum of the norm of the bundle curvature is bounded.
\end{itemize}

In the following section, \S \ref{s:LTexist}, we initially prove Theorem \ref{thm:YMksubcritdim}, then in \S \ref{ss:applications} we prove Theorem \ref{thm:regularizedflow} and finally conclude with a discussion applications and potential extensions of the Yang-Mills $k$-flow.

\subsubsection*{Acknowledgments}

The author would like to thank Jeffrey Streets for his unwavering support and encouragement, thoughtful advising and patient instruction. Additionally the author gratefully acknowledges support from the Simons Center for Geometry and Physics, Stony Brook University at which some of the research for this paper was performed.

This material is based upon work supported by the National Science Foundation Graduate Research Fellowship under Grant No. DGE-1321846. Any opinion, findings, and conclusions or recommendations expressed in this material are those of the author and do not necessarily reflect the views of the National Science Foundation.

\subsection{Notation and Conventions}\label{ss:notation}

Let $(E,h) \to (M,g)$ be a vector bundle over a smooth, compact, finite dimensional Riemannian manifold without boundary. Let $S(E)$ denote the smooth sections of $E$. For each point $x \in M$ choose a local orthonormal basis of $TM$ given by $\{ \del_i \}$ with dual basis $\{ d^i \}$ and a local basis for $E$ given by $\{ \mu_{\ga} \}$ with dual basis $\{ \mu_{\ga}^* \}$. Let $\Lambda^p(M)$ denote the set of smooth $p$-forms over $M$ and set $\Lambda^p(E) := \Lambda^p(M) \otimes S(E)$. Next set $\End E := E \ten E^*$, where $E^*$ denotes the dual space of $E$ and take
\begin{equation*}
\Lambda^p( \Ad E) := \{ \gw \in \Lambda^p(\End E) : \gw_{\ga}^{\gb} = - \gw_{\gb}^{\ga} \}.
\end{equation*}

The set of all smooth, bundle metric compatible connections on $E$ will be denoted by $\mathcal{A}_{E}$. The action of given connection $\N$ are extended to other tensor bundles by coupling it with the corresponding Levi-Civita connection on $(M,g)$. Given a chart containing $p \in M$ the action of a connection $\N$ on $E$ and on $M$ is captured by the \textit{(associated) connection coefficient matrices} $\Gamma  = ( \gG_{i \ga}^{\gb} d^i \otimes \mu_{\gb} \otimes \mu_{\ga}^* )$ and $G  = ( G_{i j}^{k} d^i \otimes d^j \otimes \del_k )$ respectively, where
\begin{align*}
\N \mu_{\gb} &= \gG_{i \gb}^{\gd} d^i \otimes \mu_{\gd} \\
\N \del_j &= G _{ij}^{k} d^i \otimes \del_k.
\end{align*}
The actions of $\N$ may be further extended to tensorial combinations of $TM$ and $E$ as well as their dual spaces.

Let $D$ be the \textit{exterior derivative}, or skew symmetrization of $\N$ over the tensor products of $TM$ where
\begin{align*}
D^{p}&: \Lambda^p(E) \to \Lambda^{p+1}(E) \\
&: \gw \mapsto \Alt(\N \gw),
\end{align*}
where $\Alt$ is the unscaled alternating map on tensors. The $p$ superscript will be typically supressed. The \emph{curvature tensor on $E$} and its coordinate formulation are given by
\begin{align*}
F_{\N} := D^{1} \circ D^{0} &: \Lambda^0(E) \rightarrow \Lambda^2(E) \\
&: (F_{\N})_{ij \ga}^{\gb} = \del_i \gG_{j \ga}^{\gb}  - \del_j \gG_{i \ga}^{\gb}  - \gG_{i \ga}^{\gd} \gG_{j \gd}^{\gb} + \gG_{j \ga}^{\gd} \gG_{i \gd}^{\gb}. \label{eq: Fcoords}
\end{align*}
The pointwise and global ($L^2$) inner products are given respectively by
\begin{align*}
\langle \cdot, \cdot \rangle, \left( \cdot, \cdot \right) : & \ \Lambda^p(\Ad E) \times \Lambda^p(\Ad E) \to \mathbb{R} \\
& \langle \gw , \gz \rangle := - \left( \prod_{v=1}^p g^{i_v j_v}\right) \gw_{i_1 \dots i_p \ga}^{\gb} \gz_{j_1 \dots j_p \gb}^{\ga}\\
&  \hsp \brs{\gz} := \sqrt{\ip{\gz,\gz}} \\
& (\gw,\gz ):= \int_M{\langle \gw , \gz \rangle dV_g}\\
&  \hsp  \brs{\brs{\gz}}_{L^2(M)} := \sqrt{\prs{\gz,\gz}},
\end{align*}
where $dV_g$ denotes the canonical volume form with respect to the metric $g$. Note that the definition of $\langle \cdot, \cdot \rangle$, and thus $\prs{\cdot,\cdot}$, can be extended, when necessary, as follows. Let $p , q \in \mathbb{N}$ with $p < q$. Let $K =(k_i)_{i=1}^p $ and $L =( l_i)_{i=1}^q$ multiindices, and let $\gw \in \Lambda^{p}(\Ad E)$ and $\xi \in \Lambda^{q}(\Ad E)$ respectively. Then
\begin{equation*}
\langle \gw , \xi \rangle := - \left( \prod_{i=1}^p{ g^{ k_i l_i}} \right) \gw_{K \gd}^{\gb} \xi_{L \gb}^{\gd}.
\end{equation*}
Considering the nonextended inner product, we set $\N^*$ to be the formal $L^2$ adjoint of $\N$. For computational purposes, $D^*_{\N}$, will denote a rescaled version of the formal $L^2$ adjoint of $D_{\N}$, satisfying for $\gw \in \Lambda^{p}(\Ad E)$ and $\psi \in \Lambda^{p-1}(\Ad E)$,
\begin{equation}\label{eq:Dstardefn}
\int_M \ip{D_{\N}^*\gw , \psi} dV_g = \frac{1}{p} \int_M \ip{ \gw, D_{\N}\psi }dV_g.
\end{equation}
In coordinates, this is given simply by $(D^*_{\N} \gw) _{i_1 \dots i_{p} \ga}^{\gb} = g^{j i_1}\N_{j} \gw_{i_1 \dots i_{p} \ga}^{\gb}$. The \emph{rough} and \emph{Hodge Laplacian} are given respectively by
\begin{align*}
\lap &: \Lambda^p(E) \rightarrow \Lambda^p(E) \\
&: \gw \mapsto - \N^* \N \gw,\\
\lap_{D_{\N}} &: \Lambda^p(E) \rightarrow \Lambda^p(E) \\
&: \omega \mapsto (D_{\N}^* D_{\N} + D_{\N} D_{\N}^*)\omega.
\end{align*}
\subsubsection*{Generalized operations.} Let $\gw, \gz \in \Lambda^p(E)$. Let $\gw \ast \gz$ express any normal-valued, multilinear form depending on $\gw$ and $\gz$ in a universal bilinear way. Furthermore by the Cauchy-Schwartz inequality there exists some $C > 0 $ such that
\[ \brs{ \gw \ast \gz } \leq C \brs{\gw} \brs{\gz}, \]
where here $\ast$ simply denotes some algebraic action.
\subsubsection*{Unconventional operators.} 
Let $J:= (j_i)_{i=1}^{\brs{J}}$ and $K := (k_i)_{i=1}^{\brs{K}}$ be multiindices and let $j$, $k$ be a distinct indices from those in $J$ and $K$ respectively. The operation \textit{pound} is given by
\begin{align*}
\# &: (T^*M)^{\ten \brs{J} + 1} \ten (\End E) \times (T^*M)^{\ten \brs{K} + 1} \ten (\End E)
\rightarrow (T^*M)^{\brs{J} + \brs{K}} \ten ( \End E), \\
 & (\gw \# \gz )\prs{\del_{j_1}, ..., \del_{j_{\brs{J}}}, \del_{k_1}, ... , \del_{k_{\brs{K}}} } = \sum_{i=1}^n \gw \prs{\del_{i}, \del_{j_1}, ..., \del_{j_{\brs{J}}}} \gz\prs{
\del_i,\del_{k_1}, ... , \del_{k_{\brs{K}}}}.
\end{align*}
In coordinates this is written in the form $ (\gw \# \gz)_{JK \ga}^{\gb} = g^{jk}\gw_{j J \gd}^{\gb}\gz_{k K \ga}^{\gd}$.
Roughly speaking, $\#$ is matrix multiplication combined with contraction of the first two forms. We define the corresponding \emph{pound bracket} by
\begin{align}
\begin{split}\label{eq:pndbrackdefn}
[\gw,\gz]^{\#}\prs{ \del_{j_1}, ... , \del_{j_{\brs{J}}}, \del_{k_1}, ..., \del_{k_{\brs{K}}} } &:= (\gw \# \gz)\prs{\del_{j_1}, ..., \del_{j_{\brs{J}}}, \del_{k_1}, ... , \del_{k_{\brs{K}}}} \\
& \hsp - (\gz \# \gw)\prs{\del_{j_1}, ..., \del_{j_{\brs{J}}}, \del_{k_1}, ... , \del_{k_{\brs{K}}}}.
\end{split}
\end{align}
\subsubsection*{Derivatives.}\label{sss:Pdefn} When not performing coordinate computations, for a connection $\N$, we will reserve upper indices without parentheses for indexing sequences and with parentheses for iterations of differentiation. That is,
\begin{align*}
\{ \N^i \}_{i \in \mathbb{N}} &= \{ \N^1, \N^2, \cdots \},\\
\N^{\prs{i}} &= \underset{\text{$i$ times}}{\underbrace{\N \cdots \N}}.
\end{align*}
Note the lower index is reserved for the temporal parameter.

We utilize the $P$ notation introduced in \cite{KS}. Let $v \in \mathbb{N}$ and $W := (w_i)_{i=1}^v$, and set $s = \sum_{i=1}^v w_i$. Let $R$ denote some background tensor dependent only on $g$, so that
\begin{equation*}
P^{\prs{s}}_{v} \brk{ \gw } := (\N^{(w_1)} \gw) \ast \cdots \ast (\N^{(w_v)} \gw) \ast R. 
\end{equation*}
\subsubsection*{Notation.} Many quantities are one-parameter families, and we will often call this parameter the `temporal' parameter. This parametrization will be denoted with $t$ parameter, however when understood the $t$ subscript will be dropped. Differentiation with respect to $t$ will sometimes be indicated with `$\cdot$' for notational convenience. A geodesic ball centered at a point $x \in M$ with radius $r$ will be denoted by $B_x(r)$.

\section{Preliminary computations}\label{s:prelimcomps}
\subsection{General variations}

The variations of general one-parameter families will be computed in preparation for the work of \S \ref{s:flowsubcrit}. For the remainder of this section, we will let $\mathcal{I}$ denote some simply connected subset of $\mathbb{R}$ over which $\N_t$ is parameterized.

\begin{lemma}\label{lem:varNellA}
Suppose $\N_t \in \mathcal{A}_E \times \mathcal{I}$ and $\gw_t \in \Lambda^{p}(\End E) \times \mathcal{I}$. Then for $\ell \in \mathbb{N}$,
\begin{equation}\label{eq:varNellA}
\frac{\del}{\del t} \lb \N_t^{(\ell)} \gw_t \rb = \sum_{i = 0}^{\ell-1}{ \left( \N_t^{(i)} \dot{\gG}_t \right) \ast \left( \N_t^{(\ell-i-1)} \gw_t \right)} + \left( \N_t^{(\ell)} \dot{\gw_t} \right).
\end{equation}

\begin{proof}
The proof will proceed by induction on $\ell \in \mathbb{N}$ satisfying \eqref{eq:varNellA}. Let $J:=(J_w)_{w=1}^{|J|}$ be a multiindex and set 
\begin{equation*}
J(w,s):=
\begin{cases}
p_r &\text{ if } r \neq w,\\
s &\text{ if } r = w.
\end{cases}
\end{equation*}
Roughly speaking $J(w,s)$ substitutes the $w$th element of the $J$ multiindex with an $s$. For $\ell = 1$ applying normal coordinates yields
\begin{align}\label{eq:varN1Acoord}
\begin{split}
\del_t \lb \N_i  \gw_{J \ga}^{\gb} \rb
  &= \del_i  \dot{\gw}_{L \ga}^{\gb} - \sum_{w =1}^{|P|}{\left( \dot{G}_{i j_w}^{s}  \gw_{J(w,s) \ga}^{\gb} + G_{i j_w}^{s} \dot{\gw}_{J(w,s) \ga}^{\gb} \right)} \\
& \hsp - \gw_{J \gd}^{\gb} \dot{\gw}_{i \ga}^{\gd}  -  \dot{\gw}_{J \gd}^{\gb} \gG_{i \ga}^{\gd} + \dot{\gG}_{i \gd}^{\gb} \gw_{J \ga}^{\gd} + \gG_{i \gd}^{\gb} \dot{\gw}_{J \ga}^{\gd}.
\end{split}
\end{align}
Hence the base case holds, giving
\begin{equation}\label{eq:varN1A}
\tfrac{\del}{\del t} \lb \N \gw \rb = \N \dot{\gw} + \dot{\gG} \ast \gw.
\end{equation}
Now assume the induction hypothesis \eqref{eq:varNellA} is satisfied for $\ell \in \mathbb{N}$ and let $L$ be a multiindex with $|L|=\ell$. We compute
\begin{align*}
\left( \tfrac{\del}{\del t} \lb \N^{(\ell+1)} \gw \rb \right)_{jLP \ga}^{\gb}
&= \del_j \lb \tfrac{\del}{\del t} \N^{(\ell)}\gw \rb_{LP \ga}^{\gb} + \tfrac{\del}{\del t}\lb  \gG_{p \gd}^{\gb}(\N^{(\ell)}\gw_{LP \ga}^{\gd}) - (\N^{(\ell)} \gw_{LP \gd}^{\gb}) \gG_{j \ga}^{\gd} \rb \\
&= \del_j \lb \N^{(i)} \dot{\gG} \ast \N^{(\ell-i-1)} \gw + \N^{(\ell)}\dot{\gw}\rb_{LP \ga}^{\gb} \\
& \hsp-\dot{\gG}_{j \ga}^{\gd} (\N^{(\ell)} \gw_{LP \gd}^{\gb}) + \dot{\gG}_{j \gd}^{\gb} (\N^{(\ell)}\gw_{LP \ga}^{\gd}).
\end{align*}
Or, written in coordinate invariantly,
\begin{align*}
\tfrac{\del}{\del t} \lb \N^{(\ell+1)} \gw \rb
&= \N^{(i+1)} \dot{\gG} \ast \N^{(\ell-i-1)} \gw + \N^{(i)}\dot{\gG} \ast \N^{(\ell-i)} \gw + \N^{(\ell+1)} \dot{\gw} + \dot{\gG}\ast \N^{(\ell)}\gw \\
&= \N^{(i)} \dot{\gG} \ast \N^{(\ell-i-1)} \gw + \N^{(\ell+1)} \dot{\gw}.
\end{align*}
Hence $\ell+1$ satisfies the induction hypothesis, so inductively the result follows.
\end{proof}
\end{lemma}

\begin{cor}\label{cor:varNellF}
Suppose $\N_t \in \mathcal{A}_E \times \mathcal{I} $. Then for $\ell \in \mathbb{N}$
\begin{equation}\label{eq:varNellF}
\frac{\del}{\del t} \lb \N^{(\ell)} F_{\N_t} \rb= \sum_{i = 0}^{\ell-1}{\N_t^{(i)}  \dot{\gG}_t \ast \N_t^{(\ell-i-1)} F_{\N_t}} + \N^{(\ell)}_t D_{\N_t} \dot{\gG}_t.
\end{equation}
Furthermore, $ \tfrac{\del F_{\N_t}}{\del t} = D^*_{\N_t}\dot{\gG}_t$. 
\end{cor}

\subsection{Flow specific variations}

We first compute the Euler Lagrange equation of the Yang-Mills $k$-energy to determine the corresponding Yang-Mills $k$-flow, and introduce a generalized version of the flow to expand the scope of our results. We then demonstrate that the Yang-Mills flow is a weakly parabolic system.

\begin{prop}\label{prop:ELYMk}
For $\N_t \in \mathcal{A}_E \times \mathcal{I}$ the variation of the Yang-Mills $k$-energy is given by
\begin{equation}
 \tfrac{d}{d t} \left[ \mathcal{YM}_{k} (\N_t) \right] = \int_M{ \left\langle \Grad \mathcal{YM}_{k} (\N_t), \dot{\gG}_t \right\rangle dV_g},
\end{equation}
where
\begin{equation}
\Grad \mathcal{YM}_{k} (\N) := (-1)^k D_{\N}^* \lap^{(k)} F_{\N} + \sum_{v = 0}^{2k-1} P^{(v)}_1 \lb F \rb  + P^{(2k-1)}_2\lb F\rb.
\end{equation}

\begin{proof}
Differentiating the Yang-Mills $k$-energy with respect to $t$ yields
\begin{align*}
\frac{d}{d t}\left[ \tfrac{1}{2} \int_M{|\N^{(k)}F|^2 dV_g}\right]  &= \int_{M}{ \left\langle \tfrac{\del}{\del t} \left[ \N^{(k)}F \right], \N^{(k)}F \right\rangle dV_g}.
\end{align*}
Appealing to Corollary \ref{cor:varNellF} for the variation of $\N^{(k)}F$ yields,
\begin{align}
\begin{split}\label{eq:ELYMk}
\frac{d}{d t}\left[ \tfrac{1}{2} \int_M{|\N^{(k)}F|^2 dV_g}\right]  &= \int_M{\left\langle \left( \sum_{i = 0}^{k-1}{\N^{(i)} \dot{\gG} \ast \N^{(k-i-1)} F} + \N^{(k)} D \dot{\gG} \right), \N^{(k)} F \right\rangle dV_g}\\
&= \int_M{\left\langle \left( \sum_{i = 0}^{k-1}{\N^{(i)} \dot{\gG} \ast \N^{(k-i-1)} F }\right), \N^{(k)} F \right\rangle dV_g} +  \int_M{ \left\langle \N^{(k)} D \dot{\gG} , \N^{(k)} F \right\rangle dV_g}.
\end{split}
\end{align}
For the first integral of \eqref{eq:ELYMk} integration by parts gives,
\begin{align*}
\int_M{\left\langle \left( \sum_{i = 0}^{k-1}{\N^{(i)} \dot{\gG} \ast \N^{(k-i-1)} F }\right), \N^{(k)} F \right\rangle dV_g}  &= \int_{M} \left\langle \dot{\gG} , P_2^{(2k-1)}[F] \right\rangle dVg.
\end{align*}
The second integral of \eqref{eq:ELYMk} is addressed with Lemma \ref{lem:shiftkNs} to recursively integrate by parts,
\begin{align*}
 \int_M{ \left\langle \N^{(k)} D \dot{\gG} , \N^{(k)} F \right\rangle dV_g} &= \int_M{\left\langle D\dot{\gG}, \sum_{v = 1}^{2k-2} \sum_{w = 0}^{v} \N^{(w)} \Rm \ast \N^{(2k-2-w)} F  + P^{(2k-2)}_2\lb F\rb  \right\rangle dV_g} \\
& \hsp  + \int_M{ \left\langle D \dot{\gG} , (-1)^k \lap^{(k)} F \right\rangle dV_g}\\
 &=\int_M{\left\langle D\dot{\gG}, \sum_{v = 0}^{2k-2} P^{(v)}_1 \lb F \rb  + P^{(2k-2)}_2\lb F\rb  \right\rangle dV_g}  +   \int_M{ \left\langle D \dot{\gG} , (-1)^k \lap^{(k)} F \right\rangle dV_g}\\
  &=  \int_M{\left\langle \dot{\gG}, \sum_{v = 0}^{2k-1} P^{(v)}_1 \lb F \rb  + P^{(2k-1)}_2\lb F\rb  \right\rangle dV_g} + 2 \int_M{ \left\langle \dot{\gG} , (-1)^k D^* \lap^{(k)} F \right\rangle dV_g} .
\end{align*}
Combining the integrands and comparing we therefore conclude
\begin{align*}
\Grad \mathcal{YM}_{k} (\N) = (-1)^k \lap^{(k)} F +   \sum_{v = 0}^{2k-1} P^{(v)}_1 \lb F \rb  + P^{(2k-1)}_2\lb F\rb.
\end{align*}
The result follows.
\end{proof}
\end{prop}

Given $\N_t \in \mathcal{A}_E \times \mathcal{I}$ define the \emph{Yang-Mills $k$-flow} by
\begin{equation}\label{eq:YMkflow}
\frac{\del \N_t}{\del t}  = - \Grad \mathcal{YM}_{k} (\N_t)=  (-1)^{k+1} D_{\N_t}^* \lap_t^{(k)} F +  P^{(2k-1)}_2\lb F_{\N_t}\rb.\tag{\textsf{YMkF}}
\end{equation}
Observe that setting $k = 0$ in \eqref{eq:YMkflow} and omitting the lower order terms immediately yields Yang-Mills flow,
\begin{equation*}
\frac{\del \N_t}{\del t} = - D^*_{\N_t} F_{\N_t}.
\end{equation*}
For future work it is advantageous to perform a more general analysis. To do so set, for $k \in \mathbb{N}$,
\begin{equation}\label{eq:Omegak}
\mho_k(\N) := \sum_{i=1}^{k} \sum_{j=2}^{i + 1}{P_j^{(2
i + 3 -2j)} \lb F_{\N} \rb}.
\end{equation}
We additionally set $\mho_0 \equiv 0$. The \emph{generalized Yang-Mills $k$-flow} is given by
\begin{equation*}
\frac{\del \N_t}{\del t} = (-1)^{k+1} D^*_{\N_t} \lap_t^{(k)}F_{\N_t} + \mho_k(\N_t).
\end{equation*}
Given a connection $\N \in \mathcal{A}_E$, a \emph{generalized Yang-Mills $k$-system} will be the generalized Yang-Mills $k$-flow coupled with $\N_0 := \N$ as the initial condition.

Next we demonstrate the weak ellipticity of the generalized Yang-Mills $k$-flow, which is a result of the gauge invariance of the function (cf. Corollary \ref{lem:YMkgaugeinv}).

\begin{prop}\label{prop:Phikwkellptc}
Set 
\begin{align*}
\Phi_k &: \mathcal{A}_E \to \Lambda^1(\End E) \\
&: \N \mapsto (-1)^{k+1} D_{\N}^* \lap^{(k)} F_{\N}.
\end{align*}
Then $\Phi_k$ is a weakly elliptic operator.

\begin{rmk}
Note that $\Phi_k$ is the highest order term of the generalized Yang-Mills $k$-flow and thus its symbol is completely determined by this.
\end{rmk}

\begin{proof}
Let $I = (i_v)_{v=1}^{k}$ be some multiindex. Appealing to Corollary \ref{cor:varNellF} for the variation of $\N^{(2k+1)} F_{\N_t}$ yields,
\begin{align*}
\frac{\del}{\del t} \left[ \left( \Phi_k(\N_t) \right)_{r \ga}^{\gb} \right] 
&= (-1)^{k} \frac{\del}{\del t} \left[ g^{pq} \left( \prod_{\ell = 1}^k{g^{i_{\ell} j_{\ell}}} \right) \N_p \N_{i_1 j_1 \cdots i_k j_k} F_{q r \ga}^{\gb}\right] \\
&=  (-1)^{k} g^{pq} \left( \prod_{\ell = 1}^k{g^{i_{\ell} j_{\ell}}} \right) \frac{\del}{\del t} \left[ \N_p \N_{i_1 j_1 \cdots i_k j_k} F_{q r \ga}^{\gb}\right] \\
&=(-1)^{k} g^{pq} \left( \prod_{\ell = 1}^k{g^{i_{\ell} j_{\ell}}} \right)\left( \N_p \N_{i_1 j_1 \cdots i_k j_k} D_{q}\dot{\gG}_{r \ga}^{\gb}  + \left(\sum_{i=0}^{k-1} \N^{(i)}\dot{\gG} \ast \N^{(k-i-1)} F\right)_{p i_1 j_1 \cdots i_k j_k q r \ga}^{\gb} \right).
\end{align*}
The left hand term is of dominating order and thus determines the symbol, which is given by, for $B \in \Lambda^1(E)$,
\begin{equation*}
\left( \sigma \left[ \Phi_k \right](B) \right)_{r \ga}^{\gb} = (-1)^k  g^{pq} \left( \prod_{\ell = 1}^k{g^{i_{\ell} j_{\ell}}} \right) \del_p \del_{i_1 j_1 \cdots i_k j_k} \left( \del_{q}B_{r \ga}^{\gb} - \del_{r}B_{q \ga}^{\gb} \right),
\end{equation*}
and thus
\begin{align}
\begin{split}\label{eq:LPhik}
\left( L_{\Phi_k}^{\xi} (B) \right)_{r \ga}^{\gb}&:= (-1)^k   g^{pq} \left( \prod_{\ell = 1}^k{g^{i_{\ell} j_{\ell}}} \right) \xi_p \left( \prod_{s = 1}^k{\xi_{i_s} \xi_{j_s}} \right) \left( \xi_{q}B_{r \ga}^{\gb} - \xi_{r}B_{q \ga}^{\gb} \right)\\
&=(-1)^k |\xi|^{2k} \left( |\xi|^2 B_{r \ga}^{\gb} - \xi_r \langle B, \xi \rangle_{\ga}^{\gb}  \right).
\end{split}
\end{align}
Therefore it follows that
\begin{equation*}
\left\langle L_{\Phi_k}^{\xi}(B) , B \right\rangle = (-1)^{k+1} |\xi|^{2k} \left( |\xi|^2 |B|^2 - \left| \langle B, \xi \rangle \right|^2 \right).
\end{equation*}
This term is nonnegative by the Cauchy Schwartz Inequality. This completes the proof, but let us proceed and identify the kernel of $L_{\Phi_k}^{\xi}$. After changing bases (via rotation and dilation of the space of $\xi$) one can take $\xi: = (\delta_{1}^{\ell})_{\ell = 1}^{n}$. Evaluating at such $\xi$ gives that
\[
\left\langle L_{\Phi_k}^{\xi}(B) , B \right\rangle = (-1)^{k+1}  \prs{|B|^2 - B_{1 \ga}^{\gd} B_{1 \gd}^{\ga}}.
\]
Therefore the kernel of the operator
\begin{align*}
\mathfrak{L}_{\Phi_k}^{\xi} &: \Lambda^1(\End E) \to \mathbb{R} \\
&: B \mapsto \langle L_{\Phi_k}^{\xi} (B), B \rangle.
\end{align*}
is the set of sections $B \in S(T^*M \ten \End E)$ of the form $B = (\gd_{1}^{k} B_{k \ga}^{\gb})$. Thus $\dim{( \ker{ \mathfrak{L}_{\xi} })} = \dim{(\End E)} $.
\end{proof}
\end{prop}

\section{Existence and regularity results}\label{s:flowsubcrit}
\subsection{Short time existence}\label{ss:shorttimeexist}

We next demonstrate the short time existence of generalized Yang-Mills $k$-flow. Despite the weak ellipticity of the operator $\Phi_k$ demonstrated in Proposition \ref{prop:Phikwkellptc}, one may construct a `moving gauge' which actively shifts the flow to an equivalent parabolic system, thus ensuring short time existence.

The active shift of gauge transformations within the gauge group $\mathcal{G}_E$ is achieved by solving the one-parameter family of gauge transformations which satisfy the following ordinary differential equation. Given a family of connections $\N_t \in \mathcal{A}_E \times \mathcal{I}$ and a one-parameter family $\varsigma_t \in S(\Aut E) \times \mathcal{I}$ satisfying the flow
\begin{equation}\label{eq:YMsflow}
\frac{\del \varsigma_t}{\del t} = (-1)^{k+1} \left( \lap^{(k)}_t D_{\N_t}^* (\gG_t - \gG_0) \right)\varsigma_t,
\end{equation}
with the additional imposed initial condition $\vs_0 := \Id$. The corresponding parabolic system will be written with respect to the following operator.
\begin{align}
\begin{split}\label{eq:Psik}
\Psi_k(\N,\tN) &: \mathcal{A}_E \times \mathcal{A}_E  \to \Lambda^1(\End E) \\
&: (\N,\tN) \mapsto (-1)^{k+1} D_{\N}^* \lap^{(k)} F_{\N}+ \mho_k(\N) + (-1)^{k}  D_{\N} \lap^{(k)} D_{\N}^* \prs{\gG - \tgG }.
\end{split}\tag{\textsf{$\Psi$kF}}
\end{align}

\begin{defn}[$(\Psi,k)$-flow]
We will call a one-parameter family $\N_t \in \mathcal{A}_E \times \mathcal{I}$ with initial condition $\N_0$ a solution to the \emph{$(\Psi,k)$-flow} if
\begin{equation*}
\tfrac{\del \N_t}{\del t} = \Psi_k(\N_t,\N_0).
\end{equation*}
\end{defn}

\begin{lemma}\label{lem:Psikellptc} For a fixed connection $\N_0 \in \mathcal{A}_E$, $\Psi_k(\cdot,\N_0)$ is an elliptic operator.

\begin{proof} The symbol of $\Psi_k$ will be computed as follows: since the variation of the first term, which is $\Phi_k$, was computed in Proposition \ref{prop:Phikwkellptc}, it is sufficient to first consider the variation of the latter quantities. Set
\begin{align*}
\prs{\Theta_k \prs{\N,\tN}}_{p \ga}^{\gb}
&:=(-1)^{k} \left( D_{\N} \lap^{(k)} D_{\N}^* \prs{\gG - \tgG} \right)_{p \ga}^{\gb}\\
 &= (-1)^{k+1} g^{qr} \left( \prod_{\ell = 1}^k{g^{i_{\ell} j_{\ell}}} \right) \left( \N_p \N_{i_1 j_1 \cdots i_k j_k} \N_{q} \left( \gG - \tgG \right)_{r \ga}^{\gb} \right).
\end{align*}
Then take $\tN = \N_0$, and consider a one-parameter family $\N_t \in \mathcal{A}_E \times \mathcal{I}$ with $\N_0$ as the initial condition. We differentiate temporally and appeal to Lemma \ref{lem:varNellA} to observe that there is only one term of highest order (specifically order $2k+3$).
\begin{align*}
\frac{\del}{\del t} \lb (\Theta_k \prs{\N_t, \N_0 })_{p \ga}^{\gb} \rb &= (-1)^{k+1} g^{qr} \left( \prod_{\ell = 1}^k{g^{i_{\ell} j_{\ell}}} \right) \left( \N_p \N_{i_1 j_1 \cdots i_k j_k} \N_{q} \dot{\gG}_{r \ga}^{\gb} \right).
\end{align*}
Therefore for $B \in \Lambda^1(\End E)$,
\begin{align*}
(\sigma[\Theta_k](B) )_{p \ga}^{\gb}  &=  (-1)^{k}  g^{qr} \left( \prod_{\ell = 1}^k{g^{i_{\ell} j_{\ell}}} \right) \del_p( \del_{i_1} \del_{j_1} \cdots \del_{i_k} \del_{j_k} \del_{q} B_{r \ga}^{\gb} ), 
\end{align*}
and so
\begin{equation}\label{eq:LThetak}
\left(L_{\Theta_k}^{\xi}(B) \right)_{p \ga}^{\gb} = (-1)^{k}\xi_{p} |\xi|^{2k} \langle \xi, B \rangle_{\ga}^{\gb}.
\end{equation}
Then by combining \eqref{eq:LPhik} and \eqref{eq:LThetak} and noting they have the same orders,
\begin{equation*}
\left( L_{\Psi_k}^{\xi}(B) \right)_{p \ga}^{\gb} = \left( L_{\Phi_k}^{\xi}(B) + L_{\Theta_k}^{\xi}(B) \right)_{p \ga}^{\gb} = (-1)^{k+1} |\xi|^{2k+2} B_{p \ga}^{\gb}.
\end{equation*}
Now we observe that
\begin{equation*}
\left\langle L_{\Psi_k}^{\xi}(B), B \right\rangle = (-1)^k |\xi|^{2k+2} |B|^2.
\end{equation*}
Thus $\langle L_{\Psi_k}^{\xi}(\cdot), \cdot \rangle$ is either strictly positive definite or negative definite depending on the parity of $k$. We conclude that $\Psi_k$ is an elliptic operator and the result follows.
\end{proof}
\end{lemma}

We now develop some necessary identities regarding the action of gauge transformations on various quantities. The majority are included within the appendix in the gauge transformations section (\S \ref{ss:gaugetransformations}), though the most relevant will be stated here.

\begin{rmk}
For $\varsigma \in S( \Aut E)$ and $\N \in \mathcal{A}_E$, set
\begin{equation*}
\lap_\varsigma := g^{ij}(\varsigma^* \N)_i (\varsigma^* \N)_j. 
\end{equation*}
With this notation, the action of $\varsigma$ on $\lap$ we have the identities
\begin{align*}
\varsigma^* \lap (\cdot) := \varsigma^{-1} \lap (\varsigma \cdot) = \lap_{\varsigma} (\varsigma \cdot).
\end{align*}
Additionally an analogous statement of Lemma \ref{lem:gaugeactid1} applies where $\lap$ replaces $\N$:
\begin{equation*}
\varsigma^* \left[ \lap (\gw) \right] = (\varsigma^* \lap) \left[ \varsigma^* \gw \right] = \lap_{\varsigma} \left[ \varsigma^* \gw \right].
\end{equation*}
\end{rmk}

To understand the intuition behind the following proof, we recall its primary inspiration and most basic case (Yang-Mills flow, $k=0$) given in (\cite{DonB}, pp. 233-235), though we have simplified the strategy. We then address the general setting for $k \in \mathbb{N}$ after. The short time existence is not immediately clear since Yang-Mills flow itself fails to be parabolic due to the infinite-dimensional gauge symmetry group. One correctly expects, given a solution $\N_t$, that its the geometric content should be preserved in its projection $\brk{ \N_t } \in \mathcal{A}_E / \mathcal{G}_E \times \mathcal{I}$ in the space of connections modulo gauge transformation. One chooses another family within $\brk{ \N_t }$ which is not gauge invariant but moves smoothly and transversely to the action of the gauge group. This ensures that the degeneracy is removed and thus the family is parabolic and so exists for short time. This new flow can be represented uniquely by a family of gauge transformations $\varsigma_t \in \mathcal{G}_E \times \mathcal{I}$ applied to the initial flow.

Before continuing we define a notational convention which will condense more complicated terms produced from the lower order terms in differentation.

\begin{defn}[Partition strings]\label{defn:partitionstring}
Let $L := (l_i)_{i=1}^{|L|}$ denote some multiindex, $\mathcal{P}_r(L)$ denote the \emph{partition strings of $L$}, that is, the collection of multiindices of length $r \leq |L|$ which contains entries ordered with respect to $L$,
\begin{equation*}
 \mathcal{P}_{r}(L) := \left\{ (l_{s_v})_{v = 1}^r : s_{v} \in [1,m], s_{v} < s_{v+ 1} \right\}.
 \end{equation*}
For example,
\begin{align*}
\mathcal{P}_1(L) &:= \left\{ (l_s): s \in [1,m] \cap \mathbb{N} \right\}, \\ \mathcal{P}_2(L) &:= \left\{ (l_{s_1}, l_{s_2}): s_1,s_2 \in [1,m] \cap \mathbb{N}, s_1 < s_2 \right\}.
\end{align*}
 Given $\mathsf{P} \in \mathcal{P}_r(L)$ of the form $\mathsf{P} := (l_{s_v})_{v = 1}^r$ we let $\mathsf{P}^c$ denote the complimentary string, where $\mathsf{P}^c \in \mathcal{P}_{m-r}(L)$ and, roughly speaking as sets, $(\mathsf{P} \cup \mathsf{P}^c)= L$.
\end{defn}

We next define the following operator which is formulated for notational convenience. Its construction is motivated by the following Lemma \ref{lem:YMkslaplap}, and is purely a technical quantity in terms of the lower order objects. This is utilized in in demonstrating uniqueness of the flow.
\begin{align}
\begin{split}\label{eq:akdefn}
\mathfrak{a}_k &: S(\Aut E) \times \mathcal{A}_E^{\times 2} \to S(\End E), \\
& : (\varsigma,\N,\widetilde{\N}) \mapsto \varsigma_{\tau}^{\gb} \N^{(k)} \left( D_{\N}^* \lb \gG - \widetilde{\gG} \rb \right)_{\ga} ^{\tau} + g^{ij} \varsigma^{\gb}_{\tau} \lap^{(k)} \N_j \lb (\varsigma^{-1})_{\gz}^{\tau} \rb (D_{\N} \varsigma)_{i \ga}^{\gz}  \\
& \hsp \hsp \hsp \hsp \hsp + g^{ij} \varsigma_{\tau}^{\gb} \lap^{(k)} \lb \varsigma^{-1} \rb_{\gz}^{\tau} \lap \lb \varsigma \rb_{\ga}^{\gz}  \\
& \hsp \hsp \hsp \hsp \hsp + g^{ij}\varsigma_{\tau}^{\gb} \left( \prod_{\ell = 0}^{k} g^{i_{\ell} j_{\ell}}\right) \sum_{r = 1}^{k-1}{ \sum_{\mathsf{P} \in \mathcal{P}_r(L)}{\left( \N_{\mathsf{P}} (\varsigma^{-1})_{\gz}^{\tau} \right)\left( \N_{\mathsf{P}^c}(\lap)\left[ \varsigma \right]_{\ga}^{\gz} \right) }}.
\end{split}
\end{align}

\begin{lemma}\label{lem:YMkslaplap} Let $\N, \tN \in \mathcal{A}_E$ and $\varsigma \in S(\Aut E)$. The following equality holds
\begin{equation}\label{eq:YMkslaplap}
\prs{\lap^{(k)} \left[ D^*_{\N} \prs{\gG - \tgG } \right]} \varsigma = - \slap^{(k+1)} \varsigma + \mathfrak{a}_k \prs{\varsigma, \varsigma^*\N , \tN}.
\end{equation}

\begin{rmk} The equality given in \eqref{eq:YMkslaplap} is the key to establishing uniqueness of the generalized Yang-Mills $k$-flow by establishing the correspondence between this flow and the $(\Psi,k)$-flow. This relationship lies in this miraculous `dictionary' equality between the flows. On the left side of \eqref{eq:YMkslaplap} is an algebraic interaction of tensors on the gauge transformation. On the right is an act of differentiation of the gauge transformation plus lower order terms.
\end{rmk}

\begin{proof} 
First an expression for the difference of the connection coefficient matrices $\gG$ and $\tgG$ with respect to $\sNG$ will be attained through first forming `$\varsigma$-conjugations' of $\gG$ and the coordinate expansion of $\sNG$ (Lemma \ref{lem:coordsN}).
\begin{align*} 
\gG_{i \ga}^{\gb} - \tgG_{i \ga}^{\gb} &= \varsigma^{\gb}_{\gd} (\varsigma^{-1})^{\gd}_{\rho} \gG_{i \tau}^{\rho} \varsigma^{\tau}_{\gt} (\vs^{-1})^{\gt}_{\ga} - \tgG_{i \ga}^{\gb} \\
&= \varsigma^{\gb}_{\gd} (\sNG)_{\gamma}^{\gd} (\varsigma^{-1})_{\ga}^{\gamma}- \varsigma^{\gb}_{\gd}(\varsigma^{-1})^{\gd}_{\gt}(\del_i \varsigma^{\gt}_{\gamma}) (\varsigma^{-1})_{\ga}^{\gamma} - \tgG_{i \ga}^{\gb} \\
&=\varsigma^{\gb}_{\gd} \left( (\sNG)_{\gamma}^{\gd} - (\varsigma^{-1})^{\gd}_{\gt}(\del_i \varsigma^{\gt}_{\gamma})\right)(\varsigma^{-1})_{\ga}^{\gamma} - \tgG_{i \ga}^{\gb}.
\end{align*}
Using one more identity,
\begin{align*}
(D_{\varsigma^* \N}\varsigma)_{i \ga}^{\gb} &= \del_i \varsigma_{\ga}^{\theta} + (\sNG)_{i \gd}^{\gb} \varsigma_{\ga}^{\gd} - (\sNG)_{i \ga}^{\gd} \varsigma_{\gd}^{\gb},
\end{align*}
one sees that
\begin{align*}
\gG_{i \ga}^{\gb} - \tgG_{i \ga}^{\gb}
&= \varsigma^{\gb}_{\gd} \left( (\varsigma^* \N)_{i \gamma}^{\gd} - (\varsigma^{-1})^{\gd}_{\gt} \left( (D_{\varsigma^* \N}\varsigma)_{i \gamma}^{\gt}  - (\sNG)_{i \tau}^{\gt}\varsigma_{\gamma}^{\tau} + \varsigma^{\gt}_{\tau} (\sNG)^{\tau}_{ i \gamma} \right)\right)(\varsigma^{-1})_{\ga}^{\gamma} - \tgG_{i \ga}^{\gb} \\
&= \varsigma^{\gb}_{\gd} \left( - (\varsigma^{-1})^{\gd}_{\gt} \left( (D_{\varsigma^* \N}\varsigma)_{i \gamma}^{\gt}  - (\sNG)_{i \tau}^{\gt}\varsigma_{\gamma}^{\tau} \right) - (\varsigma^{-1})^{\gd}_{\gz} \tgG_{i \gt}^{\gz} \varsigma^{\gt}_{\gamma} \right)(\varsigma^{-1})_{\ga}^{\gamma} \\
&= - (D_{\varsigma^* \N} \varsigma)_{i \gamma }^{\gb} (\varsigma^{-1})_{\ga}^{\gamma} + (\sNG)_{i \ga}^{\gb} - \tgG_{i \ga}^{\gb}.
\end{align*}
We apply this to the following computation.
\begin{align*}
 \lap^{(k)} D_{\N}^*(\gG - \tgG)_{\gd}^{\gb} \varsigma_{\ga}^{\gd}
&= (\varsigma^{-1} \varsigma)^* \left( \lap^{(k)} D_{\N}^*(\gG - \tgG)_{\gd}^{\gb} \right) \varsigma_{\ga}^{\gd}\\
&=- (\varsigma^{-1} \varsigma)^* \lap^{(k)} D^*_{\N} \left[  (D_{\varsigma^* \N} \varsigma)_{i \gamma }^{\gb} (\varsigma^{-1})_{\gd}^{\gamma} - (\sNG)_{i \gd}^{\gb} + \tgG_{i \gd}^{\gb} \right] \varsigma_{\ga}^{\gd} \\
&= - (\varsigma^{-1})^* (\slap)^{(k)} D_{\varsigma^* \N}^* \left[  (\varsigma^{-1})^{\gb}_{\gz} (D_{\varsigma^* \N} \varsigma)_{i \gd }^{\gz} \right]  \varsigma_{\ga}^{\gd} \\ & \hsp - (\varsigma^{-1})^* (\slap)^{(k)} D_{\varsigma^* \N}^* \left[ - (\varsigma^{-1})^{\gb}_{\gz} (\sNG)_{i \rho}^{\gz} \varsigma^{\rho}_{\gd}   + (\varsigma^{-1})^{\gb}_{\gz} \tgG_{i \rho}^{\gz} \varsigma^{\rho}_{\gd} \right] \varsigma_{\ga}^{\gd} \\
&= - \varsigma^{\gb}_{\tau} (\slap)^{(k)} D_{\varsigma^* \N}^* \left[  (\varsigma^{-1})^{\tau}_{\gz} (D_{\varsigma^* \N} \varsigma)_{\ga}^{\gz}  - \varsigma^*\left( (\sNG)_{i \ga}^{\tau}  - \tgG_{i \ga}^{\tau} \right)  \right] \\
&=  -\varsigma^{\gb}_{\tau} (\slap)^{(k)} D_{\varsigma^* \N}^* \left[  (\varsigma^{-1})^{\tau}_{\gz} (D_{\varsigma^* \N} \varsigma)_{\ga}^{\gz} \right]  + \varsigma^{\gb}_{\tau}\left[  (\slap)^{(k)} D_{\varsigma^* \N}^*  \left[  \varsigma^*\left( (\sNG)_{\ga}^{\tau}  - \tgG_{\ga}^{\tau} \right)  \right] \right]_{T_1}.
\end{align*}
Expanding the left side term
\begin{align*}
(\slap)^{(k)}& \left[D^*_{\varsigma^* \N} \left( (\varsigma^{-1})_{\zeta}^{\tau}(D_{\varsigma^* \N} \varsigma)\right) \right]_{\ga}^{\gz} \\ &=  - g^{ji} (\slap)^{(k)} (\varsigma^* \N)_j \left[(\varsigma^{-1})_{\zeta}^{\tau}(D_{\varsigma^* \N} \varsigma)_{i \ga}^{\gz} \right] \\
&= - g^{ij} (\slap)^{(k)} \left[ (\varsigma^* \N)_j \left[ \left(\varsigma^{-1} \right)_{\zeta}^{\tau} \right] (D_{\varsigma^*\N} \varsigma)_{i \ga}^{\zeta} + \left(\varsigma^{-1} \right)_{\zeta}^{\tau} (\varsigma^* \N)_j \left[ (\varsigma^*\N)(\varsigma)\right]_{i \ga}^{\zeta} \right]\\
&= - \left[g^{ij} (\slap)^{(k)} \left[ (\varsigma^* \N)_j \left[ \left(\varsigma^{-1} \right)_{\zeta}^{\tau} \right] (D_{\varsigma^*\N} \varsigma)_{i \ga}^{\zeta} \right] \right]_{T_2} -  g^{ij} (\slap)^{(k)} \left(\varsigma^{-1} \right)_{\zeta}^{\tau} (\varsigma^* \N)_j \left[ (\varsigma^*\N)(\varsigma)\right]_{i \ga}^{\zeta}. 
\end{align*}
We further expand the term on the right, with the intent of drawing out the $\vs^{-1}$ term. Using Lemma \ref{lem:slapswap} applied to $\varsigma^{-1} \in S(\Aut E)$ we have
\begin{align*}
(\slap)^{(k)} \left[ (\varsigma^{-1})_{\gz}^{\gb}(\slap) \left[\varsigma \right]_{\ga}^{\gz} \right] 
&= (\varsigma^{-1})_{\zeta}^{\gb}(\slap)^{(k +1)} [\varsigma]_{\ga}^{\zeta} + \left[ (\slap)^{(k)}[\varsigma^{-1}]_{\zeta}^{\gb} (\slap)[\varsigma]_{\ga}^{\zeta} \right]_{T_3} \\
& \hsp +  \left[\left( \prod_{\ell = 0}^{k} g^{i_{\ell} j_{\ell}}\right) \sum_{r = 1}^{k-1}{ \sum_{\mathsf{P} \in \mathcal{P}_r(L)}{\left( (\varsigma^*\N)_{\mathsf{P}} (\varsigma^{-1})_{\gz}^{\gb} \right)\left( (\varsigma^*\N)_{\mathsf{P}^c}(\slap)\left[ \varsigma \right]_{\ga}^{\gz} \right) }} \right]_{T_4}.
\end{align*}
Therefore
\begin{align*}
(\slap)^{(k)} \brk{D_{\N}^*(\gG - \tgG)_{\gd}^{\gb}} \varsigma_{\ga}^{\gd} &= - (\slap)^{(k+1)} \varsigma_{\ga}^{\gb} + ( \mathfrak{a}_k(\varsigma,\varsigma^*\N, \tN) )_{\ga}^{\gb},
\end{align*}
where
\begin{equation*}
\mathfrak{a}_k \prs{\varsigma, \varsigma^*\N , \tN}_{\ga}^{\gb} = \varsigma_{\tau}^{\gb} \left(T_1 + T_2  +  T_3 +  T_4 \right)_{\ga}^{\tau}.
\end{equation*}
Note that $\mathfrak{a}_k \prs{\varsigma, \varsigma^*\N, \tN }$ is lower order than $(\slap)^{(k+1)} \varsigma$. The result follows.
\end{proof}
\end{lemma}

We now demonstrate the short time existence and uniqueness of the flow.

\begin{prop}\label{prop:YMkstexist}
Let $(E,h) \to (M^n,g)$ be a vector bundle over a compact manifold. Given some metric compatible connection $\N_0$ on $E$, there exists some $\varepsilon > 0$ such that the generalized Yang-Mills $k$-system with initial condition $\N_0$ has a unique solution $\N_t$ for $t \in [0,\varepsilon)$.

\begin{proof} We first prove existence and then uniqueness. Let $\N_0 \in \mathcal{A}_E$ and consider the following two systems. First, a system of connections $\N_t \in \mathcal{A}_E \times \mathcal{I}$ given by
\begin{equation}\label{eq:Psiksystem}
\begin{cases}
\frac{\del \N_t}{\del t} &= \Psi_k(\N_t, \N_0) \\
\left. \N_t \right|_{t=0} &= \N_0.
\end{cases}
\end{equation}
Next, a system of gauge transformations $\vs_t \in S(End E) \times \mathcal{I}$ satisfying
\begin{equation}\label{eq:YMkssystem}
\begin{cases}
\frac{\del \varsigma_t}{\del t} &=(-1)^{k+1} \left( \lap^{(k)} D^*_t (\gG_t - \gG_0) \right) \varsigma_t\\
\varsigma_0 &= \Id.
\end{cases}
\end{equation}
\noindent \textbf{Existence.} Consider the flow $\N_t$ with initial condition $\N_0$. Since $\Psi_k$ is an elliptic operator by Lemma \ref{lem:Psikellptc}, a solution $\N_{t}$ to the parabolic system \eqref{eq:Psiksystem}, the $(\Psi,k)$-flow, exists on some $t \in [0,\epsilon)$ for $\epsilon > 0$. Choose the unique solution $\vs_t$ to the system \eqref{eq:YMkssystem} and consider $\vs_t^* \N_{t}$. This is a solution to the generalized Yang-Mills $k$-system with initial condition $\N_{0}$ as seen through the following computation  which utilizes \eqref{eq:vargauge} and \eqref{gaugebeamid},
\begin{align*}
\left( \frac{\del \vs^* \N_t}{\del t} \right)_{r \ga}^{\gb}&= (D_{\vs^*\N})_r (\vs^{-1} \dot{\vs})_{r \ga}^{\gb} + (\vs^{-1})_{\gd}^{\gb} (\dot{\Gamma}_{r \gt}^{\gd}) \vs_{\ga}^{\gt} \\
&= (-1)^{k+1}(D_{\vs^* \N})_r \left[ (\vs^{-1})_{\gd}^{\gb} (\slap)^{(k)}\left[ D^*_{\N} (\gG - \gG_0)_{\gz}^{\gd} \right] \vs_{\ga}^{\gz} \right]  + (-1)^{k+1}(\vs^{-1})_{\gd}^{\gb}\left( (D^*_{\N} (\slap)^{(k)} F_{\N})_{r \gz}^{\gd} \right)\\
&\hsp + (-1)^{k+1} \left( \mho_k(\N) \right)_{r \gz}^{\gd} - (D_{\N})_r (\slap)^{(k)}\left[ D^*_{\N} \Gamma - \Gamma_0)_{\gz}^{\gd}  \right] \vs_{\ga}^{\gz}\\
&= (-1)^{k+1} \left( D^*_{\vs^* \N} (\slap)^{(k)} F_{\vs^* \N}  \right)_{r \ga}^{\gb} +  \left(\mho_k(\vs^*\N)\right)_{r \ga}^{\gb}.
\end{align*}
Therefore $\vs_t^* \N_{t}$ is a solution to the generalized Yang-Mills $k$-system. The first result follows.

\noindent \textbf{Uniqueness.} Suppose that $\N_{0}$ is some connection with two solutions $\N_{t}$ and $\widetilde{\N}_{t}$ to the generalized Yang-Mills $k$-system. Let $\vr_t \in S(\Aut E) \times \mathcal{I}$ be the solution to the following system of gauge transformations:
\begin{align*}
\begin{cases}
 \frac{\del \vr_t}{\del t} &= (-1)^{k} (\lap_{\vr_t })^{(k+1)} \vr_t + (-1)^{k}  \mathfrak{a}_k(\vr,\N_t,\N_0)\\
\vr_0 &= \Id.
\end{cases}
\end{align*}
Similarly let $\widetilde{\vr}_t \in S(\Aut E) \times \mathcal{I}$ be the solution to the following:
\begin{align*}
\begin{cases}
\frac{\del \widetilde{\vr}_t}{\del t} &=  (-1)^{k} (\lap_{\widetilde{\vr}_t })^{(k+1)}\widetilde{\vr}_t + (-1)^{k}  \mathfrak{a}_k \prs{ \widetilde{\vr},\widetilde{\N}_t,\widetilde{\N}_0 } \\
\widetilde{\vr}_0 &= \Id.
\end{cases}
\end{align*}
These are strictly parabolic and lower order hence the solutions exists for all time. The next task is to verify that with the initial condition $\N_0$, the one-parameter family $(\vr_t^{-1})^* \N_{t}$ is a solution to \eqref{eq:Psiksystem}. Observe that by the equivalence demonstrated by Lemma \ref{lem:YMkslaplap},
\[ \left( \frac{\del \vr}{\del t} \right)_{\ga}^{\gb}=(-1)^{k+1} (\lap_{\vr^{-1} })^{(k)} \left[ D_{(\vr^{-1})^*\N}^* \left[ (\vr^{-1})^* \left[ \gG -\gG_0 \right] \right] \right]_{\gd}^{\gb} \vr_{\ga}^{\gd}. \]
With this in mind and utilizing the expression for the derivative of a gauge acting on a connection \eqref{eq:vargauge},
\begin{align*}
\left( \frac{\del \left((\vr^{-1})^* \N \right)}{\del t} \right)_{i \ga}^{\gb}
&= (D_{(\vr^{-1})^* \N})_i \left[ \vr_{\gd}^{\gb} \del_t(\vr^{-1})_{\ga}^{\gd} \right] + \vr_{\gd}^{\gb} \dot{\gG}_{i \gz}^{\gd} (\vr^{-1})^{\gz}_{\ga}\\
&=- (D_{(\vr^{-1})^* \N})_i \left[ \dot{\vr}_{\gd}^{\gb} (\vr^{-1})^{\gd}_{\ga} \right] + \vr_{\gd}^{\gb} \left( (-1)^{k+1} D^*_{\N} (\lap)^{(k)} F_{\N} + 
\mho_k(\N)
\right)_{i \gz}^{\gd} (\vr^{-1})^{\gz}_{\ga}\\
& =(-1)^{k} (D_{(\vr^{-1})^* \N})_i \left( (\lap_{\vr^{-1}})^{(k)} \left[ D_{(\vr^{-1})^*\N}^* \left[ (\vr^{-1})^* \left[ \gG -\gG_0 \right]\right] \right]^{\gb}_{\ga} \right) \\
& \hsp + (-1)^{k+1} D_{(\vr^{-1})^*\N}^* (\lap_{\vr^{-1}})^{(k)} \left[ F_{(\vr^{-1})^*\N} \right]_{i \ga}^{\gb} + \left( \mho_k ((\vr^{-1})^*\N) \right)_{i \ga}^{\gb}.
\end{align*}
This is precisely $\Psi_k((\vr^{-1}_t)^* \N_t, \N_0)$. The computation could be done identically with $(\widetilde{\vr}_t^{-1})^* \widetilde{\N}_t$ instead, giving that $ (\widetilde{\vr}_t^{-1})^* \widetilde{\N}_t$ and $ (\vr^{-1}_t)^* \N_t$ are both solutions to system \eqref{eq:Psiksystem} with the same initial condition. Since solutions to system \eqref{eq:Psiksystem} are unique, $(\vr^{-1}_t)^* \N_t = (\widetilde{\vr}^{-1}_t)^* \widetilde{\N}_t$. Hence $\vr_t$ and $\widetilde{\vr}_t$ must satisfy system \eqref{eq:YMkssystem}, but since this is a linear ordinary differential equation \eqref{eq:YMkssystem} on a compact manifold with no boundary, $\vr_t \equiv \widetilde{\vr}_t$, which implies that $\N_t \equiv \widetilde{\N}_t$. Therefore uniqueness follows, and the proof is complete.
\end{proof}
\end{prop}

\subsection{Smoothing estimates}\label{ss:smoothingest}

In this section our goal is to compute, assuming a supremal bound on $|F_{\N_t}|$, the associated local bounds on the $L^2$ norms of covariant derivatives of $F_{\N_t}$. To accomplish this we first compute necessary variational identities.

\begin{lemma}\label{lem:YMkgradellF}
Suppose $\N_t \in \mathcal{A}_E \times \mathcal{I}$ is solution to generalized Yang-Mills $k$-flow. For $\ell \in \mathbb{N}$ the following evolution equation holds,
\begin{align}\label{eq:YMkgradell}
\begin{split}
\frac{\del }{\del t} \lb \N^{(\ell)}_t F_{\N_t} \rb
&= 
(-1)^k \lap^{(k+1)}_t \lb \N^{(\ell)}_tF_{\N_t} \rb+ P_2^{(\ell + 2k)}\lb F_{\N_t} \rb +  \sum_{s=0}^{\ell + 2k}P^{(s)}_1[F_{\N_t}] \\
&\hsp +   \sum_{i=1}^{k} \sum_{j=2}^{i + 1}{P_j^{(\ell + 2
i + 4 -2j)} \lb F_{\N_t} \rb } +  \sum_{i=1}^{k} \sum_{j=2}^{i + 1}{ P_{j+1}^{(\ell + 2
i + 2 -2j )} \lb F_{\N_t} \rb },
\end{split}
\end{align}
and for $\ell=0$,
\begin{equation}\label{eq:YMkgrad0}
\frac{\del F_{\N_t}}{\del t}  =  (-1)^{k+1}\lap^{(k+1)}_t F_{\N_t} +  P^{(2k)}_1 \lb F_{\N_t} \rb + P_{2}^{(2k)} \lb F_{\N_t} \rb + \sum_{i=1}^{k} \sum_{j=2}^{i + 1}{P_j^{(2
i + 4 -2j)} \lb F_{\N_t} \rb }.
\end{equation}

\begin{proof}
To vary $F_{\N_t}$, we differentiate and then apply the Bochner formula (Proposition \ref{prop:bochner}) to obtain
\begin{align*}
\frac{\del F}{\del t}
&= D \dot{\gG}\\
&= D((-1)^{k+1} D^* \lap^{(k)} F + \mho_k(\N))\\
&= (-1)^{k+1} \lap_D \lap^{(k)} F + D(\mho_k(\N))\\
&= (-1)^{k+1} \lap^{(k+1)} F + (-1)^{k+1} (\Rm + F)\ast (\lap^{(k)} F)+  D(\mho_k(\N))\\
&= (-1)^{k+1}\lap^{(k+1)} F +  P^{(2k)}_1 \lb F \rb + P_{2}^{(2k)} \lb F \rb + \left(  \sum_{i=1}^{k} \sum_{j=2}^{i + 1}{P_j^{2i + 4 -2j} \lb F \rb}  \right).
\end{align*}
To vary $\N^{(\ell)}_t F_{\N_t}$ we apply \eqref{cor:varNellF} and then insert the equation of generalized Yang-Mills $k$-flow,
\begin{align*}
\frac{\del \N^{(\ell)}  F}{\del t}
&=\N^{(\ell)}D\dot{\gG} +  \sum_{i=0}^{\ell-1}{\left( \N^{(i)} \dot{\gG} \ast \N^{(\ell-i-1)}F \right)}\\
&=\left[ \N^{(\ell)}D \left( (-1)^{k+1} D^* \lap^{(k)} F + \mho_k(\N) \right) \right]_{T_1}   \\
& \hsp +  \left[ \sum_{i=0}^{\ell-1}{\left( \N^{(i)} \left( (-1)^{k+1} D^* \lap^{(k)} F + \mho_k(\N) \right)  \ast \N^{(\ell-i-1)}F \right)} \right]_{T_2}.
\end{align*}
We manipulate $T_1$ first. Using the Bochner formula (Proposition \ref{prop:bochner}) to decompose the first quantity yields
\begin{align*}
T_1 &= (-1)^{k+1}\N^{(\ell)} \lap_D \lap^{(k)} F + \N^{(\ell)} D \mho_k(\N) \\
&=  \N^{(\ell)}\lb \lap^{(k+1)} F \rb + \N^{(\ell)}\lb (\Rm+F) \ast \N^{(2k)} F \rb +  \N^{\ell} D \lb \mho_k(\N) \rb \\
& =  \N^{(\ell)}\lb \lap^{(k+1)} F \rb + \sum_{q =0}^{\ell} P^{(2k+q)}_1 \lb F \rb + P_2^{(2k+\ell)} \lb F \rb +   \sum_{i=1}^{k} \sum_{j=2}^{i + 1}{P_j^{(\ell + 2i + 4 -2j)} \lb F \rb}.
\end{align*}
Using Corollary \ref{cor:Nllapkcomut} yields
\begin{align*}
\N^{(\ell)} \lb \lap^{(k+1)} F \rb &= (-1)^k\lap^{(k+1)}  \N^{(\ell)} F +  \sum_{v=0}^{\ell - 1} \sum_{j=0}^{2k+1}{ \left(\N^{(v+j)}(\Rm + F) \ast (\N^{(\ell - v +2k -j)} F) \right)}\\
&= (-1)^k \lap^{(k+1)}F \lb \N^{(\ell)}F \rb+ \sum_{s=0}^{\ell + 2k} P^{(s)}_1 \lb F \rb + P_2^{(\ell + 2k)}\lb F \rb.
\end{align*}
Which gives that
\begin{align*}
T_1 & =  (-1)^k \lap^{(k+1)} \lb \N^{(\ell)}F \rb+  \sum_{s=0}^{\ell + 2k}P^{(s)}_1[F] + P_2^{(\ell + 2k)}\lb F \rb + \sum_{q =0}^{\ell} P_1^{(2k+q)}[F]\\
&\hsp +   \sum_{i=1}^{k} \sum_{j=2}^{i + 1}{P_j^{(
\ell + 2i + 4 -2j)} \lb F \rb } .
\end{align*}
Next we manipulate $T_2$,
\begin{align*}
T_2 &= \sum_{i=0}^{\ell-1}{\left( \N^{(i)} \left( (-1)^{k+1} D^* \lap^{(k)} F + \mho_k(\N) \right)  \ast \N^{(\ell-i-1)}F \right)} \\
&= \sum_{i=0}^{\ell - 1}\left( (\N^{(i + 2k + 1)} F) \ast (\N^{(\ell - i -1)} F) + \sum_{v=1}^{k} \sum_{j=2}^{v + 1}{P_j^{(2
v + 3 -2j + i)} \lb F \rb}  \ast (\N^{(\ell - i -1)} F)  \right)\\
&= P_{2}^{(2k + \ell)}\lb F \rb +  \sum_{v=1}^{k} \sum_{j=2}^{v + 1}{ \left( P_{j+1}^{(\ell + 2
v + 2 -2j )} \lb F \rb \right)}.
\end{align*}
Combining $T_1$ and $T_2$ yields
\begin{align*}  
\frac{\del \N^{(\ell)} F }{\del t} &= 
(-1)^k \lap^{(k+1)} \lb \N^{(\ell)} F \rb+ P_2^{(\ell + 2k)}\lb F \rb + \sum_{s=0}^{\ell + 2k}P^{(s)}_1[F] \\
&\hsp + \sum_{i=1}^{k} \sum_{j=2}^{i + 1}{P_j^{(\ell + 2i + 4 -2j)} \lb F \rb} +  \sum_{i=1}^{k} \sum_{j=2}^{i + 1}{\left( P_{j+1}^{(\ell + 2i + 2 -2j )} \lb F \rb \right)} .
\end{align*}
The result follows.
\end{proof}
\end{lemma}

We next begin the discussion of our local smoothing estimates. While the inclusion of a bump function forces the computations to be significantly more involved, they are highly necessary. During the blowup analysis in Proposition \ref{prop:YMkblowup}, while working within a coordinate chart, we will require these local estimates to address that the domains of the connections under consideration are restricted to open subsets of $\mathbb{R}^n$.

\begin{defn}[Bump function set and bounding quantity]\label{defn:bumpset}
Let $\mathcal{B} := \left\{ \eta \in C^{\infty}_c(M): 0 \leq \eta \leq 1 \right\}$, that is, the family of `bump' functions. Let $\ell \in \mathbb{N}$ and set, for a given $\N \in \mathcal{A}_E$,
\begin{equation*}
\jmath_{\eta}^{(\ell)} :=\sum_{q =0}^{\ell} \brs{\brs{ \N^{(q)} \eta }}_{L^{\infty}(M)} .
\end{equation*}
\end{defn}
We first prove the following lemma, which will be essential in the manipulations of Lemma \ref{lem:YMketavpest}. This is a technical result demonstrating how to shift derivatives within the integrands terms which will be commonly featured. This relies primarily on integration by parts while taking the interaction of the bump function into account.

\begin{lemma}\label{lem:Pbalance} Let $p,q, r, s \in \mathbb{N}$, $\N \in \mathcal{A}_E$ and $\eta \in \mathcal{B}$. Then if $s \in \mathbb{N}\backslash \{ 1 \}$,
\begin{align}
\begin{split}\label{eq:Pbalance}
\int_M \left( P_1^{(p)} \lb F_{\N} \rb \ast P^{(q + r)}_1 \lb F_{\N} \rb \right) \eta^s dV_g 
&\leq  \int_M \left( P_1^{(p+r)} \lb F \rb \ast P_1^{(q)} \lb F \rb \right) \eta^s dV_g \\
& \hsp + \sum_{j=0}^{r-1} \jmath_{\eta}^{(1)} \int_M \left( P_1^{(p + j)} \lb F \rb \ast P_1^{(q + (r -1 - j) )} \lb F \rb \right)  \eta^{s-1} dV_g.
\end{split}
\end{align}

\begin{proof}
The proof follows by induction on $r$. For the base case, observe that by integration by parts,
\begin{align*}
\int_M{\left( P_1^{(p)} \lb F \rb \ast P^{(q + 1)}_1 \lb F \rb \right) \eta^s dV_g}
&\leq \int_M \left( P_1^{(p+1)} \lb F \rb \ast P_1^{(q)} \lb F \rb \right) \eta^s dV_g\\
& \hsp  + \jmath_{\eta}^{(1)} \int_M \left( P_1^{(p)} \lb F \rb \ast P_1^{(q)} \lb F \rb \right)  \eta^{s-1} dV_g.
\end{align*}
The base case follows, now we assume the induction hypothesis \eqref{eq:Pbalance} holds for $r$. Then by instead applying the identity with $p$ replaced by $p+1$ and $q$,
\begin{align*}
\int_M \left( P_1^{(p)} \lb F \rb \ast P_1^{(q+r+1)} \lb F \rb \right) \eta^s dV_g
&\leq \int_M \left( P_1^{(p+1)} \lb F \rb \ast P_1^{(q + r)} \lb F \rb \right) \eta^s dV_g\\
& \hsp + \jmath_{\eta}^{(1)} \int_M \left( P_1^{(p)} \lb F \rb \ast P_1^{(q + r)} \lb F \rb \right)  \eta^{s-1} dV_g \\
&= \int_M \left( P_1^{(p+(r+1))} \lb F \rb \ast P_1^{(q)} \lb F \rb \right) \eta^s dV_g \\
& \hsp +  \sum_{j=0}^{r} \jmath_{\eta}^{(1)} \int_M \left( P_1^{(p + j)} \lb F \rb \ast P_1^{(q + (r - j) )} \lb F \rb \right)  \eta^{s-1} dV_g.
\end{align*}
The result follows.
\end{proof}
\end{lemma}

\begin{lemma}\label{lem:YMketavpest}
Let $\ell \in \mathbb{N}$, $\eta \in \mathcal{B}$, and suppose $\N_t \in \mathcal{A}_E \times \mathcal{I}$ a solution to generalized Yang-Mills $k$-flow with $\sup_{M \times \mathcal{I}} \brs{F_{\N_t}} < \infty$. Set $\vp_t := \N^{(\ell)}_t F_{\N_t} $ and choose $K > \max\left\{ \sup_{M \times \mathcal{I}} \brs{F_{\N_t}}, 1 \right\}$. Then for $s \geq 2( k + \ell + 1)$ there exists $C: = C \prs{ \dim M, \rank E, k, s,\ell, g, \jmath_{\eta}^{(s)}}$ such that
\begin{equation*}
\tfrac{d}{d t} || \eta^{s/2} \vp_t ||_{L^2(M)}^2 \leq -   || \eta^{s/2} (\N^{(k+1)}_t \vp_t) ||_{L^2(M)}^2  + CK^{2k+2} || F_{\N_t} ||_{L^2(M), \eta>0}^2.
\end{equation*}

\begin{proof}
We differentiate $|| \eta^{s/2} \vp_t ||_{L^2}^2$ using the variation computation in Lemma \ref{lem:YMkgradellF},
\begin{align*}
\frac{d}{d t} \lb  \int_M{\eta^s | \vp |^2 dV_g} \rb
&= \int_M{2 \left\langle \tfrac{\del}{\del t} \lb \vp \rb, \eta^s \vp \right\rangle dV_g} \\
& =  \left[ (-1)^k 2 \int_M \left\langle \lap^{(k+1)} \vp, \eta^s \vp \right\rangle dV_g \right]_{T_1}+\left[ \int_M \left\langle  P_2^{(\ell + 2k)} \lb F \rb , \eta^s \vp \right\rangle dV_g \right]_{T_2} \\
& \hsp  + \left[  \sum_{q=0}^{\ell + 2k} \int_M \langle P^{(q)}_1[F], \eta^s \vp \rangle dV_g \right]_{T_3}\\
& \hsp + \left[ \sum_{i=1}^{k} \sum_{j=2}^{i+ 1}{ \int_M  \left\langle P_j^{(\ell + 2
i + 4 -2j)} \lb F \rb  , \eta^s \vp \right\rangle dV_g } \right]_{T_4}\\
& \hsp + \left[ \sum_{i=1}^{k} \sum_{j=2}^{i + 1}{\left\langle  \int_M P_{j+1}^{(\ell + 2
i + 2 -2j )} \lb F \rb  , \eta^s \vp \right\rangle dV_g} \right]_{T_5}.
\end{align*}
We address each labelled term separately. Note that the analysis of the constraint on $s$ contributed by each term requires two main considerations. Let $\ga, \gb , \gz, r \in \mathbb{N}$ and $(i_j)_{j=1}^r$ be some multiindex. First, an application of Corollary \ref{cor:KS5.3interp} requires that, if applied to $\brs{\brs{ \eta^{\ga/2} \N^{(\gb)} F }}$, we must have $\ga \geq 2 \gb$. The application of Lemma \ref{lem:shiftkNs} requires that to estimate $\int_M \eta^{\ga} \N^{(i_1)} F \cdots \N^{(i_r)} F dV_g$ with $\sum_{j=1}^r i_j = 2\gz$, then $\ga \geq 2\gz$. The constant $C \in \mathbb{R}_{>0}$ to appear in the following manipulations will be updated, increasing through computations. \newline

\noindent \fbox{\textbf{$T_1$ estimate.}} We manipulate $T_1$ using Lemma \ref{lem:shiftkNs} to shift $\N$ across the inner product.
\begin{align*}
T_1
&= (-1)^k 2 \int_M{ \left\langle \lap^{(k+1)} \vp, \eta^s \vp \right\rangle dV_g}\\
&= - 2 \int_M{\left\langle \N^{(k+1)} \vp , \N^{(k+1)} \lb \eta^s \vp \rb \right\rangle dV_g} + \left\langle  \sum_{q = 1}^{2k-2} \sum_{w = 0}^{q} \left(\N^{(w)} (\Rm + F) \ast \N^{(2k-2-w)} \vp  \right), \vp \right\rangle \\
&=  \left[ -2 \int_M{\left\langle \N^{(k+1)} \vp , \N^{(k+1)} \lb \eta^s \vp \rb \right\rangle dV_g} \right]_{T_{11}} + \left[ \int_M \left( P^{(2k - 2 + 2 \ell)}_3 \lb F \rb \right) \eta^s dV_g \right]_{T_{12}}\\
& \hsp + \left[ \sum_{q=0}^{2 \ell -2}  \int_M \left(  P_2^{(2k+q)} \lb F \rb \right) \eta^s dV_g \right]_{T_{13}}.
\end{align*}
We address each term above separately. For $T_{11}$, we differentiate, resulting in a summation, draw out one term to `absorb' the others and address the rest of the indices.
\begin{align*}
T_{11} &= -2 \sum_{j=0}^{k+1} \int_M { \N^{(j)} \brk{\eta^{s}}\ast \left\langle \N^{(k+1-j)} \vp, \N^{(k+1)}\vp \right\rangle dV_g} \\
& \leq -2 \| \eta^{s/2} \N^{(k+1)} \vp \|_{L^2}^2 + \sum_{j=1}^{k+1} \prs{ \jmath^{(k+1)}_{\eta} \int_M{ \eta^{s-j}\left\langle \N^{(k+1-j)} \vp, \N^{(k+1)}\vp \right\rangle dV_g} }.
\end{align*}
We manipulate each term in the summation on the right, first by a weighted H\"{o}lder's inequality and then applying Corollary \ref{cor:KS5.3interp},
\begin{align*}
\jmath^{(k+1)}_{\eta} \int_M & { \eta^{s-j}\left\langle \N^{(k+1-j)} \vp, \N^{(k+1)}\vp \right\rangle dV_g} \\
& \leq \epsilon \brs{ \brs{ \eta^{s/2} \N^{(k+1) }\vp }}_{L^2}^2+ C\brs{ \brs{ \eta^{(s-2j)/2} \N^{(k+1 - j)}\vp }}_{L^2}^2\\
& = \epsilon \brs{ \brs{ \eta^{s/2} \N^{(k+1) }\vp }}_{L^2}^2+ C\brs{ \brs{ \eta^{(s-2j)/2} \N^{(k+1+ \ell - j)}F }}_{L^2}^2\\
&\leq \epsilon \brs{ \brs{ \eta^{s/2} \N^{(k+1) }\vp }}_{L^2}^2+ \epsilon \brs{ \brs{ \eta^{(s-2j + 2j)/2} \N^{(k+1+ \ell)}F }}_{L^2}^2 
+ C \brs{ \brs{F}}_{L^2, \eta>0}^2 \\
& \leq 2\epsilon\brs{ \brs{ \eta^{s/2} \N^{(k+1)} \vp  }}_{L^2}^2 + C \brs{ \brs{ F }}_{L^2, \eta>0}^2.
\end{align*}
Therefore we conclude, by summing over all terms,
\begin{align*}
T_{11}
&\leq -2 \brs{ \brs{ \eta^{s/2} \N^{(k+1)} \vp}}_{L^2}^2 + (k+1) \prs{ 2 \epsilon \brs{\brs{\eta^{s/2} \N^{(k+1)} \vp}} + C \brs{\brs{ F } }_{L^2, \eta > 0}^2 } \\
&\leq \prs{-2 + 2 \epsilon (k+1) } \brs{ \brs{ \eta^{s/2} \N^{(k+1)} \vp}}_{L^2}^2 + C \brs{\brs{ F } }_{L^2, \eta > 0}^2.
\end{align*}

\noindent \textbf{$\mathbf{T_{11}}$ bump function constraints.} Now we analyze the maximum power of the bump  function in this setting. Corollary \ref{cor:KS5.3interp} requires that $s - 2j > 2 (k+1+ \ell-j) $, namely $s > 2(k+1 + \ell)$. There are no other constraints on the bump function. \newline

Next we estimate $T_{12}$ by applying Lemma \ref{lem:kuwert5.5} and then Corollary \ref{cor:KS5.3interp}.
\begin{align*}
T_{12} & \leq \int_M \left( P^{(2k - 2 + 2 \ell)}_3 \lb F \rb \right) \eta^s dV_g \\
& \leq Q_{(3,k-1+\ell)} K \left( || \eta^{s/2} \N^{(k-1)} \vp ||_{L^2}^2  + || F ||_{L^2,\eta>0}^2 \right) \\
& = C K \left( || \eta^{s/2} \N^{(k + \ell + 1 -2)} F ||_{L^2}^2  + || F ||_{L^2,\eta>0}^2 \right) \\
& \leq \epsilon  || \eta^{(s + 2)/2} \N^{(k+1)} \vp ||_{L^2}^2  + C K^2 || F ||_{L^2,\eta>0}^2 \\
& \leq \epsilon  || \eta^{s/2} \N^{(k+1)} \vp ||_{L^2}^2  + C K^2 || F ||_{L^2,\eta>0}^2.
\end{align*}
\noindent \textbf{$\mathbf{T_{12}}$ bump function constraints.} The application of Lemma \ref{lem:kuwert5.5} and  Corollary \ref{cor:KS5.3interp} require that $s \geq 2(k-1+\ell)$, giving the restraint here. \newline

Next we estimate $T_{13}$. We divide up the summation into cases when the index $q$ is either odd or even and apply Lemma \ref{lem:Pbalance} to `balance out' the order of the connection application across terms,
\begin{align*}
T_{13}
 &=  \sum_{q: q \in 2 \mathbb{N} \cup \{ 0 \} }^{2 \ell -2}  \int_M \left( P_2^{(2k+q)} \lb F \rb  \right) \eta^s dV_g + \sum_{q: q \in 2 \mathbb{N} -1 }^{2 \ell -3}  \int_M \left( P_2^{(2k+q)} \lb F \rb \right) \eta^s dV_g \\
 &= \brk{ \sum_{q: q \in 2 \mathbb{N} \cup \{ 0 \} }^{2 \ell -2}  \int_M \left( P_2^{(2k+q)} \lb F \rb  \right) \jmath_{\eta}^{(1)} \eta^{s-1} dV_g }_{T_{13,E}} \\
 & \hsp + \brk{\sum_{q: q \in 2 \mathbb{N} -1 }^{2 \ell -3}  \int_M \left( P_1^{\lceil \frac{2k+q}{2} \rceil} \lb F \rb \ast P_1^{\lfloor \frac{2k+q}{2} \rfloor} \lb F \rb  \right) \eta^s dV_g}_{T_{13,O}}.
\end{align*}
For each index $q$ of $T_{13,E}$ we apply Lemma \ref{lem:kuwert5.5} and then Corollary \ref{cor:KS5.3interp}, noting that we maximize $q$ at $2 \ell - 2$ to obtain the final line,
\begin{align*}
 \int_{M} \left( P_2^{(2k+q)} \lb F \rb \right) \jmath_{\eta}^{(1)} \eta^{s-1} dV_g
 &\leq Q_{(2, k + \frac{q}{2})}\left( || \eta^{(s-1)/2} \N^{\left( k + \frac{q}{2} \right)} F ||^2_{L^2} + || F ||_{L^2, \eta > 0}^2\right) \\
&= C \left( || \eta^{(s-1)/2} \N^{\left( k + \ell + 1 - \prs{\ell - \tfrac{q}{2} } \right)} \vp ||^2_{L^2} + || F ||_{L^2, \eta > 0}^2\right) \\
& \leq  \epsilon \brs{\brs{ \eta^{\frac{s -1 + 2 \ell - q}{2}} \N^{(k+1)} \vp }}_{L^2}^2+ C || F ||_{L^2, \eta > 0}^2 \\
& \leq  \epsilon \brs{\brs{ \eta^{\frac{s -1 + 2 \ell - (2 \ell - 2)}{2}} \N^{(k+1)} \vp }}_{L^2}^2+ C || F ||_{L^2, \eta > 0}^2 \\
& =  \epsilon \brs{\brs{ \eta^{\frac{s +1 }{2}} \N^{(k+1)} \vp }}_{L^2}^2+ C || F ||_{L^2, \eta > 0}^2 \\
& \leq  \epsilon \brs{\brs{ \eta^{\frac{s }{2}} \N^{(k+1)} \vp }}^2_{L^2}+ C || F ||_{L^2, \eta > 0}^2.
\end{align*}
Therefore we conclude that
\begin{align*}
T_{13,E} & \leq  \ell  \epsilon\prs{ \brs{\brs{ \eta^{s/2} \N^{(k+1)} \vp}}_{L^2}^2 } + C \brs{\brs{F}}_{L^2,\eta >0}^2.
\end{align*}
\noindent \textbf{$\mathbf{T_{13,E}}$ bump function constraints.} The application of Lemma \ref{lem:kuwert5.5} and Corollary \ref{cor:KS5.3interp} require that $s-1 \geq 2k + q$, which is at worst when $q$ is maximized ($q = 2 \ell - 2$). Thus $s \geq 2(k + \ell) - 1$. \newline

Next we address $T_{13,O}$. For each term in the summation we apply H\"{o}lder's inequality followed by an application of Lemma \ref{lem:kuwert5.5} to each term, then lastly and application of Corollary \ref{cor:KS5.3interp}.
\begin{align*}
\int_M & \left( P_1^{\lceil \frac{2k+q}{2} \rceil} \lb F \rb \ast P_1^{\lfloor \frac{2k+q}{2} \rfloor} \lb F \rb  \right) \eta^s dV_g \\
&\leq \tfrac{1}{2} \int_M P_2^{\prs{2 \lceil \frac{2k+q}{2} \rceil}} \lb F \rb \eta^s dV_g + \tfrac{1}{2} \int_M P_2^{\prs{2 \lfloor \frac{2k+q}{2} \rfloor}} \lb F \rb \eta^s dV_g \\
&\leq \tfrac{1}{2} Q_{(2, \lfloor \frac{2k+q}{2} \rfloor)} \prs{ \brs{\brs{ \eta^{s/2} \N^{\prs{\lfloor \frac{2k+q}{2} \rfloor}} F}}_{L^2}^2 + \brs{\brs{F}}_{L^2,\eta>0}^2} \\
& \hsp + \tfrac{1}{2} Q_{\prs{2, \lceil \frac{2k+q}{2} \rceil }}\prs{ \brs{\brs{ \eta^{s/2} \N^{\prs{ \lceil \frac{2k+q}{2} \rceil}} F}}_{L^2}^2 + \brs{\brs{F}}_{L^2,\eta>0}^2}\\
&= C \prs{ \brs{\brs{ \eta^{s/2} \N^{\prs{  k + 1 + \ell - \prs{ \ell - \lfloor \frac{q}{2} \rfloor } }} F}}_{L^2}^2 + \brs{\brs{ \eta^{s/2} \N^{\prs{ k + 1 + \ell  - \prs{ \ell - \lceil\frac{q}{2} \rceil}}} F}}_{L^2}^2}\\
& \hsp + C \brs{\brs{F}}_{L^2,\eta>0}^2 \\
& \leq \epsilon \brs{ \brs{ \eta^{\frac{s + \prs{ \ell - \lfloor \frac{q}{2} \rfloor } }{2}} \N^{(k+1)} \vp }}_{L^2}^2 + \epsilon \brs{ \brs{ \eta^{\frac{s + \prs{ \ell - \lceil\frac{q}{2} \rceil} }{2}} \N^{(k+1)} \vp }}_{L^2}^2 + C \brs{ \brs{F}}_{L^2, \eta > 0}^2 \\
& \leq \epsilon \brs{ \brs{ \eta^{\frac{s +\ell - \prs{\ell -  \lceil \frac{3}{2} \rceil} }{2}} \N^{(k+1)} \vp }}_{L^2}^2 + \epsilon \brs{ \brs{ \eta^{\frac{s + \ell - \prs{\ell -  \lfloor \frac{3}{2} \rfloor} }{2}} \N^{(k+1)} \vp }}_{L^2}^2 + C \brs{ \brs{F}}_{L^2, \eta > 0}^2 \\
& \leq  2 \epsilon \brs{\brs{ \eta^{s/2} \N^{(k+1)} F }}_{L^2}^2 + C \brs{ \brs{F}}_{L^2, \eta > 0}^2.
\end{align*}
\noindent \textbf{$\mathbf{T_{13,O}}$ bump function constraints.} The applications of both Lemma \ref{lem:kuwert5.5} and Corollary \ref{cor:KS5.3interp} require that $s \geq 2 \prs{ k + \lceil \tfrac{q}{2} \rceil }$. We note that $q \leq 2 \ell - 3$, so we conclude that $s \geq 2 \prs{ k + \ell - \lfloor  \tfrac{3}{2} \rfloor } =  2 \prs{ k + \ell - 1 }$. 

Therefore we conclude that
\begin{equation*}
T_1 \leq \prs{ -2 + \epsilon \prs{5 + \ell}} \brs{ \brs{ \eta^{s/2} \N^{(k+1)} \vp}}_{L^2}^2 + C K^2 \brs{ \brs{ F }}_{L^2, \eta > 0}^2.
\end{equation*}
\noindent \textbf{$\mathbf{T_1}$ bump function constraints.} Based on the above computations we conclude the cumulative constraint across all subterms $T_{1i}$ that $s \geq 2 \prs{ k + \ell + 1}$. \newline

\noindent \fbox{\textbf{$\mathbf{T_2}$ estimate.}} For $T_2$ we apply Lemma \ref{lem:kuwert5.5} and then Corollary  \ref{cor:KS5.3interp},
\begin{align*}
T_2
&= \int_{M}\left\langle P_2^{(\ell + 2k)} \lb F \rb, \eta^s \vp \right\rangle dV_g\\
&= \int_{M} \left(P_3^{(2 \ell + 2k)} \lb F \rb  \right)  \eta^s  dV_g\\
& \leq Q_{(3, \ell + k)} K \left( \left\| \eta^{s/2} \left( \N^{(k + \ell)} F \right) \right\|_{L^2}^2 + \| F \|_{L^2, \eta>0}^2 \right) \\
& = C K \left( \left\| \eta^{s/2} \left( \N^{(k + 1+ \ell - 1)} F \right) \right\|_{L^2}^2 + \| F \|_{L^2, \eta>0}^2 \right) \\
& \leq \epsilon \left\| \eta^{(s+1)/2} \N^{(k+1)} \vp \right\| + C K^2 \| F \|_{L^2, \eta>0}^2 \\
& \leq \epsilon \left\| \eta^{s/2} \N^{(k+1)} \vp \right\| + C K^2 \| F \|_{L^2, \eta>0}^2.
\end{align*}

\noindent \textbf{$\mathbf{T_2}$ bump function constraints.} For the application of Lemma \ref{lem:kuwert5.5} and Corollary  \ref{cor:KS5.3interp} we required $s \geq 2(\ell + k)$. \newline

\noindent \fbox{\textbf{$\mathbf{T_3}$ estimate.}} For $T_3$ we divide up terms between an even and odd number of derivatives and have, via integration by parts and collecting up derivatives of $\eta$ accordingly, noting that $s>k+\ell \geq \left\lfloor \frac{q+\ell}{2} \right\rfloor$,
\begin{align*}
T_{3} &= \sum_{q=0}^{\ell + 2k} \int_M \left\langle P^{(q)}_1[F], \eta^s \vp \right\rangle dV_g \\
 &=   \sum_{q=0}^{\ell + 2k} \int_M  \left( P^{(q+\ell)}_2[F] \right) \eta^s dV_g \\
 &= \sum_{q:q+\ell \in 2 \mathbb{N}}^{2k+ \ell}  \int_M \left( P^{(q+\ell)}_2[F]  \right)  \eta^s dV_g   + \sum_{q:q+\ell \in 2 \mathbb{N} -1 }^{2k+ \ell - 1}  \int_M \left( P^{(q+\ell)}_2[F] \right) \eta^s dV_g \\
 &\leq \brk{\sum_{q:q+\ell \in 2 \mathbb{N}}^{2k+ \ell}  \jmath_{\eta}^{(1)} \int_M  \left( P^{(q+\ell)}_2[F] \right) \eta^{s-1} dV_g}_{T_{3,E}}  \\
 & \hsp  + \brk{ \sum_{q:q+\ell \in 2 \mathbb{N} -1 }^{2k+ \ell - 1}  \int_M \left( P^{\left( \lceil \frac{q+\ell}{2} \rceil \right)}_1[F] \ast P^{\left( \lfloor \frac{q+\ell}{2} \rfloor \right)}_1[F]\right)  \eta^s dV_g }_{T_{3,O}}.
\end{align*}
For $T_{3,E}$ we have that by Lemma \ref{lem:kuwert5.5} and then Corollary \ref{cor:KS5.3interp}, noting that since $q \leq 2k+\ell$, then we have that $ k+ \tfrac{\ell - q}{2} \geq 0$, each term of the summation becomes
\begin{align*}
\jmath_{\eta}^{(1)} \int_M{ \left( P_2^{(q+\ell)}[F] \right) \eta^{s-1} dV_g}
 &\leq Q_{\left(2,\frac{q+\ell}{2} \right)} \left( || \eta^{(s-1)/2}\N^{\left(\frac{q+\ell}{2} \right)} F ||_{L^2}^2 + || F ||_{L^2, \eta > 0} \right) \\
  &= C \left( || \eta^{(s-1)/2}\N^{\left(k + \ell + 1 - \prs{ k + 1 + \frac{\ell -q}{2}} \right)} F ||_{L^2}^2 + || F ||_{L^2, \eta > 0}^2 \right) \\
    &\leq  \epsilon || \eta^{(s +  k + \frac{\ell -q}{2} )/2}\N^{\left(k  + 1\right)} \vp ||_{L^2}^2 + C || F ||_{L^2, \eta > 0}^2  \\
 & \leq \epsilon || \eta^{s/2} \N^{\left(k+1 \right)} \vp ||_{L^2}^2 + C || F ||_{L^2, \eta > 0}^2.
\end{align*}
\noindent \textbf{$\mathbf{T_{3,E}}$ bump function constraints.} The applications of Lemma \ref{lem:kuwert5.5} and Corollary \ref{cor:KS5.3interp} required $s \geq q + \ell + 1$. Maximizing the right side of the inequality with respect to $q$ we conclude that $s \geq 2\prs{k + \ell} + 1$.
For the second term we manipulate with H\"{o}lder's inequality, apply Lemma \ref{lem:kuwert5.5} and then Corollary \ref{cor:KS5.3interp}, noting that since $q \leq 2k+\ell-1$, then we have that $\lceil \frac{q+\ell}{2} \rceil \leq k+\ell$, so
\begin{align*}
 \int_M  \left( P^{\left( \lceil \frac{q+\ell}{2} \rceil \right)}_1[F] \ast P^{\left( \lfloor \frac{q+\ell}{2} \rfloor \right)}_1[F] \right)  \eta^s dV_g
& \leq   \int_{M}\left( P^{\left(2 \lceil \frac{q+\ell}{2} \rceil \right)}_2[F] \right) \eta^s   dV_g +  \int_{M}\left(  P^{\left(2 \lfloor \frac{q+\ell}{2} \rfloor \right)}_2[F]  \right) \eta^s dV_g\\
& \leq Q_{(2, \lfloor \frac{q+\ell}{2} \rfloor)}\left( || \eta^{s/2 } \N^{(\lfloor \frac{q+\ell}{2} \rfloor )} F ||_{L^2}^2 + || F ||^2_{L^2,\eta>0}\right) \\
& \hsp +  Q_{(2,\lceil \frac{q+\ell}{2} \rceil)}  \left( ||  \eta^{s/2 }\N^{(\lceil \frac{q+\ell}{2} \rceil)} F ||_{L^2}^2 + || F ||^2_{L^2,\eta>0} \right) \\
& = C || \eta^{s/2 } \N^{( k + \ell + 1 - (k + \ell + 1 - \lfloor \frac{q+\ell}{2} \rfloor ))} F ||_{L^2}^2 + C|| F ||^2_{L^2,\eta>0} \\
& \hsp + C||  \eta^{s/2 }\N^{( k + \ell + 1 - (k + \ell + 1 - \lceil \frac{q+\ell}{2} \rceil) )} F ||_{L^2}^2    \\
& \leq \epsilon || \eta^{s + (k + \ell + 1 - \lceil \frac{q+\ell}{2} \rceil)/2 }\N^{(k+1)} \vp ||_{L^2}^2+ C || F ||_{L^2, \eta > 0}^2 \\
& \hsp + \epsilon || \eta^{s + (k + \ell + 1 - \lfloor \frac{q+\ell}{2} \rfloor)/2 }\N^{(k+1)} \vp ||_{L^2}^2 \\
& \leq 2 \epsilon || \eta^{s/2 }\N^{(k+1)} \vp ||_{L^2}^2 + C || F ||_{L^2, \eta > 0}^2. 
\end{align*}
We explain the manipulation of the bump function power from the second to last line. Note here that since the maximum value of $q$ is $2k+\ell-1$, then we have $k + \ell + 1 - \lceil \tfrac{q + \ell}{2} \rceil \geq  1 $, which implies that the powers of $\eta$ are always larger than $s$.

\noindent \textbf{$\mathbf{T_{3,O}}$ bump function constraint.} The applications of Lemma \ref{lem:kuwert5.5} and then Corollary \ref{cor:KS5.3interp} require that $s \geq \lceil \tfrac{q+\ell}{2} \rceil$. Maximizing $q$ at the value $2k + \ell -1$, we have $s \geq 2(k+\ell)$. \newline
We thus conclude that
\begin{align*}
T_3 & \leq 3 \epsilon || \eta^{s/2 }\N^{(k+1)} \vp ||_{L^2}^2 + C || F ||_{L^2, \eta > 0}^2.
\end{align*}
\textbf{$\mathbf{T_3}$ bump function constraint.} We combine the cumulative lower bounds of $T_{3,E}$ and $T_{3,O}$ with the constraint that $s \geq 2(k+\ell)$. \newline

\noindent \fbox{\textbf{$\mathbf{T_4}$ estimate.}} For $T_4$ we apply Lemma \ref{lem:kuwert5.5} and then Corollary \ref{cor:KS5.3interp}.
\begin{align*}
T_4 &= \sum_{i=1}^k \sum_{j=2}^{i+1} \int_M \left( P_{j+1}^{(2 \ell + 2i + 2 - 2j)}[F] \right) \eta^s  dV_g  \\
&\leq \sum_{i=1}^k \sum_{j=2}^{i+1} Q_{(j+1,\ell + 1 - j)} K^{j-1} \left( || \eta^{s/2} \N^{(\ell +i+ 1 - j)} F ||_{L^2}^2 + || F ||_{L^{2}, \eta > 0}^2 \right) \\
&=\sum_{i=1}^k \sum_{j=2}^{i+1} C K^{j-1} \left( || \eta^{s/2} \N^{(k + \ell + 1-(k+ j  - i))} F ||_{L^2}^2 + || F ||_{L^{2}, \eta > 0}^2 \right) \\
& \leq \sum_{i=1}^k \sum_{j=2}^{i+1}  \epsilon || \eta^{(s+(k+ j  - i))/2} \N^{(k + 1)} \vp ||_{L^2}^2 + C K^{2(j-1)} || F ||_{L^{2}, \eta > 0}^2  \\
& \leq  \tfrac{k(k+1)}{2} \epsilon ||\eta^{s/2} \N^{(k+1)} \vp ||_{L^2}^2 +  C K^{2k} || F ||_{L^{2}, \eta > 0}^2.
\end{align*}
Note that the second to last line results from the fact that we must consider the minimizing bump function power by implementing the bounds on $i$ and $j$:  with $k - i + j \geq k -k + 2 = 2$, and thus the lowest power of $\eta$ is $s+2$, strictly greater than $s$. \\
\noindent \textbf{$\mathbf{T_4}$ bump function constraint.} For the application of Lemma \ref{lem:kuwert5.5} and Corollary \ref{cor:KS5.3interp} we require $s \geq 2 \prs{\ell + i + 1 - j}$. The right hand side of this inequality is maximized when $i = k$ and $j=2$, that is when $s \geq 2 \prs{ k+ \ell -1}$. \newline

\noindent \fbox{\textbf{$\mathbf{T_5}$ estimate.}} For $T_5$ we apply Lemma \ref{lem:kuwert5.5} and then Corollary \ref{cor:KS5.3interp}, 
\begin{align*}
T_5
&\leq  \sum_{i=1}^k \sum_{j=2}^{i+1} \int_M  \left( P^{(2 \ell + 2i + 2 - 2j)}_{j+2}[F] \right) \eta^s dV_g \\
& \leq \sum_{i=1}^k  \sum_{j=2}^{i+1}  Q_{(j+2, \ell + i + 1 -j)} K^{j} \left(||\eta^{s/2} \N^{(  \ell + i + 1 -j)} F ||_{L^2}^2 + || F ||_{L^2, \eta > 0}^2\right) \\
& = \sum_{i=1}^k  \sum_{j=2}^{i+1} C K^j  \left(||\eta^{s/2} \N^{( k + \ell + 1 -( k - i + j))} F ||_{L^2}^2 + || F ||_{L^2, \eta > 0}^2\right) \\
& \leq \sum_{i=1}^k  \sum_{j=2}^{i+1} \prs{\epsilon ||\eta^{(s+( k - i + j))/2} \N^{( k + 1)} \vp ||_{L^2}^2 + C K^{2j} || F ||_{L^2, \eta > 0}^2} \\
& \leq \sum_{i=1}^k  \sum_{j=2}^{i+1} \prs{\epsilon ||\eta^{s/2} \N^{( k + 1)} \vp ||_{L^2}^2 + C K^{2j} || F ||_{L^2, \eta > 0}^2} \\
&\leq \tfrac{k(k+1)}{2}\epsilon ||\eta^{s/2} \N^{(k+1)} \vp ||_{L^2}^2 + C K^{2k + 2} || F ||_{L^2, \eta > 0}^2 .
\end{align*}
Note that the second to last line results from the fact that we must consider the minimizing bump function power by implementing the bounds on $i$ and $j$, as done for $T_4$ previously.

\noindent \textbf{$\mathbf{T_5}$ bump function constraint.} The application of Lemma \ref{lem:kuwert5.5} and Corollary \ref{cor:KS5.3interp} require $s \geq 2 \prs{\ell + i + 1 - j}$. Using the bounds on the $i$ and $j$ the right side is maximized when $i = k$ and $j =2$, that is, when $s \geq 2 \prs{k + \ell -1}$. \newline

\noindent \fbox{\textbf{Final estimate.}}
Now we combine everything. Summing up the $T_i$ estimates and combining the constraints on $s$ we conclude that
\begin{align*}
\tfrac{d}{dt} \brk{\int_M \brs{\vp}^2 dV_g} & \leq \sum_{i=1}^5 T_{i} \\
& \leq -2  \brs{\brs{ \N^{(k+1)} \vp}}_{L^2}^2 + \prs{k(k+1) + \ell + 9} \epsilon \brs{\brs{ \N^{(k+1)} \vp}}_{L^2}^2 + CK^{2k+2} \brs{\brs{F}}_{L^2,\eta > 0}^2,
\end{align*}
where $ s \geq 2(k+\ell+1)$. Taking $\epsilon = \prs{k(k+1) + \ell + 9}^{-1}$ yields the desired result.
\end{proof}
\end{lemma}

\begin{rmk}
The bounds on $| \N^{(i)} \Rm |$ for $i \in [0,i_{k,\ell}] \cap \mathbb{N}$ (where $i_{k,\ell}$  is an index dependent on its subscripts) contribute to the coefficient $C$ in the above estimates, though these are undisplayed and incorporated into the $P_v^{(w)}$ notation. In the later blowup analysis (Proposition \ref{prop:YMkblowup}), although the manifold $M$ is changing, the bounds on $\Rm$ are actually decreasing since the base manifold is tending toward $\mathbb{R}^n$ in the blowup sequence.
\end{rmk}

\begin{thm}\label{thm:YMkNmbdbump} 
Let $q \in \mathbb{N}$, $\eta \in \mathcal{B}$, and suppose $\N_t \in \mathcal{A}_E \times \mathcal{I}$ is a solution to generalized Yang-Mills $k$-flow with $\sup_{M \times \mathcal{I}} \brs{F_{\N_t}} < \infty$.  Choose $s > 2(k + q + 1)$ and $K > \max\{ \sup_{M \times \mathcal{I}} \brs{F_{\N_t}}, 1\}$. Then for $t \in [0 , T) \subset \mathcal{I}$ with $T < K^{-2(k+1)}$ there exists $C_q := C_q\prs{ \dim M, \rank E, k, s, q, g, \jmath_{\eta}^{(s)}} \in \mathbb{R}_{>0}$, such that the following estimates hold.
\begin{equation}\label{eq: YMkgradkbdbumpeq}
\brs{\brs{ \eta^{s} \N^{(q)}_t F_{\N_t} }}_{L^2(M)}^2 \leq \frac{C_q \prs{\sup_{[0,T)}\brs{
\brs{ F_{\N_t}} }_{L^2(M), \eta>0}^2}}{t^{\frac{q}{k+1}}}.
\end{equation}

\begin{proof}
Set $\ga_q :=1$, and let $\{ \ga_{\ell} \}_{\ell=0}^{q-1} \subset \mathbb{R}$ be coefficients to be determined. Then set
\begin{equation*} 
\Phi(t) := \sum_{\ell=0}^q{\ga_\ell t^{\ell} \left\| \eta^{s}  \N_t^{(( k +1)\ell )} F_t \right\|_{L^2}^2}.
\end{equation*}
Then differentiating, reindexing and applying Lemma \ref{lem:YMketavpest} yields
\begin{align*}
\frac{d \Phi}{d t}&=  \sum_{\ell=1}^q{\ell \ga_{\ell} t^{\ell-1} ||\eta^{s}  \N^{((k+1) \ell)} F_t ||_{L^2}^2} +\sum_{\ell=0}^q{\ga_{\ell} t^{\ell} \frac{d}{d t} \lb || \eta^{s} \N^{((k+1) \ell)} F_t ||_{L^2}^2 \rb}\\
& \leq \sum_{\ell =0}^{q-1}{\left( (\ell+1) \ga_{\ell + 1} t^{\ell} || \eta^{s} \N^{((k+1) (\ell+1))} F_t ||_{L^2}^2 \right)} \\
& \hsp + \sum_{\ell =0}^q{ \left( \ga_{\ell} t^{\ell} \left(  -  \left\| \eta^{s} \left(\N^{(k+1)(\ell +1)} F_t \right) \right\|_{L^2}^2  + CK^{2(k+1)} \left\| F_t \right\|_{L^2, \eta>0}^2 \right)  \right)}\\
&= - t^q  \left|\left| \eta^s \left( \N^{((k+1)(q+1))} F_t \right) \right|\right|_{L^2}^2 +  \sum_{\ell = 0}^{q-1} \left( \ga_{\ell +1}(\ell + 1) - \ga_{\ell} \right)t^{\ell} \left| \left| \eta^s \left( \N^{((k+1)(\ell+1))} F_t \right) \right| \right|_{L^2}^2 \\ 
& \hsp + CK^{2(k+1)} \sum_{\ell = 0}^q \ga_{\ell} t^{\ell} ||F_t||_{L^2,\eta>0}^2.
\end{align*}
Using the initial condition $\ga_q = 1$, we choose constants satisfying the recursion relation
\[ \ga_{\ell +1}(\ell + 1) - \ga_{\ell} \leq 0,\]
so in particular, we choose constants which satisfy $\ga_{\ell} \geq \tfrac{q!}{\ell !}$. Then incorporating the fact that $t$ is bounded above by $K^{-2(k+1)}$ and choosing $C_{(k+1)q} \geq C (\sum_{\ell = 0}^q \ga_{\ell})$,
\[
\frac{d \Phi}{d t} =CK^{2(k+1)} \sum_{\ell = 0}^q \ga_{\ell} t^{\ell} || F_t ||_{L^2, \eta>0}^2 \leq C_{(k+1)q} K^{2(k+1)} || F_t ||_{L^2, \eta>0}^2
\]
integrating both sides with respect to the temporal variable yields that
\[
\Phi(t) - \Phi(0) \leq C_{(k+1)q}  K^{2(k+1)} \int_0^t{ || F_{\tau} ||_{L^2, \eta>0}^2 d \tau}, \]
therefore
\begin{align*}
t^q || \eta^{s} \N^{((k+1)q)} F_t ||_{L^2}^2
&\leq C_{(k+1)q} t K^{-2(k+1)} \left( \sup_{t \in [0,T]} || F_t ||_{L^2, \eta>0}^2 \right) + \Phi(0)\\
&\leq C_{(k+1)q} \left( \sup_{t \in [0,T]} || F_t ||_{L^2, \eta>0}^2 \right) + q! \brs{\brs{F_0}}_{L^2, \eta>0}.
\end{align*}
We conclude
\[
|| \eta^{s} \N^{((k+1)q)} F ||_{L^2}^2 \leq \frac{ C_{(k+1)q} }{t^q} \left( \sup_{t \in [0,T]} || F_t ||_{L^2, \eta>0}^2 \right).
\]
To complete the proof we consider remaining derivative types (that is, $|| \eta^{s}  \N^{((k+1)\ell+w)} F_t ||_{L^2}^2$ where $m \in \mathbb{N}\cup \{ 0 \}$ and $w \in [1,k] \cap \mathbb{N}$).
We observe that by Corollary \ref{cor:KS5.3interp} combined with the fact that $T < 1$, then
\begin{align*}
\left\| \N^{((k+1)\ell + w)} F_t \right\|_{L^2}^2 & \leq \epsilon \left\| \N^{((k+1)(\ell + 1))} F_t \right\|_{L^2}^2 +  C_{\epsilon} \left\| F_t \right\|_{L^2, \eta>0}^2 \\
&\leq  \frac{C_{(k+1)(\ell + 1)} \left( \sup_{[0,T)} \| F_t \|_{L^2, \eta>0}^2 \right)}{t^{\ell + 1}} +   \frac{C_{\epsilon} \left\| F_t \right\|_{L^2, \eta > 0}^2}{t^{\ell + 1}}\\
& \leq   \frac{C \sup_{[0,T)} \| F_t \|_{L^2, \eta>0}^2 }{t^{\frac{(k+1)(\ell + 1) +w}{(k+1)}}}.
\end{align*}
We have established the inequality for all $q \in \mathbb{N}$, and the result follows.
\end{proof}
\end{thm}

\begin{cor}\label{cor:Nindest}
Suppose $\N_t \in \mathcal{A}_E \times [0,\tau]$ is a solution to generalized Yang-Mills $k$-flow and $\eta \in \mathcal{B}$. Set $\bar{\tau} := \min \{ \tau, 1 \}$ and $K \geq \sup_{M \times [0,\bar{\tau}]} |F_{\N_t}|$. Then for $s, \ell \in \mathbb{N}$ with $s > (k +\ell + 1)$ there exists $Q_{\ell} := Q_{\ell}\prs{K, s, \ell, \jmath_{\eta}^{(s)}, \rank E, g, \dim M, \bar{\tau}} \in \mathbb{R}_{>0}$ such that
\begin{equation}
\sup_M \left| \eta^s \N^{(\ell)}F_{\N_{\bar{\tau}}} \right|^2 \leq Q_{\ell} \prs{\sup_{M \times [0,\bar{\tau})} \left\| F_{\N_t} \right\|_{L^2(M), \eta>0}^2}.
\end{equation}

\begin{rmk}
Note that the corollary results have no dependency on the initial connection $\N_0$.
\end{rmk}

\begin{proof}
By the smoothing estimates of Theorem \ref{thm:YMkNmbdbump} we have that
\begin{equation*}
\sup_{M}|| \eta^{s} \N^{(\ell)} F_{\bar{\tau}} ||_{L^2}^2 \leq Q_{\ell} \sup_{M \times [0,\bar{\tau})}||F_t||^2_{L^2, \eta>0}.
\end{equation*}
Since $| \eta^{s} \N^{(\ell)}F |^2$ is a real valued function on $M \times [0, \infty)$, and any $j> \frac{n}{2}$, we use Kato's Inequality (Lemma \ref{lem:kato}) combined with the second Sobolev Imbedding Theorem (Theorem \ref{thm:sobolevimbedding}), which gives that $H_{j}^2 \subset C^0_B $, yielding the Sobolev constant $\mathsf{S}_{\ell}$ so that
\begin{equation*}
\sup_{M} \left| \eta^{s} \N^{(\ell)}F_{\bar{\tau}} \right|^2 \leq \mathsf{S}_{\ell} \sum_{w=0}^{j}{|| \N^{(w)} \left(| \eta^{s}\N^{({\ell})} F_{\bar{\tau}}| \right) ||_{L^2}}  \leq \mathsf{S}_{\ell} \sum_{w=0}^{j}{|| \eta^{s} \N^{(k+w)} F_{\bar{\tau}}||_{L^2}} . 
\end{equation*}
Therefore, combining the two above inequalities,
\begin{align*}
\sup_M \left| \eta^{s} \N^{(\ell)}F_{\bar{\tau}} \right|^2 &\leq \mathsf{S}_{\ell} \sum_{w=0}^j{ \left( C_{\ell+w} \sup_{[0,\bar{\tau})}||F_t||^2_{L^2, \eta>0} \right)^{1/2}} \\
& \leq \mathsf{S}_{\ell} \left( \max_{w \in [0,j] \cap \mathbb{N}}{C_{\ell+w}} \right) \left( \sup_{M \times [0,\bar{\tau})}||F_t||_{L^2, \eta>0} \right).
\end{align*}
Setting $Q_{\ell} :=\mathsf{S}_{\ell} \left( \max_{w \in [0,j] \cap \mathbb{N}}{C_{\ell+w}} \right)$, the result follows.
\end{proof}
\end{cor}

\begin{cor}\label{cor:YMkFunifbd0toT}
Suppose $\N_t \in \mathcal{A}_E \times [0,T)$  is a solution to generalized Yang-Mills flow for $T \in [0,\infty)$ and that $\eta \in \mathcal{B}$. Furthermore suppose that
\begin{equation*}
\max \left\{\sup_{[0,T)} \brs{ \brs{ F_{\N_t} }}_{L^2(M)}, \sup_{[0,T)} \brs{ \brs{ F_{\N_t} }}_{L^{\infty}(M)} \right\} \leq K \in [1, \infty).
\end{equation*}
Then for $t \in [0,T)$, $s,\ell \in \mathbb{N}$ with $s > 2(k + \ell + 1)$ there, exists some
\begin{equation*}
Q_{\ell} := Q_{\ell}(\N_0, T, K, \ell, s, g, \rank(E), \dim(M) , \jmath_{\eta}^{(s)}) \in \mathbb{R}_{>0}
\end{equation*}
such that
\begin{equation*}
\sup_{M \times [0, T)}\left| \eta^{s} \N_t^{(\ell)}F_t \right|^2 \leq Q_{\ell}.
\end{equation*}

\begin{proof}
Since $\N_t$ exists and is smooth on the compact interval $\left[ 0,\frac{T}{2} \right]$ over $M$, which is also compact, then for each $\ell \in \mathbb{N}$ there exists a constant $B_{\ell} >0$ dependent on $\N_0$ such that
\[
\sup_{\left[0, \frac{T}{2}\right]} \left|\left| \eta^{s} \N_t^{(\ell)}F \right| \right|_{L^2}^2 \leq B_{\ell}.
\]
Let $\tau := \min{\{ \frac{1}{K^{2(  \ell +k )}}, \frac{T}{2} \}}$. Then if we consider any time $t > \tfrac{T}{2}$ we can consider the estimate given on the interval $[ t- \tau, t]$ by applying Corollary \ref{cor:Nindest} to obtain a local pointwise bound. However, since this bound is independent of each $t$ (only relying on the time $\tfrac{T}{2}$), then we in fact have a uniform bound over $\left[\tfrac{T}{2},T \right)$. Taking the maximum of this and $B_{\ell}$ we achieve the desired result.
\end{proof}
\end{cor}

\subsection{Long time existence obstruction}\label{ss:longtimeexist}

We first prove two general lemmas which are in fact independent of the flow. The first is a completely general manipulation, and the second only relies on the bounds on $| \N^{(\ell)}_t F_{\N_t} |$ for $\ell \in \mathbb{N}$ for a given family $\N_t \in \mathcal{A}_E \times \mathcal{I}$.

\begin{lemma}\label{lem:difofNs} Let $\N,\widetilde{\N} \in \mathcal{A}_E$ and set $\Upsilon := \widetilde{\N} - \N$. Then for all $\zeta$ in some tensor product of $TM$, $E$, and their corresponding duals,
\begin{align*}
\widetilde{\N}^{(\ell)} \lb \zeta \rb &= \N^{(\ell)} \lb \zeta \rb + \sum_{j=0}^{\ell-1} \sum_{i=0}^j \left( P^{(i)}_{\ell-1-i} \lb \gY \rb \ast P_1^{(j-i)} \lb \zeta \rb \right) \\
&=\N^{(\ell)} \lb \zeta \rb  +  \sum_{j=0}^{\ell-1} \sum_{i=0}^j \left( \widetilde{P}^{(i)}_{\ell-1-i} \lb \gY \rb \ast \widetilde{P}_1^{(j-i)} \lb \zeta \rb \right).
\end{align*}

\begin{proof}
The first summation follows simply from direct computation, so we address the second. Note that any quantity in $P_w^{(v)} \lb \gY \rb$ has the form, for some multiindex $(r_j)_{j=1}^{w}$ where $\sum_{j=1}^w r_j = v$,
\begin{equation*}
P_w^{(v)} = \N^{(r_1)}\lb \gY \rb \ast \cdots \ast \N^{(r_w)}\lb \gY \rb.
\end{equation*}
Then replacing $\N = \widetilde{\N} + \Upsilon$ yields
\begin{align*}
P_w^{(v)} & = \N^{(r_1)}\lb \gY \rb \ast \cdots \ast \N^{(r_w)}\lb \gY \rb \\
&= (\widetilde{\N} + \gY \ast)^{(r_1)}\lb  \gY \rb \ast \cdots \ast (\widetilde{\N} + \gY \ast)^{(r_w)} \lb \gY \rb \\
&=  \left( \sum_{q=0}^{r_1} \sum_{p=0}^{q} \left( \widetilde{P}^{(q)}_{r_1+1-p} \lb \gY \rb  \right) \right) \ast \cdots \ast \left( \sum_{q=0}^{r_w} \sum_{p=0}^{q} \left( \widetilde{P}^{(q)}_{r_w + 1-p} \lb \gY \rb  \right) \right) \\
& = \sum_{q=0}^{v} \sum_{p=0}^q \left( \widetilde{P}^{(q)}_{v+w-p} \lb \gY \rb  \right).
\end{align*}
Thus we have that 
\begin{align*}
\left( P^{(i)}_{\ell-1-i} \lb \gY \rb \ast P_1^{(j-i)} \lb \zeta \rb \right) &= \left( \sum_{q=0}^{i} \sum_{p=0}^q \left( \widetilde{P}^{(q)}_{\ell - 1-p} \lb \gY \rb  \right) \right) \ast P_1^{(j-i)} \lb \zeta \rb ,
\end{align*}
and thus
\begin{align*}
 \sum_{j=0}^{\ell-1} \sum_{i=0}^j  \left( \sum_{q=0}^{i} \sum_{p=0}^q \left( \widetilde{P}^{(q)}_{\ell - 1-p} \lb \gY \rb  \right) \right) \ast P_1^{(j-i)} \lb \zeta \rb &= \sum_{j=0}^{\ell-1} \sum_{i=0}^j \left( \widetilde{P}^{(i)}_{\ell-1-i} \lb \gY \rb \ast \widetilde{P}_1^{(j-i)} \lb \zeta \rb \right).
\end{align*}
The result follows.
\end{proof}
\end{lemma}

For the following proof let $\N_t \in \mathcal{A}_{E} \times [0,T)$ for some $T < \infty$, and set

\begin{equation*}
\Upsilon_s  := \int_{0}^{s}{\left( \frac{\del \N}{\del t} \right) dt}.
\end{equation*}
Note that for all $s < T$, we have $\Upsilon_s := \N_s - \N_0$.

\begin{prop}\label{prop:Ncmpt} Let $\N_t \in \mathcal{A}_E \times [0,T)$ for some $T < \infty$. Suppose further that for all $\ell \in \mathbb{N}$ there exists $C_{\ell} \in \mathbb{R}_{>0}$ such that 
\[
\sup_{M \times [0,T)}{\left| \N_{t}^{(\ell)} \lb \tfrac{\del \N}{\del t} \rb \right|} \leq C_{\ell}.
\]
Then $\lim_{t \to T}{\N_t} := \N_T$ exists and is smooth. 
\end{prop}

\begin{proof}
We first demonstrate that $\N_T := \lim_{t \to T}\N_t$ exists in $C^0(M)$. For all $s \leq T$ we have
\begin{equation}\label{eq:abd}
\left| \gY_s \right| = \left| \int_{0}^s{\left(\tfrac{\del \N}{\del t} \right) dt}\right| \leq TC_0.
\end{equation}
This implies, since $\N_0$ is continuous, that $\N_T$ is continuous.

Next we demonstrate smoothness of $\N_T$. The proof proceeds by induction on the $\ell \in \mathbb{N}$ satisfying $\left|\N_0^{(\ell)} \lb \gY_T \rb \right| < \infty$. Let $s < T$. For the base case,
\begin{align*}
\left| \N_0 \lb \gY_s \rb \right|&= \left|\int_{0}^{s}{\N_{0} \lb \tfrac{\del \N}{\del t} \rb dt} \right|\\
&= \left| \int_{0}^{s}{\left(\N_t \lb \tfrac{\del \N}{\del t} \rb - \gY_t \ast \left(\tfrac{\del \N}{\del t} \right) \right) dt} \right| \\
& \leq \int_{0}^s{\left(\left| \N_t \lb \tfrac{\del \N}{\del t} \rb \right| + C |\gY_t| \left| \left(\tfrac{\del \N}{\del t} \right) \right| \right)dt},
\end{align*}
where $C= C \prs{\dim M, \rank E} \in \mathbb{R}_{\geq 0}$.
Applying this to the above computations,
\[
| \N_0 \lb \gY_s \rb | \leq \int_{0}^s{\left(\left| \N_t \lb \tfrac{\del \N}{\del t} \rb \right| + C |\gY_t | \left| \tfrac{\del \N}{\del t} \right| \right)dt}\leq TC_1+CTC_0^2 < \infty.
\]
Since $\gY_s$ is continuous and the bound above is uniform across $s \in [0, T)$ then $|\N_0 \lb \gY_s \rb | < \infty$, completing the proof of the base case.

Now let $\ell \in \mathbb{N}$ and suppose the induction hypothesis is satisfied for $\{ 1, \dots, \ell-1 \}$. Expanding $\N_0^{(\ell)} \lb \gY_s \rb$ and applying Lemma \ref{lem:difofNs}, 
\begin{align*}
\N_0^{(\ell)} \lb  \gY_s \rb &= \int_0^s{\N_0^{(\ell)} \lb \tfrac{\del \N}{\del t} \rb dt} \\
&= \int_0^s{ \left( \N^{(\ell)}_t\lb \tfrac{\del \N}{\del t} \rb +  \sum_{j=0}^{\ell-1} \sum_{i=0}^j \left( P^{(i)}_{\ell-1-i} \lb \gY_t \rb \ast P_1^{(j-i)} \lb \tfrac{\del \N}{\del t} \rb \right) \right) dt.}
\end{align*}
Where here $P$ is taken with respect to $\N_t$. Taking the norm of the quantities gives, for $C > 0$,
\begin{align*}
\left| \N_0^{(\ell)} \lb  \gY_s  \rb \right| 
&= \int_0^s{ \left( \left| \N^{(\ell)}_t \lb   \tfrac{\del \N}{\del t} \rb \right|  +  \sum_{j=0}^{\ell-1} \sum_{i=0}^j \left(  P^{(i)}_{\ell-1-i} \lb \gY_t \rb \ast P_1^{(j-i)} \lb \tfrac{\del \N}{\del t} \rb \right) \right) dt.}
\end{align*}
Each term is bounded by assumption, and in particular, all terms on the right are bounded by the induction hypothesis, $\left|\N_0^{(\ell)} \lb \gY_s  \rb \right| < \infty$. Since our choice of bounds are uniform for all $t \in [0,T)$ and $\gY_s$ is continuous, it follows that $\left|\N_0^{(\ell)} \lb \gY_T \rb \right| < \infty$, so the induction hypothesis is satisfied by $\ell$. Thus $\gY_T$ is smooth, so since
\[ \gY_T = \lim_{t \to T} \gY_t = \lim_{t \to T} \left( \gG_0 - \gG_t \right) = \gG_0 - \lim_{t \to T}\gG_t, \]
then  $\N_t$ may be extended to $\N_T : = \lim_{t \to T}{\N_t}$, which is smooth. The result follows. 
\end{proof}

Using the previous results we demonstrate that the only obstruction to long time existence of the flow is a lack of supremal bound on the curvature tensor.

\begin{thm}\label{thm:YMktexistobs}
Suppose $\N_t$ is a solution to generalized Yang-Mills $k$-flow for some maximal $T< \infty$. Then 
\begin{equation*}
\sup_{M \times [0, T)} | F_{\N_t} | = \infty.
\end{equation*}

\begin{proof} Suppose to the contrary that $\sup_{M \times [0,T)}|F_{\N_t}| \leq K < \infty$. Then by Corollary \ref{cor:YMkFunifbd0toT} for all $t \in [0,T)$ and $\ell \in \mathbb{N}$, we have $\sup_M |\N^{(\ell)}F_{\N_t}|$ is uniformly bounded and so by Proposition \ref{prop:Ncmpt}, $\N_T := \lim_{t \to T}{\N_t}$ exists and is a smooth solution to generalized Yang-Mills $k$-flow for such $t$. However, by Proposition \ref{prop:YMkstexist}, there exists $\epsilon >0$ such that $\N_{t}$ exists over the extended temporal domain $[0, T+ \epsilon)$, which contradicts the assumption that $T$ was maximal. Thus $\sup_{M \times [0, T)} | F_{\N_t} | = \infty$, and the result follows.
\end{proof}
\end{thm}

\subsection{Blowup analysis}\label{ss:blowupanalysis}

We now address the possibility of Yang-Mills $k$-flow singularities given no bound on the spatial supremum of curvature. To do so we require the following theorem of Uhlenbeck \cite{Uhl}.

\begin{thm}[Coulomb Gauge Theorem, Theorem 1.3, pp. 33 of \cite{Uhl}]\label{thm:Uhlenbeck}  Suppose $2p \geq n$. There are constants $\kappa_n, c_n \in \mathbb{R}_{>0}$ such that any connection $\N$ on the trivial bundle over $B_0(1) \subset \mathbb{R}^n$, the unit ball, with $|| F_{\N} ||_{L^{(n/{2})}(\mathbb{R}^n)} < \kappa_n$ is gauge equivalent to a connection $\widetilde{\N}$ with coefficient matrix $\widetilde{\gG}_t$ over $B_0(1)$ such that
\begin{enumerate}
\item $D^*_{\widetilde{\N}} \widetilde{\gG} = 0$,
\item $\| \widetilde{\gG} \|_{C^{p, 1}(\mathbb{R}^n)} \leq c_n || F_{\N} ||_{C^{p,0}(\mathbb{R}^n)} = c_n \left\| F_{\widetilde{\N}} \right\|_{C^{p,0}(\mathbb{R}^n)}$.
\end{enumerate}
\end{thm}

For our setting we will need to extend this choice of gauges over a large region. This will be accomplished using the following result stated in Donaldson and Kronheimer's text \cite{DonB}.

\begin{thm}[Gauge patching theorem, Corollary 4.4.8, pp.159 of \cite{DonB}]\label{thm:gaugepatch} Suppose $\{ \N^i \}$ is a sequence of connections on $E$ over $M$ with the following property: for each $x \in M$ there is a neighborhood $U_x$ and a subsequence $\{ \N^{i_j} \}$ with corresponding sequence of gauge transformations $s_{i_j}$ defined over $M$ such that $s_{i_j}^*\N^{i_j}$ converges over $U_x$. Then there is a single subsubsequence $\{ \N^{i_{j_k}} \}$ defined over $M$ such that $s^*_{i_{j_k}} \N^{i_{j_k}}$ converges over all of $M$.
\end{thm}

Now we establish some preliminary scaling laws of Yang-Mills $k$-flow, and discuss the effect that this scaling has on the generalized flow, which will be key in the proceeding blowup analysis argument.

\begin{lemma}\label{lem:YMkscal}
Suppose $\N_t \in \mathcal{A}_E \times \mathcal{I}$ is a solution to Yang-Mills $k$-flow with local coefficient matrices $\gG_t$. Define the one-paramater family $\N^{\la}_t \in \mathcal{A}_E \times \mathcal{I}$ with coefficient matrices given by
\begin{equation}\label{eq:YMkscal}
\Gamma^{\la}_t (x):= \la \Gamma_{\la^{2(k+1)} t} (\la x).
\end{equation}
Then $\N^{\la}_t$ is also a solution to Yang-Mills $k$-flow.
\begin{proof}

Using the scaling as in \eqref{eq:YMkscal} we set $F^{\la}_t : = F_{\N^{\la}_t}$ and $D^{\la}_t := D_{\N^{\la}_t}$. Then we insert $\N^{\la}_t$ into the Yang-Mills $k$-flow equation. First, we take the temporal derivative. 
\begin{equation*}
\frac{\del \N^\la_t}{\del t} = \frac{\del \Gamma^\la_t}{\del t}  = \la^{2k+3}\frac{\del \Gamma_t}{\del t} = \la^{2k+3} \frac{\del \N_t}{\del t}.
\end{equation*}
Now we compare this to the scaling of $\Grad \YM_1$. If we revisit the proof of Proposition \ref{prop:ELYMk}, rather than commute connections and perform integration by parts (in order to get clean Laplacian pairings), if we instead just directly integrate by parts, one confirms that
\begin{align*}
\Grad \YM_1 (\N) = P_{1,-\Rm}^{(2k+1)}\brk{F_{\N}} + P_{2,-\Rm}^{(2k-1)}\brk{F_{\N}},
\end{align*}
where here $P_{v, -\Rm}^{(u)}$ notation is that of Definition \ref{sss:Pdefn} with the added constraint that the quantities are written entirely in terms of derivatives of $F_{\N}$. With this in mind, as a consequence of Lemma \ref{lem:Nscal},
\begin{align*}
\Grad \YM_1 (\N^{\la}_t) &= \la^{2k+3} \Grad \YM_1 (\N_t).
\end{align*}
Thus the desired scaling law holds through the Yang-Mills $k$-flow equation. The result follows.
\end{proof}
\end{lemma}

\begin{rmk}\label{rmk:mhoirrel}
While the proof of Lemma \ref{lem:YMkscal} only applies in the case of Yang-Mills $k$-flow, we may utilize the result on the generalized flow in the subsequent blow up analysis. To see this we must verify that the order of $\mho_k$ works favorably with respect to the scaling law (that is, if we send $\la \to 0$). Rescaling with $\Gamma^{\la}_t (x):= \la \Gamma_{\la^{2(k+1)} t} (\la x )$ we have, by appealing to Lemma \ref{lem:Nscal},
\begin{align*}
\mho_k( \N^{\la}) = \sum_{i=1}^{k} \sum_{j=2}^{i + 1}{(P_{\la})_j^{(2
i + 3 -2j)} \lb F_{\N^{\la}}\rb} = \sum_{i=1}^{k} \sum_{j=2}^{i + 1}{\la^{2i + 3} P_j^{(2
i + 3 -2j)} \lb F_{\N} \rb} .
\end{align*}
We have that $5 \leq 2i + 3 \leq 2k + 3$. As $\la \to 0$ we have that all terms are dominated except those of the form
\begin{equation*}
\sum_{j=2}^{k+1} \la^{2k+3}P_j^{(2
k + 3 -2j)} \lb F_{\N} \rb.
\end{equation*}
These are the dominant quantities of $\mho_{k}$ in the context of rescaling for small $\la$ and agree with the scaling law of Yang-Mills $k$-flow, thus ultimately preserving the behavior of the blowup limit.
\end{rmk}

We now demonstrate the construction of a generalized Yang-Mills $k$-flow blowup limit in the following proposition.

\begin{prop}\label{prop:YMkblowup}
Suppose there exists a maximal $T \in [0, \infty]$ such that $\N_t \in \mathcal{A}_E \times [0,T)$ is a solution to generalized Yang-Mills $k$-flow and $\lim_{t \to T}\brs{\brs{F_{\N_t}}}_{L^{\infty}(M)} = \infty$. Then a blowup sequence $\{ \N^i_t \}$ exists and converges pointwise to a smooth solution $\bm{\N}_t$ to generalized Yang-Mills $k$-flow with domain $\mathbb{R}^n \times \mathbb{R}_{\leq 0}$. Additionally for each $q \in \mathbb{N}$ there exists $C_q= C_q\prs{K, q, \rank E, g, \dim M}$ such that
\begin{equation*}
\sup_{(x,t) \in \mathbb{R}^n \times \mathbb{R}_{\leq 0}} \left| \bm{\N}^{(q)}_t F_{\bm{\N}_t} \right| \leq C_q,
\end{equation*}
and furthermore $\brs{F_{\bm{\N}_0}(0)} = 1$.

\begin{proof} Choose a sequence $ \tau_i  \nearrow T$ within $[0,T)$. For each $\tau_i$, there exists a point $(x_i,t_i) \in M \times [0,T)$ so that
\begin{equation*}
\brs{F_{\N_{t_i}}(x_i)} = \sup_{(x,t) \in M \times [0,\tau_i]} \brs{F_{\N_{t}}(x)}.
\end{equation*}
Choose a subsequence so that $\{ x_i \}$ converges to some $x_{\infty} \in M$. There exists some chart about the blowup center $x_{\infty}$ so that the tail of the sequence $\{ x_i \}$ is contained within the single chart mapping into $B_0(1)$. We will only consider the sequence tail, so it is therefore sufficient to utilize the coordinate chart and assume that the sequence is contained in $\mathbb{R}^n$. Thus we may identify connections with their coefficient matrices in this argument.

 Let $\{\la_i \} \subset \mathbb{R}_{> 0}$ be constants to be determined, and set 
\begin{equation*}
\Gamma^i_t(x):= \la_i^{1/(2k+2)} \Gamma_{\la_i t + t_i}(\la_i^{1/(2k+2)} x + x_i). 
\end{equation*}
Note the corresponding curvatures $F_t := F_{\N_t}$ is scaled in the following manner,
\[ F_{t}^i(x) = \la_i^{1/(k+1)} F_{\la_i t + t_i} (\la_i^{1/(2k+2)}x + x_i) .\]
By Lemma \ref{lem:YMkscal} all corresponding $\N^i_t$ are also solutions to the generalized Yang-Mills $k$-flow (though with different initial conditions and scaled lower order terms). The domain for each $\N^i_t$ is $B_{x_i}{\prs{(\la_i)^{-1/(2k+2)}}} \times [ \frac{-t_i}{\la_i}, \frac{T - t_i}{\la_i}]$. We will choose $\la_i$ to mitigate the `blowing up' of the sequence curvatures $F^i_t$ by observing the following:
\begin{align*}
\sup_{t \in \lb -\frac{t_i}{\la_i},\frac{T - t_i}{\la_i} \rb}{\brs{F^i_t(x)}} &= \left|  \la_i^{1/(k+1)} \right| \sup_{t \in \left[-\frac{t_i}{\la_i},\frac{T - t_i}{\la_i}\right]}{\left|F_{\la_i t + t_i}(\la_i^{1/(2k+2)} x + x_i)\right|} \\
&= \left|  \la_i^{1/(k+1)} \right| \sup_{t \in [0,t_i ]}{|F_t(x)|} \\
&= \left|  \la_i^{1/(k+1)} \right|| F_{t_i}(x_i)|.
\end{align*}
Setting $\la_i = |F_{t_i}(x_i)|^{-(k+1)}$ gives that 
\begin{equation}\label{eq:controlledsup}
 1 = \brs{F^i_0(0)} =  \la_i^{1/(k+1)}{\brs{F_{0+ t_i}(0 + x_i)}} = \sup_{t \in \lb -\frac{t_i}{\la_i}, 0 \rb}{|F^i_t(x)|}.
 \end{equation}
Note that we have no information on the behavior of the $F^i_t$ for $t>0$, so this bound \eqref{eq:controlledsup} induced by the choices of $\lambda_i$ is only guaranteed on $\mathbb{R}_{\leq 0}$.

Next we construct smoothing estimates for the sequence $\{ \N^i_t \}$. Let $y \in \mathbb{R}^n$, $\tau \in \mathbb{R}_{\leq 0}$ and consider sequence indices $i \in \mathbb{N}$ such that the domain of $\N^i_t$ contains $B_{y}(1) \times [\tau -1, \tau]$. Let $\eta_y \in C_c^{\infty}(M)$ be chosen so that
\begin{equation*}
\begin{cases}
0  \leq \eta_y \leq 1,&\\
\supp \eta_y &= B_{y}\left( 1 \right),\\
\left. \eta_y \right|_{B_{y}\left( \frac{1}{2} \right)} &= 1.
\end{cases}
\end{equation*}
For any $s \in \mathbb{N}$ note that $\sup_{[\tau-1,\tau]}{|\eta_y^{s}F^i_t |} \leq 1$. Since each $\eta_y^{s} F^i_t $ is smooth on $[\tau-1,\tau]$ then by Corollary \ref{cor:Nindest} for all $q \in \mathbb{N}$ one may choose $s > 2(k+q +1)$ so that there exists a $Q_q$ such that
\[ \sup_{\{ \tau \} \times B_y(\frac{1}{2}) }\brs{ \prs{\N^i_t}^{(q)} F^i_t } \leq \sup_{\{ \tau \} \times B_y(1) }\brs{ \eta^{s}_y  \prs{\N^i_t}^{(q)} F^i_t }  \leq Q_q.\]
Utilizing the same bump function and recentering it at each point, we establish uniform smoothing estimates across all of $\mathbb{R}^n$. Therefore, we obtain uniform pointwise bounds for $|(\N^i_t)^{(q)}F^i_t|$ for all $q,i \in \mathbb{N}$.

Let $\epsilon, R \in \mathbb{R}_{>0}$ and $\tau \in \mathbb{R}_{\leq 0}$. For any $m \in \mathbb{N}$ we consider the compact time interval $[\tau -m,- \frac{1}{m}]$ and all $i \in \mathbb{N}$ such that the domain of $\N_t^i$ contains $[\tau -m,- \frac{1}{m}] \times B_0({R + m + \epsilon})$. Since the $F^i_t$ are all pointwise uniformly bounded by $1$, then there exists some $\delta >0$ so that for any $y \in \mathbb{R}^n$, we have $|| F^i_t ||_{L^{n /2}{B_y({\delta})}} \leq \kappa_n$, where $\kappa_n$ is as defined in the Coulomb Gauge Theorem (Theorem \ref{thm:Uhlenbeck}). We rescale coordinates of the sequence of connections $\N^i_t$ restricted to $B_y({\delta})$ (so that $B_y({\delta}) \mapsto B_{0}(1)$), again using the scalings laws to preserve that $\N^i_t$ is a solution to Yang-Mills $k$-flow up to highest order terms, by setting
\[ \widetilde{\gG}^i_t(x) : =\frac{1}{\delta^2} \gG^i_{ t\delta^{-2} } \left(\frac{x-y}{\delta}\right). \]
We apply the Coulomb Gauge Theorem for $t= \delta^2(\tau - m)$ to obtain a sequence of  connections $\widetilde{\gY}^i_t$ which are gauge equivalent to $\widetilde{\gG}_t^i$ on $B_0(1)$ and some $c_n >0$ satisfying, for all $i \in \mathbb{N}$,
\[ \brs{\brs{ \widetilde{\gY}^i_t (\delta^2(\tau - m))}}_{C^{\rho, 1}} \leq c_n.  \]
Note that for $t = \delta^2(\tau - m)$ the curvatures corresponding to $\widetilde{\gY}^i_t$ coincide with the curvatures corresponding to $\widetilde{\gG}^i_t$. We will denote these by $\widetilde{F}_{t}^i$. As a result of Lemma \ref{lem:YMkscal} the curvatures scale with the coefficient matrices and so the derivatives of curvatures are also uniformly bounded, that is, for all $q \in \mathbb{N}$ there exists $C_q > 0$ such that
\[
\sup_{B_{0}(1) \times[\delta^2(\tau - m), \frac{-\delta^2}{m}]}{\brs{ \tN_t^{(q)} \widetilde{F}_{t}^i }} \leq C_{q}.
\]
 Consequently, as demonstrated in Proposition \ref{prop:Ncmpt}, given the short time existence of $\widetilde{\N}^i_t$ by Proposition \ref{prop:YMkstexist}, it follows that it exists for $t \in [\delta^2(\tau - m), \frac{-\delta^2}{m}]$ such that for some $C_{(\tau - m)} \in \mathbb{R}_{>0}$, 
\[ \sup_{B_{0}(1) \times [\delta^2(\tau-m), \frac{-\delta^2}{m}]}|| \widetilde{\gY}^i_t ||_{C^{\rho,1}} \leq \delta^2 C_{(\tau - m)}.  \]
We redialate and shift the coordinates to recover the domain $B_{y}(\delta)$ and thus obtain a new sequence of connections $\gY_i$ defined on $B_{y}(\delta)$ by
\[ \gY^i_t(x) := \delta^2 \widetilde{\gY}^i_{\delta^2 t}( \delta x + y),
\]
which satisfies
\[ \sup_{B_{y}(\delta) \times [\tau-m, \frac{-1}{m}]} \brs{\brs{ \gY^i_t}}_{C^{\rho, 1}} \leq C_{(\tau-m)}.  \]
Again, we emphasize that in fact each $\gY^i_t$ is a solution to generalized Yang-Mills $k$-flow by virtue of the steps used to construct it, and furthermore $\gY_t^i$ is gauge equivalent to $\gG^i_t$. Taking a countable covering of these balls for $y \in B_{0}(R+m)$, we apply the Gauge Patching Theorem (Theorem \ref{thm:gaugepatch}) and obtain sequence of connection matrices $\{ \gY^i_t \}$ defined over all of $B_{0}(R+m)$. 

Recursively consider $\rho \in \mathbb{N}$ starting with $m =1$ and choose $\ga, \ga' \in (0,1)$ with $\ga' < \ga$.  Given $\{ \gY^i_t \}$, with bounds of this sequence in $ C^{\rho, \ga}$, by the Arzela-Ascoli Theorem (Theorem \ref{thm:arzelaascoli}) there exists some subsequence $\{ \gY^{i_j}_t \}$ which converges with respect to $C^{\rho, \ga'}$ to some $(\gY_{\rho})_t^{\infty}$. Note that for any $\rho_1, \rho_2 \in \mathbb{N}$ with $\rho_1 < \rho_2$, $(\gY_{\rho_1})_t^{\infty} = (\gY_{\rho_2})_t^{\infty}$ since $ C^{\rho_2, \ga'}$ is a topological subspace of $C^{k_1, \ga'}$ so convergence coincides. Relabel the subsequence terms $\gY_t^{i_j} \mapsto \gY_t^{i}$ and repeat this refinement through all $\rho \in \mathbb{N}$. We observe that as $\rho \to \infty$, the resulting sequence, $\{\gY^{i}_t \}$ converges with respect to $C^{\infty}$ to some limit term $\gY^{\infty}_t$ on $B_{y}(R + m)$. We perform the above construction for all $m \in \mathbb{N}$ and take a diagonal subsequence of coefficient 
matrices which converge on any compact subset of $\mathbb{R}_{<0} \times \mathbb{R}^n$ to a connection coefficient matrix $\bm{\gG}_t$. The connection $\bm{\N}_t$ with $\bm{\gG}_t$ as its coefficient matrix is a solution to generalized Yang-Mills $k$-flow, in particular, given by
\begin{equation*}
\frac{\del \bm{\N}_t}{\del t} = (-1)^{k+1}D_{\bm{\N}_t}^* \bm{\lap}_{t}^{(k)} F_{\bm{\N}_t} + P_2^{2k-2} \brk{F_{\bm{\N}_t}}.
\end{equation*}
Furthermore, note that $\bm{\N}_t$ inherits all supremal bounds on the derivatives, as desired.
\end{proof}
\end{prop}
 
 \begin{rmk}
 Note that throughout the blowup procedure each connection satisfied a \textit{different} generalized Yang-Mills $k$-flow since the lower order terms within $\mho_k$ scale differently than the highest order term (cf. Remark \ref{rmk:mhoirrel}).
 \end{rmk}

\section{Long time existence results}\label{s:LTexist}

We now hone in our attention to specifically the Yang-Mills $k$-flow rather than its generalization. The explicit form allows for properties necessary to prove the main two parts of Theorem \ref{thm:YMksubcritdim}. We now demonstrate that for finite times the Yang-Mills $k$-flow yields control over the Yang-Mills energy of the curvature.

\begin{lemma}\label{lem:enrgydec}
Suppose that $\N_t \in \mathcal{A}_E \times \mathcal{I}$ is a solution to Yang-Mills $k$-flow. Then for all $T < \infty$ with $T \in \mathcal{I}$ we have that $\sup_{[0,T)}{||F_{\N_t}||_{L^2(M)}} < \infty$.

\begin{proof}
Let $\N_t$ be a solution to the Yang-Mills $k$-flow. Then differentiating the Yang-Mills $k$-energy with such $\N_t$ as the argument yields (referring to the proof of Proposition \ref{prop:ELYMk})
\begin{align*}
\tfrac{d}{d t} \lb  \mathcal{YM}_{k}(\N_t) \rb&= 2 \int_M{ \left\langle \tfrac{\del}{\del t} \brk{\N^{(k)}_t F_t }, \N^{(k)}_t F_t \right\rangle dV_g}\\
&= -2 \int_M{ \left\langle \Grad \mathcal{YM}_{k} (\N_t), \tfrac{\del \N_t}{\del t} \right\rangle dV_g}\\
&= -2 || \Grad \mathcal{YM}_{k}(\N_t) ||_{L^2}^2.
\end{align*}
This, unsurprisingly, indicates that the flow monotonically decreases the Yang-Mills $k$-energy. With the above computation in mind we will estimate $ || F_t ||_{L^2}^2$ on $[0,T]$. Differentiating with respect to the temporal parameter gives
\begin{align*}
\tfrac{d}{dt} \brk{\YM(\N_t)}
&= \int_{M}{ \left\langle \tfrac{\del F_t}{\del t} , F_t \right\rangle dV_g}\\
& = \int_{M}{ \left\langle D_t \dot{\gG}_t , F_t \right\rangle dV_g}\\
& = 2 \int_{M}{ \left\langle \dot{\gG}_t , D^*_tF_t \right\rangle dV_g}\\
& = -2 \int_{M}{  \left\langle \Grad \mathcal{YM}_{k} (\N_t) , D^*_t F_t \right\rangle dV_g}.
\end{align*}
Integrating both sides with respect to $t$ and manipulating with Young's Inequality followed by the weighted interpolation identity of Corollary \ref{cor:KS5.3interp}, for any $\epsilon >0$, we obtain
\begin{align*}
\YM(\N_T) - \YM(\N_0)
& = -2 \int_{0}^T{ \prs{\int_M{ \left\langle \Grad \mathcal{YM}_{k}(\N_t), D^*_tF_t   \right\rangle dV_g} }dt} \\
& \leq  C \int_{0}^T{\left( || \Grad \mathcal{YM}_{k}(\N_t) ||_{L^2}^2 + || D^*_tF_t ||_{L^2}^2 \right) dt} \\
&\leq C \int_0^T \tfrac{d}{dt}\brk{\YM_k(\N_t)} dt + C \int_0^T \brs{\brs{\N_t F_t}}_{L^2}^2 dt \\
&\leq C \prs{\YM_k(\N_0)-  \YM_k(\N_T)} +  \int_0^T \prs{ C_{\epsilon} \brs{\brs{\N_t^{(k)} F_t}}_{L^2}^2 + \epsilon \brs{\brs{F_t}}_{L^2}^2}dt \\
&\leq C T \prs{\YM_k(\N_0)} + \epsilon T \sup_{t\in[0,T]} \YM(\N_t).
\end{align*}
We thus have that
\begin{align*}
\sup_{t \in [0,T]}\prs{\mathcal{YM}\prs{\N_t}} \leq \frac{C T}{1-\epsilon} \prs{\mathcal{YM}_k(\N_0) + \YM(\N_0)}.
\end{align*}
Thus, choosing $\epsilon$ sufficiently small yields the desired result.
\end{proof}
\end{lemma}

\begin{thm}\label{thm:YMknoBUnleq2p} Suppose $\dim M <2p$ and $\N_t \in \mathcal{A}_E \times [0,T)$ a solution to generalized Yang-Mills $k$-flow and $\sup_{[0,T)}{ \brs{\brs{F_{\N_t}}}_{L^p(M)}} < \infty$. Then 
\[
\sup_{[0,T)}{\brs{\brs{F_{\N_t}}}_{L^{\infty}(M)}}<\infty.
\]

\begin{proof}
Set $\dim M = n< 2p$. We suppose to the contrary $\lim_{t \rightarrow T}{ || F_t ||_{L^{\infty}}} = \infty $ and construct a blowup limit $\{ \N^i_t \}$ with limit $\bm{\N}_t$ as described in Proposition \ref{prop:YMkblowup}. Since $|| F^i_t(x) ||_{L^\infty}=1< \infty$, by Fatou's Lemma and Proposition \ref{prop:scalelaw}, 
\begin{align*}
{\brs{\brs{ F_{\bm{\N}_t} }}_{L^p([0,1]^{\times n})}^p} &\leq \liminf_{i \to \infty}{||F^i_{t}||_{L^p([0,1]^{\times n})}^p} \\
&\leq  \lim_{i \to \infty}{  \la_i^{\frac{2p-n}{2k+2}} || F_t ||^p_{L^p(\mathbb{R}^n)}} .
\end{align*}
Since $\lim_{i \to \infty}{\la_i} = 0$ then  whenever $2p  > n$ the right hand side of the inequality converges to zero, which is a contradiction since the blowup limit is constructed for nontrivial curvatures. The result follows.
\end{proof}
\end{thm}

Utilizing this we may prove the complete long time existence of the Yang-Mills $k$-flow for subcritical dimensions.

\begin{proof}[Proof of Theorem \ref{thm:YMksubcritdim} (S)]
Set $\dim M = n$. By the Sobolev embedding theorem, we solve for $p$ such that $H_{2}^{k} \subset H_0^p$, namely one satisfying the formula.
\begin{equation*}
\frac{1}{p}= \frac{1}{2} - \frac{(k- 0)}{n}.
\end{equation*}
We additionally impose that $p > \tfrac{n}{2}$ to utilize Theorem \ref{thm:YMknoBUnleq2p} and solve to obtain that $2(k+2) > n$. In this case, then we have that, using the interpolation identities of Corollary \ref{cor:KS5.3interp}, where $\mathsf{S}_{k,p}$ is the Sobolev constant and $C$ is the constant induced by interpolation of these derivatives via Corollary \ref{cor:KS5.3interp},
\begin{align}
\begin{split}\label{eq:sobolevbd}
\brs{\brs{F_{\N}}}_{L^{p}(M)} &\leq \mathsf{S}_{k,p} \prs{ \sum_{j=0}^{k}\brs{\brs{\N^{(j)} F_{\N}}}_{L^2(M)}} \\
&= C \mathsf{S}_{k,p} \prs{ \sqrt{\YM_k(\N)} + \sqrt{\YM(\N)}}. 
\end{split}
\end{align}
Referring to Lemma \ref{lem:enrgydec}, since $\YM_k(\N_t)$ is decreasing along Yang-Mills $k$-flow and we have control over $\YM(\N_t)$ for any finite time, we conclude that the flow exhibits smooth long time existence.
\end{proof}

\begin{rmk} Note that this proof does not conclude that the flow exists at $t = \infty$, so it may be the case that the Yang-Mills $k$-flow admits singularities at infinite time.
\end{rmk}

We now state a theorem which generalizes the characterization of energy concentration of Yang-Mills flow in dimension $4$ introduced by Struwe \cite{Struwe}. First we characterize bubbling of $L^p$ norms in relation to the base manifold dimension. With this proposition we then may conclude the bubbling in critical dimensions.

\begin{prop}\label{prop:bubblincritLp}
Let $p \in \mathbb{N}$ and suppose $\dim M  = 2p$ and $\N_t$ is a solution to generalized Yang-Mills $k$-flow for $t \in [0,T)$ with $T$ maximal. Then there exists some $\epsilon > 0$ such that if $x \in M$ with the property that $\limsup_{t \nearrow T}{\brs{F_{\N_t}(x)}} = \infty$, then for all $r>0$, $\lim_{t \nearrow T}{||F_{\N_t}||_{L^{p}\prs{B_{x}(r)}}} \geq \epsilon$.

\begin{proof}
Choose a blowup sequence $\{ \N^i_t \}$ as described in Proposition \ref{prop:YMkblowup} with limit $\bm{\N}_t$. Then by construction $\brs{F_{\bm{\N}_0} (0)} =1$. By the derivative bounds on $\bm{\N}_t$ of Proposition \ref{prop:YMkblowup}, since $\brs{\bm{\N}_t F_{\bm{\N}_t}}$ is bounded, combined with the smoothness of $\bm{\N}_t$ over time, one has that that for $(y,t) \in B_0\prs{\delta} \times (-\delta, 0]$ we have
\[|F_{\bm{\N}_t}(y)| \geq \tfrac{1}{2}. \]
Observing this we have
\begin{align*} 
\lim_{t \nearrow 0}{|| F_{\bm{\N}_t} ||_{L^p(B_{0}(\delta))}^p} 
&= \lim_{t \nearrow 0}{ \int_{B_{0}(\delta)}{\brs{F_{\bm{\N}_t}}^p dV_g}} \\
&\geq \frac{\Vol \left[ B_{0}(\delta) \right]}{2^p}.
\end{align*}
Conversely, using the computations in Theorem \ref{thm:YMknoBUnleq2p} yields,
\begin{align*}
\brs{ \brs{ F_{\bm{\N}_t} }}^p_{L^p(B_{0}(\delta))} &= \int_{B_{\delta}(0)}{|F_{\bm{\N}_t}|^p dV_g}\\
&= \int_{B_{\delta}(0)}{\lim_{i \to \infty}|F_{\N_t^i}|^p dV_g}\\
&= \lim_{i \to \infty}{ \la_i^{\frac{2p-n}{2k+2}} \brs{\brs{ F_{\N_t} }}_{L^p \left(B_{x}(\delta \la_i^{1/(2k+2)}) \right)}^p}\\
&= \lim_{i \to \infty}{ \brs{\brs{ F_{\N_t} }}_{L^p \left(B_{x}\prs{\delta \la_i^{1/(2k+2)}} \right)}^p}.
\end{align*}
Since $\lim_{i \to \infty}{\la_i^{1/(2k+2)}} =0$ then for any $r >0$ and $t \in (T-\delta,T)$,
\[ \frac{\Vol \left[B_{0}(\delta) \right]}{2^p} \leq \lim_{i \to \infty}{ \brs{\brs{ F_{\N_t} }}^p_{L^p \left(B_{x}(\delta \la_i^{1/(2k+2)})\right) } } \leq  \brs{\brs{ F_{\N_t} }}^p_{L^p \left(B_{x}(r) \right) }. \]
Choose $\epsilon = \frac{\Vol \left[B_{0}(\delta) \right]}{2^p}$. Taking $\lim_{t \nearrow T}$ on both sides achieves the desired result.
\end{proof}
\end{prop}

We use this characterization of a discrete `quantum' of energy gathering up at a singularity to demonstrate that only a finite number of points can exhibit this blowup behavior at a singular time.


\begin{proof}[Proof of Theorem \ref{thm:YMksubcritdim} (C)]
Note that the lower bound on the amount of energy at a singularity location given in Proposition \ref{prop:bubblincritLp} is $\epsilon$ is independent of the point about which the blowup procedure occurred. We can again use the Sobolev embedding theorem as performed in the proof of Theorem \ref{thm:YMksubcritdim} (S) (cf. equation \eqref{eq:sobolevbd}) in the case where $n = 2(k+2)$ to conclude the estimate of \eqref{eq:sobolevbd}. Since $\mathcal{YM}_k(\N_t)$ and $\mathcal{YM}(\N_t)$ are both controlled for finite time along the flow, it follows that we again may bound $\lim_{t \nearrow T}\brs{\brs{F_{\N_t}}}_{L^p(M)}$ from above. Thus only a finite number of points may exhibit singularities at time $T$.
\end{proof}

\subsection{Extensions}\label{ss:applications}

Here we state the proof of our second main result, Theorem \ref{thm:regularizedflow}, and reflect on possible extensions of the flow.

\subsubsection{Regularized flow}\label{sss:regularizedflow}

As stated in the introduction, we study the Yang-Mills $(\rho,k)$-energy and corresponding gradient flow (cf. Definition \ref{eq:YMrhokenergydefn}). Utilizing the work of the previous sections combined with the presence of the Yang-Mills energy, we demonstrate subcritical long time existence and convergence.

\begin{proof}[Proof of Theorem \ref{thm:regularizedflow}]
The corresponding flow of this particular functional is given by the weighted sums of the negative gradient flows of the two participating functionals.
\begin{align}
\begin{split}\label{eq:YMrhokflow}
\frac{\del \N_t}{\del t}
&= \rho \left( (-1)^{k+1} D^*_{\N_t} \lap_t^{(k)} F_{\N_t} + P_2^{(2k-1)} \lb F_{\N_t} \rb \right)- D^*_{\N_t} F_{\N_t} \\
&= - \prs{\rho \Grad \YM_k(\N_t) + \YM(\N_t)}.
\end{split}
\end{align}
As in the work of \cite{HTY}, one would hope to apply a regularization argument on the Yang-Mills $(\rho,k)$-flow by sending $\rho \searrow 0$ to identify Yang-Mills connections. The advantage to using this flow over that of the Yang-Mills $\ga$-flow is that the Yang-Mills $(\rho,k)$-flow has long time existence and convergence in dimensions less than $2(k+2)$. This follows from simply temporally rescaling the gradient flow of \eqref{eq:YMrhokflow} to shift the dependence of $\rho$ on the highest order term to the others. Since these lower order terms satisfy the requirements of possible quantities represented by $\mho_k(\N)$ (cf. \eqref{eq:Omegak}), then we can simply apply the arguments of \S \ref{s:flowsubcrit} to obtain short time existence and uniqueness, necessary smoothing estimates, and construct blowup limits as desired. However, since this is the negative gradient flow of the weighted sum of the Yang-Mills $k$-energy and the Yang-Mills energy, we have that each individual energy is bounded over time above by a scaled multiple of $\rho \YM_k(\N_0) + \YM(\N_0)$ for all time. Therefore we obtain a subsequential limit at $t = \infty$ since a singularity cannot occur, and we conclude the result.
\end{proof}

As in the case of the Yang-Mills $\ga$-flow, one could pursue further results as in the works of \cite{HTY} and \cite{HS}, such as  verifying the Yang-Mills $k$-energy satisfies the Palais-Smale condition, which will guarantee the existence of minimizers, or proving an energy identity as in \cite{HS}.

\subsubsection{Yang-Mills $1$-flow versus bi-Yang Mills}\label{sss:BYMYM1}

In particular we now turn to the study of the Yang-Mills $1$-energy, given by
\begin{equation*}
\mathcal{YM}_1(\N) := \tfrac{1}{2}\int_{M} \brs{\N F_{\N}}^2 dV_g.
\end{equation*}
The Yang-Mills $1$-flow is closely tied to the bi-Yang-Mills flow (cf. \eqref{eq:BYM}) studied in \cite{IIU}. While the bi-Yang-Mills flow is arguably more `natural' to study (with reference to the gradient flow of Yang-Mills), the Yang-Mills $1$-flow admits long time existence. Roughly speaking, this is due to the fact that the Yang-Mills $1$-energy measures `all' of $\N F_{\N}$, while the bi-Yang-Mills energy only measures a portion (since $D^*$ is a trace of $\N$).

We explicitly demonstrate the relationship between the two energies in the following lemma.

\begin{lemma}[Comparison of Yang-Mills $1$-flow to bi-Yang-Mills]
For $\N \in \mathcal{A}_E$,
\begin{align*}
\mathcal{YM}_1(\N) &= 2 \mathcal{BYM}(\N) + \tfrac{1}{2} \int_M  g^{ij}g^{k \ell} \left( - \Rc_{\ell}^p F_{p i \ga}^{\gb}+ \Rc_{i}^p F_{p \ell \ga}^{\gb} - g^{qk}\Rm_{i \ell q}^p F_{k p \ga}^{\gb} \right) F_{jk \gb}^{\ga} dV_g\\
& \hsp  + \tfrac{1}{2} \int_M  g^{ij}g^{k \ell} \prs{ \left([F , F]^{\#} \right)_{i \ell \ga}^{\gb} - \left( [F, F]^{\#} \right)_{\ell i \ga}^{\gb} } F_{jk \gb}^{\ga} dV_g.
\end{align*}

\begin{proof}
We compute, transition to local coordinates and apply the Bochner formula (Proposition \ref{prop:bochner}),
\begin{align*}
\mathcal{YM}_1(\N)
&= \tfrac{1}{2}\int_M { \brs{\N F_{\N}}^2 dV_g}\\
&= \tfrac{1}{2} \int_M \ip{ F, \lap F } dV_g \\
&= - \tfrac{1}{2} \int_M  g^{ij}g^{k \ell} \prs{\lap F}_{i \ell \ga}^{\gb} F_{jk \gb}^{\ga} dV_g \\
&= \tfrac{1}{2} \int_M  g^{ij}g^{k \ell} \left( - \Rc_{\ell}^p F_{p i \ga}^{\gb}+ \Rc_{i}^p F_{p \ell \ga}^{\gb} - g^{qk}\Rm_{i \ell q}^p F_{k p \ga}^{\gb} \right) F_{jk \gb}^{\ga} dV_g\\
& \hsp  + \tfrac{1}{2} \int_M  g^{ij}g^{k \ell} \prs{ \left([F , F]^{\#} \right)_{i \ell \ga}^{\gb} - \left( [F, F]^{\#} \right)_{\ell i \ga}^{\gb} } F_{jk \gb}^{\ga} dV_g \\
& \hsp +\tfrac{1}{2} \int_M  g^{ij}g^{k \ell} \prs{\lap_D F}_{i \ell \ga}^{\gb} F_{jk \gb}^{\ga} dV_g.
\end{align*}
Manipulate the last integral yields
\begin{align*}
- \tfrac{1}{2} \int_M  \ip{ \lap_D F, F } dV_g
&= - \tfrac{1}{2} \int_M  \ip{ \prs{ D^* D + D D^*}F, F } dV_g \\
&= - \tfrac{1}{2} \int_M  \ip{ D D^*F, F } dV_g \\
&= \int_M  \brs{D^* F }^2 dV_g \\
&= 2 \mathcal{BYM}(\N).
\end{align*}
Combining terms, the result follows.
\end{proof}
\end{lemma}

The bi-Yang-Mills flow admits an isolation phenomena displayed in \cite{IIU}, which was in turn inspired by the work of \cite{BL}.
One would hope that an analogous isolation phenomena can be demonstrated for Yang-Mills $1$-flow. However, it appears that this trade off for properties such as guaranteed long time existence discussed above by measuring the full energy calls for more thought in demonstrating (if possible) an isolation phenomena.

\subsubsection{Conformally invariant functional}

By constructing a generalization of the Yang-Mills $k$-flow our analysis was extended to a broader range of flows. One wonders if, for each family of generalized Yang-Mills $k$-flows, there is a canonical representative for each $k$, and a corresponding canonical functional. One trait which could distinguish this canonical member is conformal invariance. Indeed, in the case of Yang-Mills flow, the conformal invariance in dimension $4$ is crucial to the work of Taubes in \cite{Taubes} regarding the process of constructing Yang-Mills instantons via gluing procedures.

\section{Appendix}\label{s:appendix}

We include various theorem statements and results essential to the arguments of the prior sections. This is comprised of analytic background (\S \ref{ss:analysis}), gauge transformations (\S \ref{ss:gaugetransformations}), and connection identities (\S \ref{ss:connmanips}).

\subsection{Analytic background}\label{ss:analysis}

We introduce some analytic results utilized within the main paper.

\begin{thm}[Arzela-Ascoli Theorem]\label{thm:arzelaascoli}
Let $M = \mathbb{R}^n$. Let $\ga \in (0,1]$ and $R \in \mathbb{R}_{>0}
$. Given some $K >0$, $\ga' < \ga$ and a sequence $\{ f_i \}$ such that
\[ || f_i ||_{C^{\varrho,\ga}(B_0(R))} \leq K, \]
there exists a subsequence $\{f_{i_j} \}$ and some $f_{\infty} \in C^{\varrho, \ga'}(B_0(R))$ such that $f_{i_j} \to f_{\infty}$ with respect to the $C^{\varrho, \ga'}(B_0(R))$ norm.
\end{thm}

\begin{thm}[Sobolev Imbedding Theorems, pp.35 of \cite{Aubin}]\label{thm:sobolevimbedding}
Set $M = \mathbb{R}^n$. Let $j, \ell \in \mathbb{N} \cup \{ 0 \}$ with $i<j $, and $p,q \in [1,\infty)$ with $1 \leq q <p$ such that \[ \frac{1}{p} = \frac{1}{q} - \frac{(j- i)}{n}.\] Then $H^{q}_{j} \subset H^{p}_{i}$ and the identity operator is continuous, and the following holds.
\begin{enumerate}
\item (The first Sobolev Imbedding Theorem) If there exists $\varrho \in \mathbb{N} \cup \{ 0 \}$ such that $\frac{(j-\varrho)}{n} > \frac{1}{q}$ then $H_j^q \subset C^{\varrho,0}$ and the identity operator is continuous.
\item (The second Sobolev Imbedding Theorem) If there exists  $\varrho \in \mathbb{N} \cup \{ 0 \}$ such that $\frac{(j-\varrho-\ga)}{n} \leq \frac{1}{q}$, then $H_j^q \subset C^{\varrho,\ga}$.
\end{enumerate}
\end{thm}

\begin{lemma}[Kato's inequality]\label{lem:kato} Suppose $L$ is some multiindex and $\gw \in (TM)^{\ten \brs{L}}$. Then whenever $|\gw| \neq 0$,
\begin{equation}\label{eq:kato}
| \N | \gw_L || \leq | \N \gw_L|.
\end{equation}
\end{lemma}

\subsubsection{Convexity estimates}\label{ss:convest}

We next extend two convexity estimates (Corollary 5.3 and 5.5) of \cite{KS} to be applied to $L^q$ norms of elements of $\Lambda^{p}(\End E)$ for $\ell \in \mathbb{N}$ (rather than elements of $C^{\infty}(M)$). The resulting corollary (cf. Corollary \ref{cor:KS5.3interp}) of the first result mentioned combined with the second result (cf. Lemma \ref{lem:kuwert5.5}) are key in the smoothing estimates of \S \ref{ss:smoothingest}. For a given $\eta \in \mathcal{B}$, recall the definition of $\jmath_{\eta}^{v}$ which is a constant bounding the $L^{\infty}(M)$ norms of $\eta$ (cf. Definition \ref{defn:bumpset}). We now state an analogue of Corollary 5.3 of \cite{KS}
\begin{lemma}\label{lem:kuwert5.3} 
Let $\N \in \mathcal{A}_E$ and $\eta \in \mathcal{B}$. For $2 \leq p < \infty$, $\ell \in \mathbb{N}$, $s \geq \ell p$, there exists $C_{\epsilon} = C_{\epsilon}\prs{\dim M, \rank E,p,  q, s, \ell , g,\jmath_{\eta}^{(1)}} \in \mathbb{R}_{>0}$ such that for $\phi \in C^{\infty}(M)$ we have,
\begin{equation*}
\prs{ \int_M \brs{\N^{(\ell)} \phi}^p \eta^s dV_g}^{1/p}
\leq \epsilon \prs{ \int_M \brs{\N^{(\ell+1)} \phi}^p \eta^{s+p} dV_g}^{1/p} + C_{\epsilon} \prs{ \int_{\eta > 0} \brs{ \phi }^p \eta^{s-\ell p } dV_g  }^{1/p}.
\end{equation*}

\begin{proof}
Note that this is simply Corollary 5.3 of \cite{KS} applied to $\brs{\N^{(\ell)} \phi}$.
\end{proof}
\end{lemma}
An immediate consequence of this lemma is an interpolation identity obtained via iterating the inequality is the following corollary.
\begin{cor}\label{cor:KS5.3interp} Let $\N \in \mathcal{A}_E$ and $\eta \in \mathcal{B}$. For $2 \leq p < \infty$, $\ell \in \mathbb{N}$, $s \geq \ell p $, there exists some $C_{\epsilon} = C_{\epsilon}\prs{\dim M, \rank E,p,  q, s, \ell , g, \jmath_{\eta}^{(1)}} \in \mathbb{R}_{>0}$ such that for $\phi \in C^{\infty}(M)$,
\begin{equation}\label{eq:KS5.3interp}
\brs{\brs{ \eta^{\tfrac{s}{p}} \N^{(\ell )} \phi}}_{L^p(M)}  \leq \epsilon \brs{ \brs{ \eta^{\tfrac{s + jp}{p}} \N^{(\ell +j)} \phi} }_{L^p(M)} + C_{\epsilon} \brs{ \brs{ \phi }}_{L^p(M), \eta > 0}.
\end{equation}
In particular for $p=2$ and some constant $K \geq 1$,
\begin{equation}\label{eq:KS5.3interpsp}
 K \brs{\brs{ \eta^{\tfrac{s}{2}} \N^{(\ell )} \phi}}_{L^2(M)}^2  \leq \epsilon \brs{ \brs{ \eta^{\tfrac{s + 2j}{2}} \N^{(\ell +j)} \phi} }_{L^2(M)}^2 + C_{\epsilon} K^2 \brs{ \brs{ \phi }}_{L^2(M), \eta > 0}^2.
\end{equation}
\begin{proof} This simply follows by induction. The base case is given by Lemma \ref{lem:kuwert5.3}. Now assume that for $j \in \mathbb{N}$ that \eqref{eq:KS5.3interp} holds. Without loss of generality, we consider the equality with $\epsilon$ replaced by $\sqrt{\epsilon}$. Manipulating the first term on the right we have, applying Lemma \ref{lem:kuwert5.3},
\begin{align*}
 \brs{ \brs{ \eta^{\tfrac{s + jp }{p}} \N^{(\ell +j)} \phi} }_{L^p} & \leq   \sqrt{\epsilon} \brs{ \brs{ \eta^{\tfrac{s + p(j+1)}{p}} \N^{(\ell +j + 1)} \phi} }_{L^p} + C_{ \sqrt{\epsilon} }\brs{ \brs{ \phi}}_{L^p,\eta > 0}.
\end{align*}
Inserting this into \eqref{eq:KS5.3interp} with $\epsilon$ replaced with $\sqrt{\epsilon}$, we conclude the result. Consequently we have the case for $j+1$, so inductively the result holds for all $\mathbb{N}$.

The second identity \eqref{eq:KS5.3interpsp} is essentially representing the direct application of this lemma in \S \ref{ss:smoothingest}, which is strictly in the setting where $p=2$ and the computations feature quantities with their $L^2$ norm \textit{squared}. Therefore we note one more manipulation where we square both sides of the inequality and apply H\"{o}lder's inequality:
\begin{align*}
a &\leq b + c \\
a^2 &\leq \prs{b + c}^2 = b^2 + 2bc + c^2 \\
& \leq b^2 + 2 \prs{\tfrac{b^2}{2}+  \tfrac{c^2}{2}} + c^2 \\
& \leq 2 \prs{b^2  + c^2}.
\end{align*}
Therefore if we are using the weighted versions of the inequality then they hold for the squared norms too. This is a minor manipulation but should be noted. Additionally, note that the shift of the `weight' $K$ featured in \eqref{eq:KS5.3interpsp} is merely a consequence of weighted H\"{o}lder's inequality being featured through each of these iterated computations. 
\end{proof}
\end{cor}

\begin{lemma}[Analogue of Corollary 5.5 of \cite{KS}]\label{lem:kuwert5.5}
Let $\N \in \mathcal{A}_E$ and $\eta \in \mathcal{B}$. Let $r,w,s \in \mathbb{N}$ and $\phi \in \Lambda^{p}(\End E)$ with $s \geq 2w$ and $0 \leq i_1 , \dots , i_r \leq w$, so that $\sum_{j=1}^r{i_j} = 2w$. Then there exists some $Q_{(w,r)} := Q \prs{ \dim M, \rank E, p, s, g,\jmath_{\eta}^{(1)}, w , r} \in \mathbb{R}_{>0}$ such that

\begin{equation*}
\int_{M}{\eta^s \left( \N^{(i_1)} \phi \ast ... \ast \N^{(i_r)} \phi  \right) dV_g} \leq Q_{(w,r)} || \phi ||_{L^{\infty}}^{r-2} \left( || \eta^{s/2} \N^{(w)} \phi ||_{L^2(M)}^2 + || \phi ||_{L^2(M), \eta>0}^2 \right).
\end{equation*}
\end{lemma}

\subsection{Connection identities}\label{ss:connmanips} 

We next provide some key identities regarding manipulations applied to the connections throughout various identities. We first state standard elementary manipulations and key formulas like the Bochner formula (cf. Proposition \ref{prop:bochner}) in the preliminary identities section (\S \ref{sss:prelimmanipid}), then state some basic scaling laws of connections and curvatures (\S \ref{sss:scalinglaw}) and then introduce manipulations to address the higher order differential operators which appear within the flow.

\subsubsection{Preliminary identities}\label{sss:prelimmanipid}
We first begin with the most basic manipulations.

\begin{prop}[Bochner formula]\label{prop:bochner}
Let $\N \in \mathcal{A}_E$ and $\gw \in \Lambda^p(\End E)$. Then the following equality holds.
\begin{equation}\label{eq:bochner}
\lap_D\omega=  -\lap \omega + \Rm \ast \omega + F \ast \gw.
\end{equation}
In particular, for $p=1$,
\begin{equation}\label{eq:bochner1}
(\lap_D \omega)_{i \ga}^{\gb} = - \N^k (\N_k \gw_{i \ga}^{\gb}) + \Rc_{i}^{p}\gw_{p \ga}^{\gb} +g^{jk} \left( F_{ij \ga}^{\gd} \gw_{k \gd}^{\gb} - F_{ij \gd}^{\gb} \gw_{k \ga}^{\gd} \right),
\end{equation}
given invariantly by $\lap_D\omega= -\lap \omega + \Rc(\gw^{\sharp}, \cdot ) + [F, \gw]^{\#}$. For $p=2$,
\begin{align}\begin{split}\label{eq:bochner2}
 (\lap_D \omega)_{i \ell \ga}^{\gb} &= -\N^k \N_k \gw_{i \ell \ga}^{\gb} +  \left( - \Rc_{\ell}^p \omega_{p i \ga}^{\gb}+ \Rc_{i}^p \omega_{p \ell \ga}^{\gb} - g^{qk}\Rm_{i \ell q}^p \omega_{k p \ga}^{\gb} \right) + \left([\gw , F]^{\#} \right)_{i \ell \ga}^{\gb} - \left( [\gw, F]^{\#} \right)_{\ell i \ga}^{\gb}.
\end{split}
\end{align}
\end{prop}

The Bochner formula can be seen as a consequence of the following technical lemma which demonstrates the terms that appear when commuting connections. We state this result in terms of explicit coordinates and demonstrate the proof for future use.

\begin{lemma}\label{lem:Ncomut}
For $K = (k_s)_{s=1}^{|K|}$ be a multiindex and $\gw \in S\prs{(T^*M)^{\ten \brs{K}} \ten \End E}$,
\begin{equation}\label{eq:Ncomut}
[\N_i, \N_j] \gw_{K \ga}^\gb=  \Rm_{ij k_{\ell}}^p(\gw_{K(\ell,p) \ga}^{\gb}) - F_{ij \ga}^{\gd} \gw_{K \gd}^{\gb} + F_{ij \gd}^{\gb} \gw_{K \ga}^{\gd}.
\end{equation}
Where `$K(\ell,p)$' means replacing $k_{\ell}$ with $p$ in the multiindex $K$.
\begin{proof}
Use of normal coordinates in the computation of the commutator of connections yields
\begin{align*}
[\N_i, \N_j] \gw_{K \ga}^{\gb}
&= \N_i(\N_j \gw_{K \ga}^{\gb} ) - \N_j(\N_i \gw_{K \ga}^{\gb} )\\
& = \del_i \left(\del_j \gw_{K \ga}^{\gb} - G_{j k_{\ell}}^{p}\gw_{K(\ell, p)\ga}^{\gb} - \gG_{j \ga}^{\gd}\gw_{K \delta}^{\gb} +\gG_{j \gd}^{\gb} \gw_{K \ga}^{\gd}\right) \\
& \hsp - \del_j \left(\del_i \gw_{K \ga}^{\gb} - G_{i k_{\ell}}^{p}\gw_{K(\ell, p)\ga}^{\gb} - \gG_{i\ga}^{\gd} \gw_{K \delta}^{\gb} +\gG_{i \gd}^{\gb} \gw_{K \ga}^{\gd}\right)\\
& = \left( \del_j G_{i k_{\ell}}^p - \del_i G_{j k_{\ell}}^p \right) \gw_{K(\ell,p) \ga}^{\gb} + \left( \del_j \gG_{i \ga}^{\gd}  - \del_i \gG_{j \ga}^{\gd} \right)\gw_{K \gd}^{\gb} \\
&\hsp + \left( - \del_j \gG_{i\gd}^{\gb } + \del_i \gG_{j \gd}^{\gb }\right) \gw_{K \ga}^{\gd}\\
&= \Rm_{ij k_{\ell}}^p(\gw_{K(\ell,p) \ga}^{\gb}) - F_{ij \ga}^{\gd} \gw_{K \gd}^{\gb} + F_{ij \gd}^{\gb}\gw_{K \ga}^{\gd}.
\end{align*}
The result follows.
\end{proof}
\end{lemma}

\subsubsection{Scaling laws}\label{sss:scalinglaw}

We first introduce key scaling properties of connections and corresponding quantities. This determines the critical dimension of the Yang-Mills $k$-flow, and is applied primarily in the blowup analysis (\S \ref{ss:blowupanalysis}) and flow long time existence results (\S \ref{s:LTexist}).

We first show how iterations of a scaled connection act on a similarly scaled $1$-form.

\begin{lemma}\label{lem:Nscal} Suppose $\N$ is a connection and let $x \in M$ such that in a coordinate chart containing $x$ the coefficient matrix of $\N$ is $\gG$. Let $\gw \in S(E)$ and set 
\[  \gG^{\la}:= \la \gG(\la x) \text{ and } \gw_{\la}(x) := \gw(\la x) .\]
Let $\N^{\la}$ denote the connection with coefficient matrix $\gG^{\la}$.
Then for all $\ell \in \mathbb{N}$,
\[
\N^{(\ell)}_{\la} \gw_{\la} =\la^{\ell} \N^{(\ell)} \gw. 
\]

\begin{proof}
We observe that in the case $\ell = 1$,
\begin{align*}
(\N^{\la})_i (\gw_{\la})^{\ga} &= \del_i (\gw_{\la})^{\ga} + (\gG^{\la})_{i \gd}^{\ga} (\gw_{\la})^{\gd}\\
&= \la \del_i \gw^{\ga} + \la \gG_{i \gd}^{\ga} \gw^{\gd}.
\end{align*}
Iterating this operation of $\N$ yields the desired result.
\end{proof}
\end{lemma}

\begin{rmk} Since $F_{\N} = D_{\N} \circ D_{\N}$, then it follows that for $\N^{\la}$ as defined above over the appropriate vector bundles,
\begin{equation*}
F_{\N^{\la}} = \la^2 F_{\N}.
\end{equation*} 
\end{rmk}

We now demonstrate how the above scaling effects the $L^p$ norm of the curvature. This is key in determining the critical dimension of our Yang-Mills $k$-energies as well as performing our blowup analyses. In preparation for these discussions, we scale the curvature more generally in the following argument.

\begin{prop}\label{prop:scalelaw} For $\la \in \mathbb{R}$, set
$F^{\la}(x) := \la^{q} F (\la^{r} x)$.
Then it follows that
\[
|| F^{\la} ||_{L^p(B_0(1))}^p = \la^{qp - nr} || F ||_{L^p(B_0(\la^r))}^p.
\]

\begin{proof}
Expressing the $L^p$ norm of $F_{\la}$ in terms of $F$, with the condensed notation $dx^{1 \dots n} = dx^1 \w \cdots \w dx^n$ yields
\begin{align*}
|| F^{\la} ||_{L^p}^p &= \int_{B_0(1)}{|F^{\la}|^p dx^{1 \cdots n}} \\
& = \int_{B_0(1)}{\la^{qp} |F( \la^{r}x)|^p dx^{1 \cdots n}}.
\end{align*}
We change variables by choosing $y^i = \la^{r} x^i$, giving that $dy^{1 \cdots n} = \la^{nr} dx^{1 \cdots n}$. Applying this to the above equality gives
\begin{align*}
|| F^{\la} ||_{L^p}^p &=  \la^{qp} \int_{B_0(\la^r)}{ | F(\la^r) |^p \frac{1}{\la^{nr}}dy^{1 \cdots n} } \\
&= \la^{qp - nr} || F ||_{L^p(B_0(\la^r))}^p.
\end{align*}
The result follows.
\end{proof}
\end{prop}

\subsubsection{Higher order identities}\label{sss:comutintid}
The following identities are key in manipulating higher order terms which appear the generalized Yang-Mills $k$-flow and are primarily results of recursive integration by parts and commuting of connections.

This next technical lemma prepares an expression to draw out iterations of the Laplacian after integrating by parts, which is performed in the following result (cf. Lemma \ref{lem:shiftkNs}).

\begin{lemma}\label{lem:kNstolaps}
Suppose $\N \in \mathcal{A}_E$ and $\xi$ in the domain of $\N$. Let $I= (i_v)_{v=1}^{\brs{I}}$ and $J = (j_v)_{v=1}^{\brs{J}}$ be two multiindices with $\brs{I}=\brs{J} = k$. Let $S$ and $Q$ be two multiindices consisting of entries corresponding to both bundle and base manifold. The following identity holds.
\begin{equation}\label{eq:kNstolaps}
\left( \N_{i_{k}} \cdots \N_{i_{1}} \N_{j_1} \cdots \N_{j_k} \xi_Q^S \right) = \left( \N_{i_{k}} \N_{j_k} \N_{i_{k-1}}\cdots \N_{i_{1}} \N_{j_1} \xi_Q^S \right) + \sum_{\ell = 1}^{2k-2} \sum_{w = 0}^{\ell} \left(\N^{(w)}(\Rm + F_{\N}) \ast \N^{(k-2-w)} \xi_Q^S \right).
\end{equation}

\begin{proof}
For notational simplicity, given a multiindex $I = (i_v)_{v=1}^{|I|}$, set $\N_{i_1 \cdots i_{|I|}}:= \N_{i_1}\cdots \N_{i_{|I|}}$. Iterating Lemma \ref{lem:Ncomut} as covariant derivatives are interchanged one obtains
\begin{align*}
\left( \N_{i_{\ell}} \cdots \N_{i_{1}} \N_{j_1} \cdots \N_{j_\ell} \xi_Q^S \right)
&=  \left( \N_{i_{\ell}} \N_{j_\ell} \N_{i_{\ell-1} \cdots i_{1}} \N_{j_1 \cdots j_{\ell-1}} \xi_Q^S \right)\\
& \hsp  + \sum_{v = 1}^{\ell-1}\left( \N_{i_\ell \cdots i_{v - 1} } [\N_{i_{v}}, \N_{j_{\ell}}] \N_{i_{v + 1} \cdots i_\ell} \N_{j_1 \cdots j_{\ell-1}} \right) \xi_Q^S \\
& \hsp  + \sum_{v = 1}^{\ell-1}\left(  \N_{i_1 \cdots i_\ell} \N_{j_1 \cdots j_{v -1}} [\N_{j_{v}}, \N_{j_\ell}] \N_{j_{v + 1} \cdots j_{\ell-1}} \right) \xi_Q^S \\
&= \left( \N_{i_{\ell}} \N_{j_{\ell}} \N_{i_{\ell-1} \cdots i_{1}} \N_{j_1 \cdots j_{\ell-1}} \xi_Q^S \right)\\
& \hsp + \sum_{v= 1}^{2\ell-2} \left( \N^{(v)} \left( (\Rm + F) \ast \N^{(2\ell-2-\ell)} \xi \right) \right)_{i_1 \cdots i_\ell j_1 \cdots j_\ell Q}^S  \\
&= \left( \N_{i_{\ell}} \N_{j_\ell} \N_{i_{\ell-1}} \cdots \N_{i_{1}} \N_{j_1} \xi_Q^S \right)\\
& \hsp + \sum_{v = 1}^{2\ell-2} \sum_{w = 0}^{v} \left(\N^{(w)}(\Rm + F) \ast \N^{(2\ell-2-w)} \xi \right)_{i_1 \cdots i_\ell j_1 \cdots j_\ell Q}^S.
\end{align*}
The result follows.
\end{proof}
\end{lemma}

\begin{lemma}\label{lem:shiftkNs}
Let $\N \in \mathcal{A}_E$, and $\zeta, \xi$ in the domain of $\N$. Then for $\ell \in \mathbb{N}$,
\begin{align}
\begin{split}\label{eq:shiftkNs}
\int_M{\left\langle \N^{(\ell)} \zeta , \N^{(\ell)} \xi \right\rangle dV_g}
&=\int_M{\langle \zeta,(-1)^k \lap^{(\ell)} \xi \rangle dV_g} + \left\langle \zeta, \sum_{v = 1}^{2\ell-2} \sum_{w = 0}^{v} \left(\N^{(w)} \lb \Rm + F\rb \ast \N^{(2\ell-2-w)} \xi  \right) \right\rangle.
\end{split}
\end{align}

\begin{proof}
Let $P$, $Q$, $S$ represent multiindices consisting of base manifold and bundle indices (roman and greek letters respectively), and let $\mathfrak{g}$ denote indices of quantities in the domain of $\N$ which are combinations of products of the bundle metric $h$ and the manifold metric $g$. Integrating by parts yields
\begin{align*}
\int_M{\left\langle \N^{(\ell)} \zeta , \N^{(\ell)} \xi \right\rangle dV_g}
&= \int_M{ \left(\prod_{v =0}^{\ell} g^{i_{v} j_{v}} \right) \mathfrak{g}^{PQ} \mathfrak{g}_{QS} \left(\N_{i_{1} \cdots i_{\ell}} \zeta_{P}^{R} \right) \left( \N_{j_1 \cdots j_{\ell}} \xi_{Q}^{S} \right) dV_g}\\
&=(-1)^{\ell}\int_M{ \left(\prod_{v=0}^{\ell} g^{i_{v} j_{v}} \right) \mathfrak{g}^{PQ} \mathfrak{g}_{RS}  \zeta_{P}^{R} \left( \N_{i_{\ell} \cdots i_{1}} \N_{j_1 \cdots j_{\ell}} \xi_{Q}^{S} \right) dV_g}.
\end{align*}
Using Lemma \ref{lem:kNstolaps} to manipulate the quantity yields
\begin{align*}
\int_M{\langle \N^{(\ell)} \zeta , \N^{(\ell)} \xi \rangle dV_g} &=\int_M{\langle \zeta,(-1)^{\ell} \lap^{(\ell)} \xi \rangle dV_g} + \left\langle \zeta, \sum_{v = 1}^{2\ell-2} \sum_{w = 0}^{v} \left(\N^{(w)} \lb \Rm + F \rb \ast \N^{(2 \ell-2-w)} \xi  \right) \right\rangle.
\end{align*}
The result follows.
\end{proof}
\end{lemma}

The next two lemmas are formal manipulations of commuting connections and laplacians in order to keep track of the lower order terms which appear when these are performed.

\begin{lemma}[Commuting multiple laplacians with the connection]\label{lem:lapkNcomut}
For $\N \in \mathcal{A}_E$, $\ell \in \mathbb{N}$ and $\gw \in \Lambda^p(\End E)$, 
\begin{equation}\label{eq:lapkNcomuteq}
\N \lap^{(\ell)} \gw = \lap^{(\ell)} \N \gw + \sum_{j=0}^{2\ell-1}{\N^{(j)}\brk{\Rm + F} \ast \N^{(2\ell-1-j)}\gw}.
\end{equation}

\begin{proof}
The proof follows by induction over $\ell \in \mathbb{N}$ satisfying \eqref{eq:lapkNcomuteq}. For the following computations we will excise the indices from $\gw$ for computational ease. For the base case, computing $\N \lap$ using Lemma \ref{lem:Ncomut} gives
\begin{align*}
\N_i \lap \gw &= g^{jk} \N_i \N_j \N_k \gw\\
&= g^{jk} \left( \N_j \N_i \N_k \gw + [\N_i, \N_j] \N_k \gw \right)\\
&= g^{jk} \left( \N_j  \N_k \N_i \gw + \N_j\left( [\N_i, \N_k] \gw \right)+ [\N_i, \N_j] \N_k \gw \right)\\
& = \lap \N_i \gw + \left( \N\brk{ (\Rm + F) \ast \gw } \right)_{i} + \left( (\Rm + F) \ast \N \gw \right)_i \\
&= \lap \N_i \gw + \left( \sum_{j=0}^1{\N^{(j)}\brk{\Rm + F} \ast \N^{(1-j)} \gw} \right)_i.
\end{align*}
The base case follows. Now let $\ell \in \mathbb{N}$ and suppose \eqref{eq:lapkNcomuteq} is satisfied by $\ell-1$. Applying this to the $\ell$ case yields
\begin{align}\label{eq:lapkNcomuteq1}
\begin{split}
\N \lap^{(\ell)} \gw &= \N \lap^{(\ell-1)} \left( \lap \gw \right) \\
&=  \lap^{(\ell-1)} \N \lap \gw + \sum_{j=0}^{2\ell-3}{\N^{(j)}\brk{\Rm + F} \ast \N^{(2\ell-3-j)} \lap \gw}.
\end{split}
\end{align}
Expanding and manipulating the left term gives
\begin{align*}
 \lap^{(\ell-1)} \N \lap \gw 
&=  \lap^{(\ell-1)} \lb \lap \N \gw + \sum_{j=0}^1{\N^{(j)}\brk{\Rm + F} \ast \N^{(1-j)} \gw} \rb \\
&=   \lap^{(\ell)} \N \gw + \lap^{(\ell-1)} \lb \sum_{j=0}^1{\N^{(j)}\brk{\Rm + F} \ast \N^{(1-j)} \gw} \rb \\
&=  \lap^{(\ell)} \N \gw +
\left( \sum_{j=0}^1{ \N^{(2\ell-2)}\left(\N^{(j)}\brk{\Rm + F} \ast \N^{(1-j)} \gw \right)} \right) \\
&=  \lap^{(\ell)} \N \gw +
\left( \sum_{j=0}^1{ \sum_{\tiny{(p+q = 2\ell-2)}}{\left(\N^{(p+j)}\brk{\Rm + F} \ast \N^{(1-j+q)} \gw \right)}} \right).
\end{align*}
Updating \eqref{eq:lapkNcomuteq1} we obtain
\[ \N \lap^{(\ell)} \gw  =  \lap^{(\ell)} \N \gw + \sum_{j=0}^{2\ell-1}{\N^{(j)}\lb \Rm + F \rb \ast \N^{(2\ell-1-j)}\gw}.\]
The result follows.
\end{proof}
\end{lemma}

\begin{cor}\label{cor:Nllapkcomut}
Let $\N \in \mathcal{A}_E$, take $v, w \in \mathbb{N}$ and 
$\gw \in \Lambda^p(\End E)$. Then
\begin{align*}
\N^{(v)} \lap^{(w)} \gw  &= \lap^{(w)} \N^{(v)} \gw + \sum_{b=0}^{v - 1} \sum_{j=0}^{2w-1}{ \left(\N^{(b+j)}\brk{\Rm + F} \ast \N^{(\ell - b +2w -2 -j)} \gw \right)}.
\end{align*}
\end{cor}

\subsection{Gauge transformations}\label{ss:gaugetransformations}

The main complications and interesting properties of the Yang-Mills flow stem from the interactions with the gauge group with the connections. We provide multiple identities which characterize these interactions and their consequences.

\begin{defn}[Gauge transformation]\label{defn:gaugetfm} A gauge transformation $\vs$ is a section of $\Aut E$. The group of gauge transformations is called the \textit{gauge group} of the metric bundle $E$, denoted by $\mathcal{G}_E$. The action of a gauge transformation $\varsigma$ on a connection is denoted by $\varsigma^* \N$ and given by
\begin{align*}
\varsigma^*
&: \mathcal{A}_{E} \rightarrow \mathcal{A}_E\\
& : \N \mapsto \varsigma^{-1} \circ \N \circ \varsigma,
\end{align*}
that is, for some $\mu \in S(E)$ we have $\varsigma^* \N = \varsigma^{-1} \N (\varsigma \mu)$. Furthermore, for $\phi \in S(\End E)$, define the action of $\varsigma$ on $\phi$ by
\[
(\varsigma^* \phi)^{\gb}_{\ga} := (\varsigma^{-1})^{\gb}_{\gz} \phi^{\gz}_{\tau} \varsigma^{\tau}_{\ga}.
\]
That is, a gauge transformation simply conjugates an endomorphism. Similarly for any $\gw \in \Lambda^{p}(\End E)$ for $p \in \mathbb{N}$ set
\[ \varsigma^*(\gw_{i \ga}^{\gb}) := (\varsigma^{-1})^{\gb}_{\gz} \gw^{i \gz}_{\tau} \varsigma^{\tau}_{\ga}.  \]
\end{defn}

The following results demonstrate the action of gauges on various objects besides those mentioned above. To obtain these, we first perform some explicit coordinate computations regarding these actions. The following lemma demonstrates what the connection coefficient matrix transforms into under the action of gauge transformation.

\begin{lemma}\label{lem:coordsN} Let $\vs \in S(\Aut E)$ and $\N \in \mathcal{A}_E$, and set $\nu = \nu^{\gb} \mu_{\gb} \in S(E)$. The coordinate expression of $\varsigma^*\N$ is
\begin{equation*}
\left( \varsigma^*\N \nu \right)_{i}^{\gb}  = \del_i \nu^{\gb} + ({\Gamma}_{\varsigma^*\N})_{i \gt}^{\gb} \nu^{\gt},
\end{equation*}
where
\begin{equation*}
(\sNG)_{i \theta}^{\gb} := (\varsigma^{-1})_{\gd}^{\gb}(\del_i \varsigma_{\gt}^{\gd}) + (\varsigma^{-1})_{\gd}^{\gb} \gG_{i \gamma}^{\gd} (\varsigma_{\gt}^{\gamma}).
\end{equation*}

\begin{rmk}
We emphasize that ${\Gamma}_{\varsigma^*\N}$ is \emph{not} equivalent to ${\varsigma^*\gG} = \vs^{-1} \gG \vs$.
\end{rmk}

\begin{proof}
Simply computing yields
\begin{align*}
\left( \varsigma^*\N \right)_{i} \nu^{\gb} &:= (\varsigma^{-1})_{\gd}^{\gb} \left( \N_i (\varsigma_{\gt}^{\gd} \nu^{\gt}) \right)\\
&= (\varsigma^{-1})_{\gd}^{\gb} \lb \del_i \varsigma_{\gt}^{\gd} \nu^{\gt}+ \varsigma_{\gt}^{\gd}(\del_i \nu^{\gt}) + \gG_{i \gamma}^{\gd}(\varsigma^{\gamma} \nu^{\gt}) \rb \\
&= \del_i \nu^{\beta} + (\varsigma^{-1})_{\gd}^{\gb}(\del_i \varsigma_{\gt}^{\gd}) \nu^{\gt} + (\varsigma^{-1})_{\gd}^{\gb} \gG_{i \gamma}^{\gd} (\varsigma_{\gt}^{\gamma} \nu^{\theta}).
\end{align*}
The result follows.
\end{proof}
\end{lemma}

We next compute the curvature of a gauge transformed connection in coordinates, and demonstrate that the gauge transformation acts on the curvature via conjugation. Note that this agrees with the declaration of the action of the gauge transformation on a $1$-form.

\begin{lemma}\label{lem:Fs} Suppose that $\N \in \mathcal{A}_E$ and $\varsigma \in S \prs{\Aut E}$. Then
\begin{equation*}
(F_{\vs^* \N})_{ij \ga}^{\gb} = (\vs^{-1})^{\gb}_{\gd} (F_{\N})^{\gd}_{ij \gt} \vs_{\ga}^{\gt}.
\end{equation*}

\begin{proof}
Observe that, by carefully matching terms yields
\begin{align*}
(F_{\vs^* \N})_{ij \ga}^{\gb} 
&= \del_i (\sNG)_{j \ga}^{\gb} -  \del_j (\sNG)_{i \ga}^{\gb} + (\sNG)_{i \gd}^{\gb}(\sNG)_{j \ga}^{\gd} - (\sNG)_{j \gd}^{\gb}(\sNG)_{i \ga}^{\gd}\\
&=  \del_i \left((\vs^{-1})^{\beta}_{\gd}(\del_j \vs_{\ga}^{\gd}) + (\vs^{-1})^{\gb}_{\gd} \gG_{j \gamma}^{\gd} \vs_{\ga}^{\gamma} \right)
- \del_j \left((\vs^{-1})^{\beta}_{\gd}(\del_i \vs_{\ga}^{\gd}) + (\vs^{-1})^{\gb}_{\gd} \gG_{i \gamma}^{\gd} \vs_{\ga}^{\gamma} \right)\\
&\hsp  + \left((\vs^{-1})^{\beta}_{\gt}(\del_i \vs_{\gd}^{\gt}) + (\vs^{-1})^{\gb}_{\gt} \gG_{i \gamma}^{\gt} \vs_{\gd}^{\gamma} \right)\left((\vs^{-1})^{\gd}_{\gz}(\del_j \vs_{\ga}^{\gz}) + (\vs^{-1})^{\gd}_{\gz} \gG_{j \eta}^{\gz} \vs_{\ga}^{\eta} \right)\\
& \hsp  - \left((\vs^{-1})^{\beta}_{\gt}(\del_j \vs_{\gd}^{\gt}) + (\vs^{-1})^{\gb}_{\gt} \gG_{j \gamma}^{\gt} \vs_{\gd}^{\gamma} \right)\left((\vs^{-1})^{\gd}_{\gz}(\del_i \vs_{\ga}^{\gz}) + (\vs^{-1})^{\gd}_{\gz} \gG_{i \eta}^{\gz} \vs_{\ga}^{\eta} \right) \\
&= ( \vs^{-1})^{\gb}_{\gt}\left( \del_i \gG_{j \mu}^{\gt} -  \del_j \gG_{i \mu}^{\gt} + \gG_{i \gd}^{\gb}\gG_{j \mu}^{\gd} - \gG_{j \gd}^{\gt} \gG_{i \mu}^{\gd} \right) \vs_{\ga}^{\mu}\\
&= (\vs^{-1})^{\gb}_{\gt} (F_{\N})_{ij \gd}^{\gt} \vs_{\ga}^{\gd}.
\end{align*}
The result follows.
\end{proof}
\end{lemma}

The following lemma demonstrates the action of a gauge on a commutation bracket.

\begin{lemma}\label{gaugebeamid}
For $\gw,\psi \in \Lambda^p(E)$ and $\vs \in S \prs{ \Aut E }$,
\begin{equation}
\vs^*[\gw,\psi]^{\#} = [ \vs^* \gw , \vs^* \psi ]^{\#}.
\end{equation}

\begin{proof}
Let $K$ and $L$ be multiindices of length $p$ and $q$ respectively, with  $K= (k_i)_{i=1}^{\brs{K}}$ and $L = (l_i)_{i=1}^{\brs{L}}$. Then computing yields
\begin{align*}
\vs^*\left([\gw,\psi] \right)_{KL \ga}^{\gb} &= \vs^*( \gw_{K \gd}^{\gb} \psi^{\gd}_{L \ga} -  \psi_{L \gd}^{\gb} \gw_{K \ga}^{\gd}) \\
&=\left( (\vs^{-1})^{\gb}_{\gz} \gw_{K \gd}^{\gz} \psi^{\gd}_{L \rho} \vs_{\ga}^{\rho}  - (\vs^{-1})^{\gb}_{\gz}\psi_{L \gd}^{\gz} \gw_{K \rho}^{\gd} \vs_{\ga}^{\rho} \right) \\
&= \left( (\vs^{-1})^{\gb}_{\gz} \gw_{K \gd}^{\gz} \vs^{\gd}_{\gt} (\vs^{-1})_{\tau}^{\gt}  \psi^{\tau}_{L \rho} \vs_{\ga}^{\rho}  - (\vs^{-1})^{\gb}_{\gz} \psi_{L \gd}^{\gz} \vs^{\gd}_{\gt} (\vs^{-1})_{\tau}^{\gt} \gw_{K \rho}^{\tau} \vs_{\ga}^{\rho} \right) \\
&=( (\vs^*\gw)_{K \gd}^{\gb} (\vs^* \psi)_{L \ga}^{\gd} - (\vs^* \psi)_{L \gd}^{\gb} (\vs^* \gw)_{K \ga}^{\gd} ) \\
&= \left([\vs^*\gw,\vs^*\psi]\right)_{KL \ga}^{\gb}.
\end{align*}
The result follows.
\end{proof}

\begin{rmk}
Note that since contraction occurs across base indices and the gauge transformation acts on the bundle, a consequence of this computation is that the gauge transformation also respects the pound bracket (defined in \eqref{eq:pndbrackdefn}), that is,
\begin{equation*}
\vs^*\prs{ \brk{\gw,\psi}^{\#}} = \brk{\vs^*\gw,\vs^*\psi}^{\#}. 
\end{equation*}
\end{rmk}
\end{lemma}

The next lemma demonstrates the action of gauge on a connection applied to endomorphism, and how the action distributes between the two objects.

\begin{lemma}\label{lem:gaugeactid1}
Let $\N \in \mathcal{A}_E$, $\phi \in S(\End E)$ and $\vs \in S \prs{ \Aut E }$. Then
\begin{equation*}
\vs^*(\N \phi ) = (\vs^*\N) ( \vs^* \phi ).
\end{equation*}

\begin{proof}
Expanding $\vs^* (\N \phi)$ yields
\begin{align*}
\vs^*( (\N)_i \phi_{\ga}^{\gb}) 
&= \left[ (\vs^{-1})_{\gt}^{\gb} (\del_i \phi_{\gamma}^{\gt}) \vs_{\ga}^{\gamma} \right]_{T_1} - \left[ (\vs^{-1})_{\gt}^{\gb} \phi_{\gd}^{\gt} \gG_{i \gamma}^{\gd} \vs_{\ga}^{\gamma} \right]_{T_2} + \left[ (\vs^{-1})_{\gt}^{\gb}\gG_{i \gd}^{\gt} \phi_{\gamma}^{\gd} \vs_{\ga}^{\gamma} \right]_{T_3}\\
&= T_1 + T_2 + T_3.
\end{align*}
Now observe that
\begin{align*}
(\vs^*\N)(\vs^{-1} \phi \vs) 
&= \del_i(\vs^{-1} \phi_{\gamma}^{\gt} \vs_{\ga}^{\gamma}) + ( (\vs^{-1} )_{\gt}^{\gb} (\del_i \vs_{\gd}^{\gt}) + (\vs^{-1})_{\gt}^{\gb} \gG_{i \rho}^{\gt} \vs_{\gd}^{\rho})(\vs^{-1} \phi \vs)_{\ga}^{\gd} \\
& \hsp - (\vs^{-1} \phi \vs)_{\gd}^{\gb} ( (\vs^{-1})_{\rho}^{\gd} \del_i \vs_{\ga}^{\rho} + (\vs^{-1})_{\rho}^{\gd} \gG_{i \gt}^{\rho} \vs^{\gt}_{\ga})\\
&= (\del_{i}\vs^{-1})_{\gt}^{\gb} \phi_{\gamma}^{\gt} \vs_{\ga}^{\gamma} + (\vs^{-1})_{\gt}^{\gb}(\del_i \phi_{\gamma} ^{\gt}) \vs_{\ga}^{\gamma} + (\vs^{-1})_{\gt}^{\gb} \phi_{\gamma}^{\gt} (\del_i \vs_{\ga}^{\gamma})\\
& \hsp +  (\vs^{-1})_{\gt}^{\gb} (\del_i \vs_{\gd}^{\gt}) (\vs^{-1})^{\gd}_{\rho} \phi^{\rho}_{\gz} s_{\ga}^{\gz} + (\vs^{-1})_{\gt}^{\gb} \gG_{i \rho}^{\gt} \phi_{\gz}^{\rho} \vs_{\ga}^{\gz} \\
& \hsp - (\vs^{-1})^{\gb}_{\gamma} \phi^{\gamma}_{\gd} \del_i \vs^{\gd}_{\ga} - (\vs^{-1})^{\gb}_{\gamma} \phi^{\gamma}_{\gd} \gG_{i \gt}^{\gd} \vs^{\gt}_{\ga}.\\
&=-\left[ (\vs^{-1})_{\gt}^{\gb} (\del_i \vs^{\gt}_{\tau})  (\vs^{-1})^{\tau}_{\gz} \phi_{\gamma}^{\gz} \vs_{\ga}^{\gamma} \right]_{T_4} + \left[ (\vs^{-1})_{\gt}^{\gb}(\del_i \phi_{\gamma} ^{\gt}) \vs_{\ga}^{\gamma} \right]_{T_1} \\
& \hsp+ \left[ (\vs^{-1})_{\gt}^{\gb} \phi_{\gamma}^{\gt} (\del_i \vs_{\ga}^{\gamma}) \right]_{T_5} +  \left[ (\vs^{-1})_{\gt}^{\gb} (\del_i \vs_{\gd}^{\gt}) (\vs^{-1})^{\gd}_{\rho} \phi^{\rho}_{\gz} \vs_{\ga}^{\gz} \right]_{T_4} \\
& \hsp+ \left[ (\vs^{-1})_{\gt}^{\gb} \gG_{i \rho}^{\gt} \phi_{\gz}^{\rho} \vs_{\ga}^{\gz} \right]_{T_3} - \left[ (\vs^{-1})^{\gb}_{\gamma} \phi^{\gamma}_{\gd} \del_i \vs^{\gd}_{\ga} \right]_{T_5} - \left[ (\vs^{-1})^{\gb}_{\gamma} \phi^{\gamma}_{\gd} \gG_{i \gt}^{\gd} \vs^{\gt}_{\ga} \right]_{T_2} \\
& = T_1 + T_2 + T_3.
\end{align*}
The equality holds and the result follows.
\end{proof}
\end{lemma}

A key property of gauges with respect to the Yang-Mills $k$-energy, which determines the nonellipticity of the flow (cf. Proposition \ref{prop:Phikwkellptc}) is demonstrated in the following lemma. Namely, that the Yang-Mills $k$-energy is invariant under gauge transformation. 

\begin{cor}\label{lem:YMkgaugeinv} For $\vs \in S \prs{ \Aut{E} }$ and $\N \in \mathcal{A}_E$,
\[
\mathcal{YM}_{k}(\N) = \mathcal{YM}_{k}(\varsigma^* \N).
\]
\begin{proof}
This is a consequence of  Lemma \ref{lem:Fs} and the definition of the action of gauge on a connection and on a $2$-form (cf. Definition \ref{defn:gaugetfm})

\end{proof}
\end{cor}

\begin{lemma}\label{lem:slapswap}
Let $L := (i_1 , j_1 ,\cdots i_k, j_k)$, $\vs \in S(\Aut E)$, and $\zeta$ some element of a tensor product of $T^*M$ and $E$ and their corresponding duals. Then
\begin{equation}
\lap^{(k)} \left[\vs_{\ga}^{\gb} \zeta_R^Q \right] = \lap^{(k)}[\vs]_{\ga}^{\gb} \zeta^{Q}_R + \left( \prod_{v = 0}^{k} g^{i_{v} j_{v}}\right) \prs{ \sum_{r = 1}^{k-1}\sum_{\mathsf{P} \in \mathcal{P}_r(L)}{\left( \N_{\mathsf{P}} \vs_{\ga}^{\gb} \right)\left( \N_{\mathsf{P}^c}\zeta_{R}^{Q}\right) }} + \vs_{\ga}^{\gb} \lap^{(k)} \left[ \zeta \right]_R^Q,
\end{equation}
where the quantity $\mathcal{P}_r(L)$ is defined in Definition \ref{defn:partitionstring}

\begin{proof}
This is simply an application of the Leibniz rule and being aware of the distribution of connection pairings (coming from each Laplacian).
\end{proof}
\end{lemma}

In the following lemma we investigate the action of this particular connection with a one-parameter family of gauge transformation and is essential in the following result.

\begin{lemma}\label{lem:coordsDs*Gamma} Let $\N \in \mathcal{A}_E$ and $\vs_t \in S\prs{\Aut E} \times \mathcal{I}$. Then 
\begin{equation}\label{eq:coordsDs*Gammaeq}
\left(\vs_t^* \N \brk{ \vs^{-1}_t \dot{\vs}_t} \right)_{\ga}^{\gb} = (\vs^{-1}_t)_{\gd}^{\gb}(\del_i \dot{\vs}_t)_{\ga}^{\gd} + (\vs^{-1}_t)_{\gt}^{\gb}(\Gamma)_{i \gamma}^{\gt}(\dot{\vs}_t)_{\ga}^{\gamma} - ( \Gamma_{\vs^*_t \N} )_{i \ga}^{\gd} \left(\vs^{-1}_t \dot{\vs}_t \right)_{\gd}^{\gb}.
\end{equation}

\begin{proof}
Note that despite the assumed time dependence of $\vs_t$ the notational dependency will be omitted. Simply computing yields
\begin{align*}
(\vs^*\N)(\vs^{-1} \dot{\vs}) &= \del_i \left( (\vs^{-1})_{\gd}^{\gb}(\dot{\vs})_{\ga}^{\gd} \right)+ \left(\sNG \right)_{i \gd}^{\gb} \left(\vs^{-1} \dot{\vs} \right)^{\gd}_{\ga} - (\sNG )_{i \ga}^{\gd} \left(\vs^{-1} \dot{\vs} \right)_{\gd}^{\gb} \\
&= -(\vs^{-1})_{\rho}^{\gb}(\del_i \vs)^{\rho}_{\gt}(\vs^{-1})_{\gd}^{\gt}(\dot{\vs})_{\ga}^{\gd} + (\vs^{-1})_{\gd}^{\gb}(\del_i \dot{\vs})_{\ga}^{\gd} + (\sNG)_{i \gd}^{\gb} (\vs^{-1} \dot{\vs})_{\ga}^{\gd} \\ & \hsp - (\sNG)_{i \ga}^{\gd} \left(\vs^{-1} \dot{\vs} \right)_{\gd}^{\gb}.\\
&= -(\vs^{-1})_{\rho}^{\gb}(\del_i \vs)^{\rho}_{\gt}(\vs^{-1})_{\gd}^{\gt}(\dot{\vs})_{\ga}^{\gd} + (\vs^{-1})_{\gd}^{\gb}(\del_i \dot{\vs})_{\ga}^{\gd} + (\vs^{-1})_{\rho}^{\gb}(\del_i \vs)^{\rho}_{\gt}(\vs^{-1})_{\gd}^{\gt}(\dot{\vs})_{\ga}^{\gd}\\ & \hsp + (\vs^{-1})_{\gt}^{\gb}(\Gamma)_{i \gamma}^{\gt}(\dot{\vs})_{\ga}^{\gamma} - (\vs^* \Gamma )_{i \ga}^{\gd} \left(\vs^{-1} \dot{\vs} \right)_{\gd}^{\gb}.\\
&= (\vs^{-1})_{\gd}^{\gb}(\del_i \dot{\vs})_{\ga}^{\gd} + (\vs^{-1})_{\gt}^{\gb} \Gamma_{i \gamma}^{\gt}(\dot{\vs})_{\ga}^{\gamma} - (\sNG)_{i \ga}^{\gd} \left(\vs^{-1} \dot{\vs} \right)_{\gd}^{\gb}.
\end{align*}
The result follows.
\end{proof}
\end{lemma}

In the following lemma both the connection and the gauge transformation are one-parameter families, though the gauge transformation does not necessarily determine how the family of connections varies, as was the case in the prior lemma.

\begin{lemma}
Let $\vs_t \in S \prs{\Aut E} \times \mathcal{I}$ and $\N_t \in \mathcal{A}_E \times \mathcal{I}$. The linearized gauge action of $\vs_t$ on $\N_t$ is given in coordinates by
\begin{align}
\begin{split}\label{eq:vargauge}
\left( \tfrac{\del}{\del t} \left[ \vs_t^* \N_t \right] \right)_{k\ga}^{\gb}  &= (\vs_t^*\N_t)(\vs^{-1}_t \dot{\vs}_t)_{k \ga}^{\gb} + (\vs^{-1}_t)_{\gd}^{\gb} \prs{\dot{\Gamma_t}}_{k \gt}^{\gd} (\vs_t)_{\ga}^{\gt}.
\end{split}
\end{align}

\begin{proof}
As in the previous lemma, note that despite the assumed time dependence of $\vs_t$ and $\N_t$ the notational dependency will be omitted. Differentiating $\vs^* \gG$ with respect to $t$ gives
\begin{align*}
\prs{\tfrac{\del}{\del t}\brk{\vs^* \N}}_{k \ga}^{\gb} 
&= \del_t ( (\vs^{-1})_{\gd}^{\gb} \Gamma_{k \gt}^{\gd} \vs_{\ga}^{\gt}) +  \del_t( (\vs^{-1})_{\gd}^{\gb} \del_i s_{\ga}^{\gd})\\
&=  \del_t(\vs^{-1})_{\gd}^{\gb} \Gamma_{k \gt}^{\gd} \vs_{\ga}^{\gt} + (\vs^{-1})_{\gd}^{\gb} ( \dot{\Gamma}_{k \gt}^{\gd}) \vs_{\ga}^{\gt}  + (\vs^{-1})_{\gd}^{\gb} \Gamma_{k \gt}^{\gd} (\dot{\vs}_{\ga}^{\gt} )\\
& \hsp + \del_t(\vs^{-1})_{\gd}^{\gb} (\del_i \vs_{\ga}^{\gd}) +  (\vs^{-1})_{\gd}^{\gb} (\del_i \dot{\vs}_{\ga}^{\gd}) \\
&= -\left( (\vs^{-1})_{\gamma}^{\gb} ( \dot{\vs}_{\gz}^{\gamma}) (\vs^{-1})_{\gd}^{\gz} \right) \Gamma_{k \gt}^{\gd} \vs_{\ga}^{\gt} + (\vs^{-1})_{\gd}^{\gb}  \dot{\Gamma}_{k \gt}^{\gd} \vs_{\ga}^{\gt}  + (\vs^{-1})_{\gd}^{\gb}  \Gamma_{k \gt}^{\gd} \dot{\vs}_{\ga}^{\gt}\\
& \hsp  -\left( (\vs^{-1})_{\gamma}^{\gb} (\dot{\vs}_{\gz}^{\gamma}) (\vs^{-1})_{\gd}^{\gz} \right) (\del_i \vs_{\ga}^{\gd}) +  (\vs^{-1})_{\gd}^{\gb} (\del_i \dot{\vs}_{\ga}^{\gd}).
\end{align*}
Applying the computations of Lemma \ref{lem:coordsDs*Gamma} and then Lemma \ref{lem:coordsN},
\begin{align*}  \left( \tfrac{\del}{\del t} \left[ \vs^* \N \right] \right)_{k\ga}^{\gb} 
&= -\left( (\vs^{-1})_{\gamma}^{\gb} ( \dot{\vs}_{\gz}^{\gamma}) (\vs^{-1})_{\gd}^{\gz} \right) \Gamma_{k \gt}^{\gd} \vs_{\ga}^{\gt} + (\vs^{-1})_{\gd}^{\gb}  \dot{\Gamma}_{k \gt}^{\gd} \vs_{\ga}^{\gt}  + (\vs^{-1})_{\gd}^{\gb}  \Gamma_{k \gt}^{\gd} \dot{\vs}_{\ga}^{\gt}\\
& \hsp  -\left( (\vs^{-1})_{\gamma}^{\gb} (\dot{\vs}_{\gz}^{\gamma}) (\vs^{-1})_{\gd}^{\gz} \right) (\del_i \vs_{\ga}^{\gd}) +  (\vs^{-1})_{\gd}^{\gb} (\del_i \dot{\vs}_{\ga}^{\gd}) \\
&= \left[ (\vs^{-1})_{\gd}^{\gb} (\del_k \dot{\vs}_{\ga}^{\gd}) + (\vs^{-1})_{\gd}^{\gb}  \Gamma_{k \gt}^{\gd} \dot{\vs}_{\ga}^{\gt} - (\vs^* \gG)_{k\ga}^{\gd} (\vs^{-1} \dot{\vs})_{\gd}^{\gb} \right] \\
& \hsp + \left[  -\left( (\vs^{-1})_{\gamma}^{\gb} ( \dot{\vs}_{\gz}^{\gamma}) (\vs^{-1})_{\gd}^{\gz} \right) \Gamma_{k \gt}^{\gd} \vs_{\ga}^{\gt} -\left( (\vs^{-1})_{\gamma}^{\gb} (\dot{\vs}_{\gz}^{\gamma}) (\vs^{-1})_{\gd}^{\gz} \right) (\del_k \vs_{\ga}^{\gd}) \right] \\
& \hsp + (\vs^* \gG)_{k\ga}^{\gd} (\vs^{-1} \dot{\vs})_{\gd}^{\gb} + (\vs^{-1})_{\gd}^{\gb} (\dot{\Gamma}_{k \gt}^{\gd}) \vs_{\ga}^{\gt}\\
&= (\vs^*\N)(\vs^{-1} \dot{\vs})_{k \ga}^{\gb} + (\vs^{-1})_{\gd}^{\gb} (\dot{\Gamma}_{k \gt}^{\gd}) \vs_{\ga}^{\gt}.
\end{align*}
The result follows.
\end{proof}
\end{lemma}

\bibliographystyle{hamsplain}

\end{document}